\documentclass[a4paper]{article}
\usepackage[utf8]{inputenc}

\usepackage[margin=2.54cm]{geometry}

\usepackage{setspace}
\usepackage{datetime}
\usepackage{fancyhdr}
\usepackage{tikz}
\usepackage{hyperref}

\usepackage{ulem}
\usepackage{hyperref}
\usepackage{authblk}
\usepackage{amsthm}
\usepackage{amsmath}
\usepackage{amssymb}
\usepackage{amsfonts}
\usepackage{mathtools}
\usepackage{algorithm}%http://ctan.org/pkg/algorithms
\usepackage{braket}
\usepackage{algpseudocode}
\usepackage{graphicx}
\usepackage{xcolor}
\usepackage{xspace}
\usepackage[graphicx]{realboxes}
\usepackage{hypbmsec}

\newtheorem{lemma}{Lemma}
\newtheorem{theorem}{Theorem}
\newtheorem{definition}{Definition}
\newtheorem{problem}{Problem}
\newtheorem{proposition}{Proposition}
\newtheorem{corollary}{Corollary}

\newtheorem*{thm_necklace_counting}{Theorem \ref{thm:necklace_counting}}
\newtheorem*{prop_trans_symb}{Proposition \ref{prop:translational_symbolisation}}
\newtheorem*{theorem_L_to_A}{Theorem \ref{thm:lyndon_to_atranslational}}
\newtheorem*{theorem21}{Theorem 2.1}
\newtheorem*{theorem4}{Theorem \ref{chp8:thm:generation_complexity}}
\newtheorem*{theorem2}{Theorem \ref{thm:ranking_complexity}}
\newtheorem*{thm_fixed_content_ranking}{Theorem \ref{thm:fixed_content_ranking}}
\newtheorem*{theorem3}{Theorem \ref{chp8:thm:unranking}}
\newtheorem*{col_fixed_content_unranking}{Corollary \ref{chp8:col:fixed_content_unranking}}
\newtheorem*{thm_alg_3}{Theorem \ref{thm:alg_3}}
\newtheorem*{thm_lynodn_counting}{Theorem \ref{thm:Lyndon_counting}}
\newtheorem*{thm_db1}{Theorem \ref{thm:de_bruijn_1}}

\newcommand{\RankingComplexity}{\ensuremath{O\left(N^5\right)}~}

\newcommand{\word}[1]{\ensuremath{\bar{#1}}}

\newcommand{\necklace}[1]{\ensuremath{\tilde{\mathbf{#1}}}}

\newcommand{\Angle}[1]{\ensuremath{\langle #1 \rangle}}

\newcommand{\vectorise}[1]{\mathbf{\overline{#1}}}

\DeclareMathOperator{\lyn}{lyn}
\DeclareMathOperator{\lcm}{lcm}

\DeclareMathOperator{\period}{Period}
\DeclareMathOperator{\indexFunc}{index}
\DeclareMathOperator{\Parikh}{P}

\DeclarePairedDelimiter{\floor}{\lfloor}{\rfloor}

\newcount\Comments  % 0 suppresses notes to selves in text
\definecolor{darkgreen}{rgb}{0,0.6,0}
\newcommand{\kibitz}[2]{\ifnum\Comments=1{\color{#1}{#2}}\fi}

\definecolor{darkred}{rgb}{0.6,0,0}

\title{Combinatorial Algorithms for Multidimensional Necklaces}

\author{Duncan Adamson\thanks{Department of Computer Science, Reykjavik University, Iceland.
        Email: \texttt{duncana@ru.is}} \and Argyrios Deligkas\thanks{Royal Holloway University of London, UK. 
		Email: \texttt{argyrios.deligkas@rhul.ac.uk}} \and Vladimir V. Gusev\thanks{Materials Innovation Factory, Department of Chemistry,  University of Liverpool, UK.
		Email: \texttt{Vladimir.Gusev@liverpool.ac.uk}} \and Igor Potapov\thanks{Department of Computer Science, University of Liverpool, UK.
		Email: \texttt{potapov@liverpool.ac.uk}}
		}

\date{ }

\begin{document}

\maketitle
\begin{abstract}
\noindent
A \emph{necklace} is an equivalence class of words of length $n$ over an alphabet under the cyclic shift (rotation) operation.
As a classical object, there have been many algorithmic results for key operations on necklaces, including counting, generating, ranking, and unranking.
This paper generalises the concept of necklaces to the multidimensional setting. 
We define \emph{multidimensional necklaces} as an equivalence classes over multidimensional words under the multidimensional cyclic shift operation.
Alongside this definition, we generalise several problems from the one dimensional setting to the multidimensional setting for multidimensional necklaces with size $(n_1,n_2,\hdots,n_d)$ over an alphabet of size $q$ including: providing closed form equations for counting the number of necklaces; an $O(n_1 \cdot n_2 \cdot \hdots \cdot n_d)$ time algorithm for transforming some necklace $\necklace{w}$ to the next necklace in the ordering; an $O((n_1\cdot n_2 \cdot \hdots \cdot n_d)^5)$ time algorithm to rank necklaces (determine the number of necklaces smaller than $\necklace{w}$ in the set of necklaces); an $O((n_1\cdot n_2 \cdot \hdots \cdot n_d)^{6(d + 1)} \cdot \log^d(q))$ time algorithm to unrank multidimensional necklace (determine the $i^{th}$ necklace in the set of necklaces).
Our results on counting, ranking, and unranking are further extended to the \emph{fixed content} setting, where every necklace has the same Parikh vector, in other words every necklace shares the same number of occurrences of each symbol.
Finally, we study the $k$-centre problem for necklaces both in the single and multidimensional settings.
We provide strong approximation algorithms for solving this problem in both the one dimensional and multidimensional settings.
\end{abstract}

\newpage

\section{Introduction}
A \emph{necklace} is an equivalence class of words of a fixed length over a finite alphabet under the cyclic shift (rotation) operation. More specifically an equivalence class of $n$-character strings/words over an alphabet of size $q$ is known as $q$-ary necklace of length $n$ and the class of aperiodic necklaces is known as {\sl Lyndon words}. 
In order to represent a necklace (or a Lyndon word) as a single word  a string of characters which
is lexicographically smallest out of all of its possible rotations is used.
Lyndon words and necklaces 
are fundamental  combinatorial objects
arising 
%have numerous applications
in the field of text algorithms \cite{Kociumaka2014}, 
in the construction of single-track Gray codes
%circular codes
\cite{Ruskey2012,Schwartz1999}, analysis of circular DNA and splicing systems \cite{CircularSplicing},
in the enumeration of irreducible polynomials over finite fields \cite{Kopparty2016}, and in the theory of free Lie algebras \cite{alev2008structure}. 
%The discovery of efficient algorithms for %generation, ranking, unranking, counting, sampling is a essential part  in combinatorics on words. 

Many computational problems have been formulated and studied for fixed length combinatorial necklaces over a finite alphabet including counting the number of necklaces, generating, 
%all such 
ranking (computing a rank according to a previously fixed order), and unranking (generation of the $i$-th combinatorial object) necklaces.
% , counting the number of all objects, etc.Hi all, 
Graham, Knuth and Patashnik provide equations for counting both the number of necklaces and  Lyndon words (aperiodic necklaces) in \cite{Graham1994}.
%Building on the results of counting the number of necklaces and Lyndon words has been a set of algorithms to generate each set for a given length and alphabet in lexicographic order. 
The first algorithms for generating necklaces were designed by Fredricksen and Kessler \cite{fredricksen1986algorithm}, and Fredricksen and Maiorana \cite{Fredricksen1978}, which were later proven to run in constant amortised time (CAT) by Ruskey, Savage and Wang \cite{Ruskey1992}.
Cattell, Ruskey, Sawada, Serra, and Miers provided a further CAT algorithm for the generation of necklaces and Lyndon words \cite{Cattell2000}.

The existence of polynomial time ranking and unranking algorithms for necklaces (cyclic words) remained an open problem for many years and has been only recently solved.
%discovered in the last few years.
 %
The first class of cyclic words to be ranked were \emph{Lyndon words} by Kociumaka, Radoszewski, and Rytter \cite{Kociumaka2014} who provided an $O(n^3)$ time algorithm, where $n$ is the length of the word.
An algorithm for ranking necklaces was given by Kopparty, Kumar, and Saks \cite{Kopparty2016}, without tight bounds on the complexity.
An $O(n^2)$ time algorithm for ranking necklaces was provided by Sawada and Williams \cite{Sawada2017}.
More recently,
the open problem of ranking $q$-ary bracelets of length $n$ (the equivalence class of words under the combination of the rotation
and reflections), posed by Sawada and Williams, was solved
in $O(q^2 \cdot n^4)$ time in \cite{Adamson2021}.

Algorithms for multidimensional 
combinatorial 
necklaces has remained a largely unexplored area in combinatorics on words \cite{lothaire_1997,Siromoney}.
A multidimensional necklace is an  equivalence  class  of  multidimensional words under \emph{translational symmetry}, which is the natural generalisation of the shift operation in 1D, see Figure~\ref{fig:translation_example}. 
This work aims to fill the gap by developing a set of efficient combinatorial algorithms for multidimensional necklaces. 

\begin{figure}[h]
    \centering
    \includegraphics[scale=1]{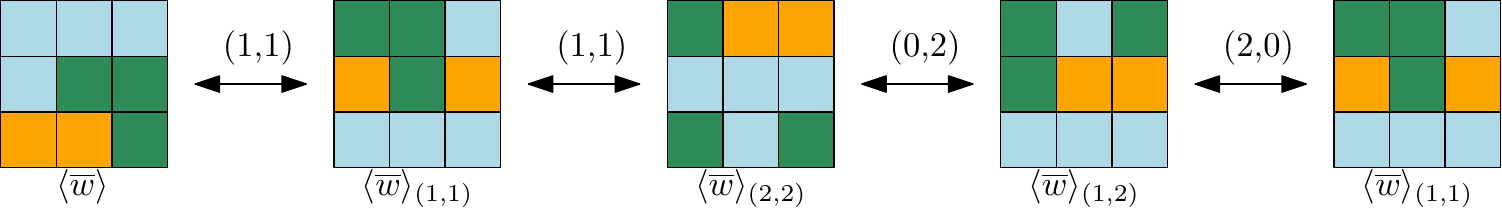}
    \caption{An illustration of translational symmetry for a $3 \times 3$ word.
    Note that all four words can be reached from one another through
     two-dimensional translation
%    cyclic shifts along two dimensions, 
    denoted $(g_1,g_2)$.
    % \duncan{CHANGE (make smaller, fix arrows)}
    }
    \label{fig:translation_example}
\end{figure}

\noindent
Two-dimensional necklaces have been recently studied with the motivation of counting the number of toroidal codes in  \cite{Anselmo2019} and can be used in the construction of 
two dimensional Gray codes \cite{Bae2000}.
However the most direct application of multidimensional necklaces up to dimension three is the combinatorial representation of crystal structures. 
%
% Many practical materials are crystalline, which means that their constituents such as ions or molecules are in a highly ordered arrangement.
% Crystals are often seen as a periodic and
% The choice
In computational chemistry, crystals are represented by periodic motives (or coloured tessellations) known as “unit cells”.
Informally, translational symmetry can be thought of as the equivalence of two crystals under translation in space.
This intuitively make sense in the context of real structures, where two different ``snapshots'' of a unit cell both represent the same periodic  and  infinite global structure.
\begin{figure}[h]
    \centering
    \begin{tabular}{c|c|c}
    \includegraphics[width=0.14\linewidth] {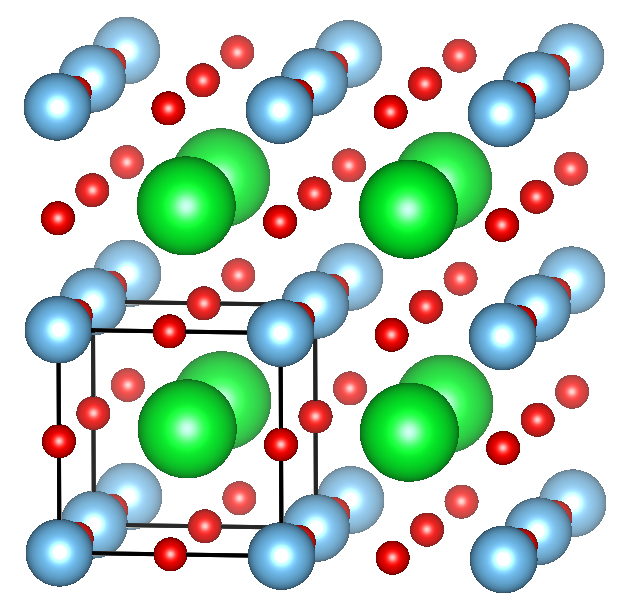}
    \includegraphics[width=0.12\linewidth]{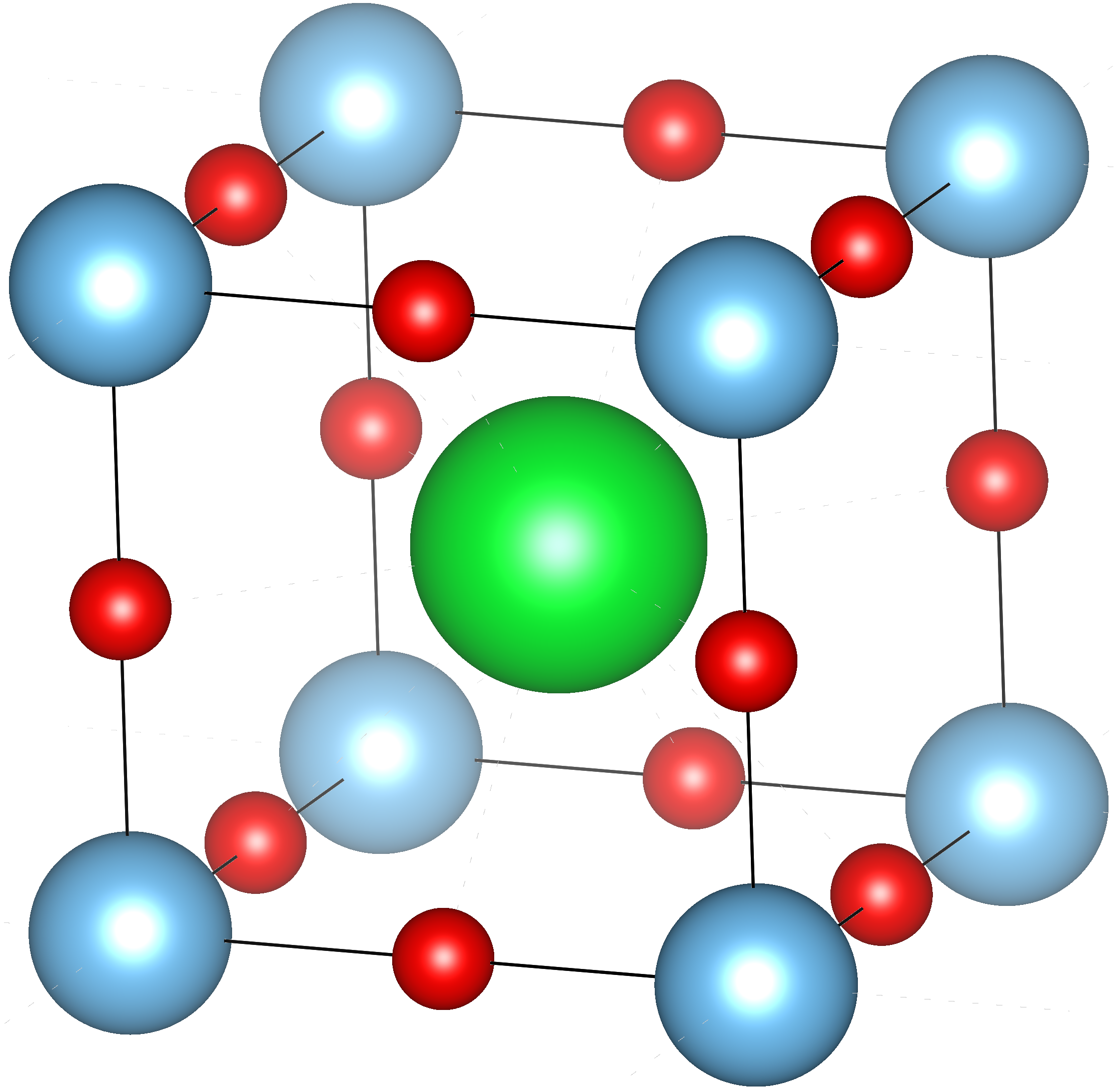} &
    \includegraphics[width=0.12\linewidth]{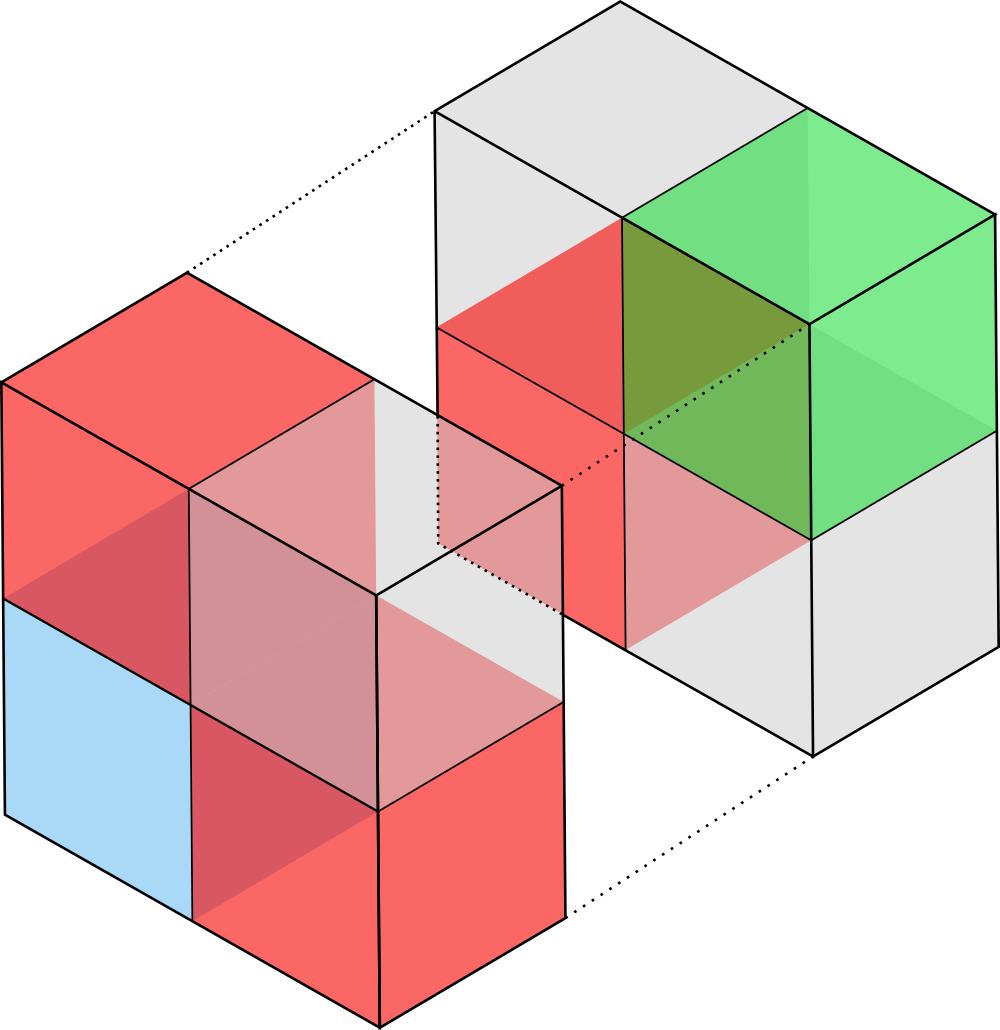} &
    \includegraphics[width=0.12\linewidth]{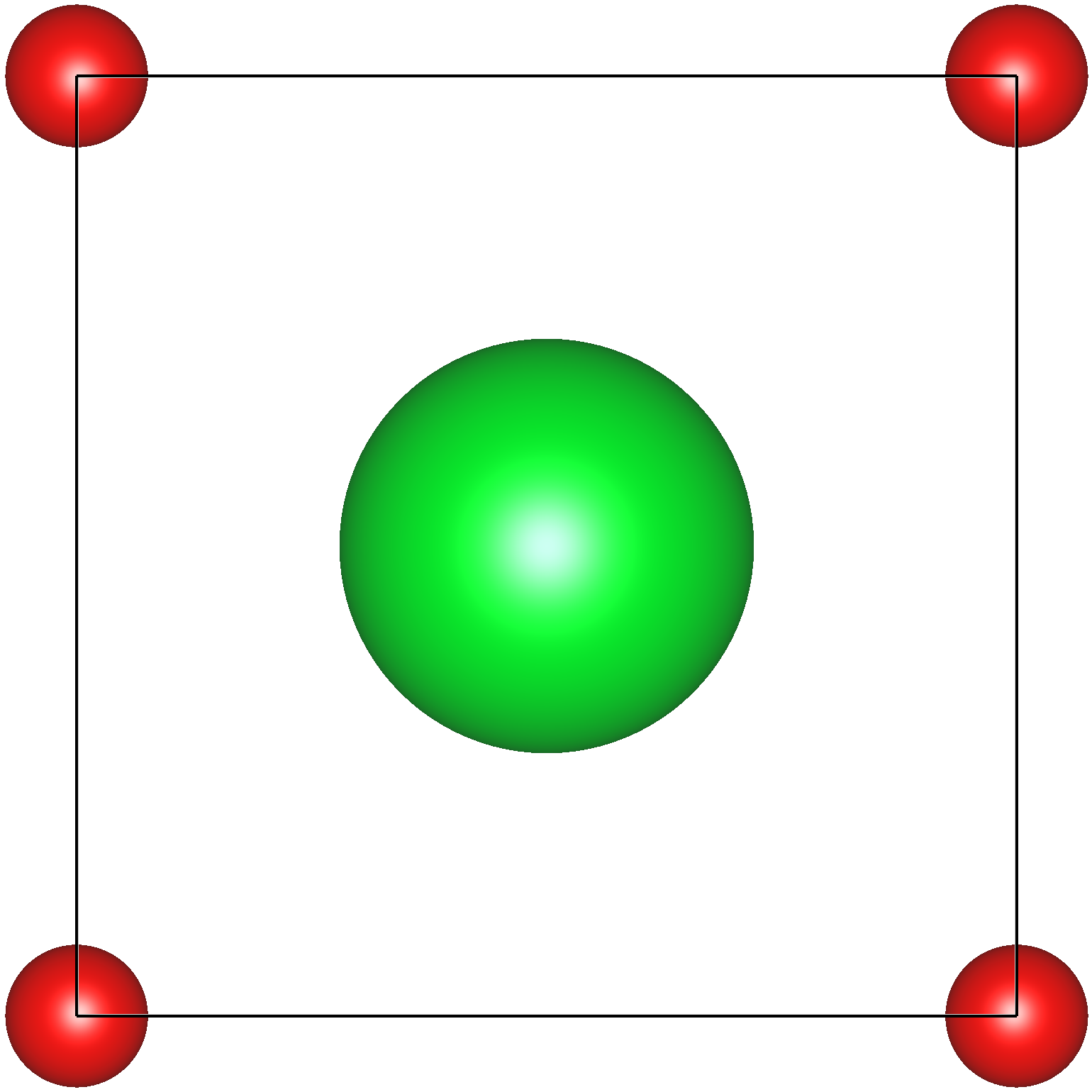}  
    \includegraphics[width=0.12\linewidth]{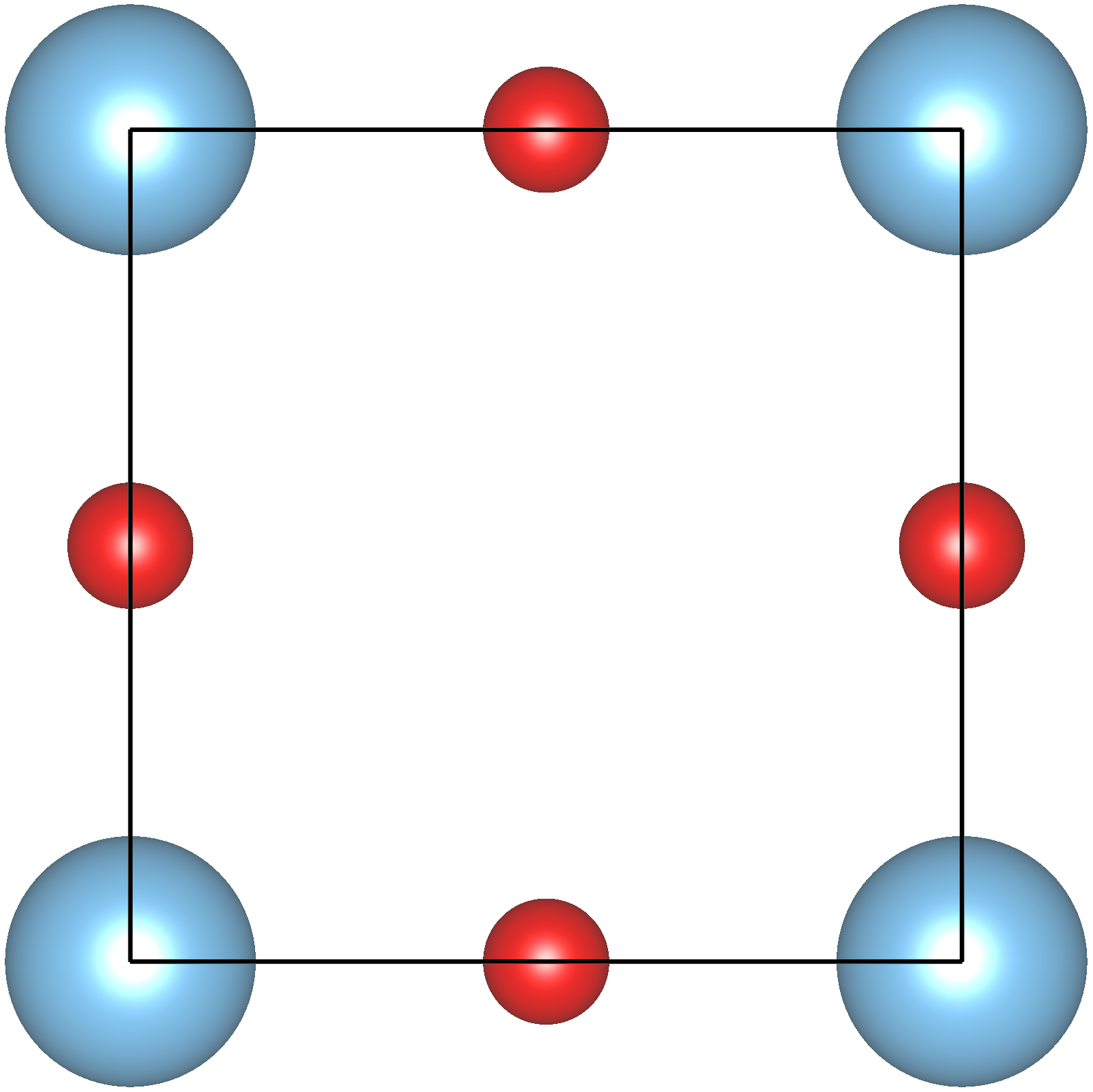}
    \includegraphics[width=0.18\linewidth]{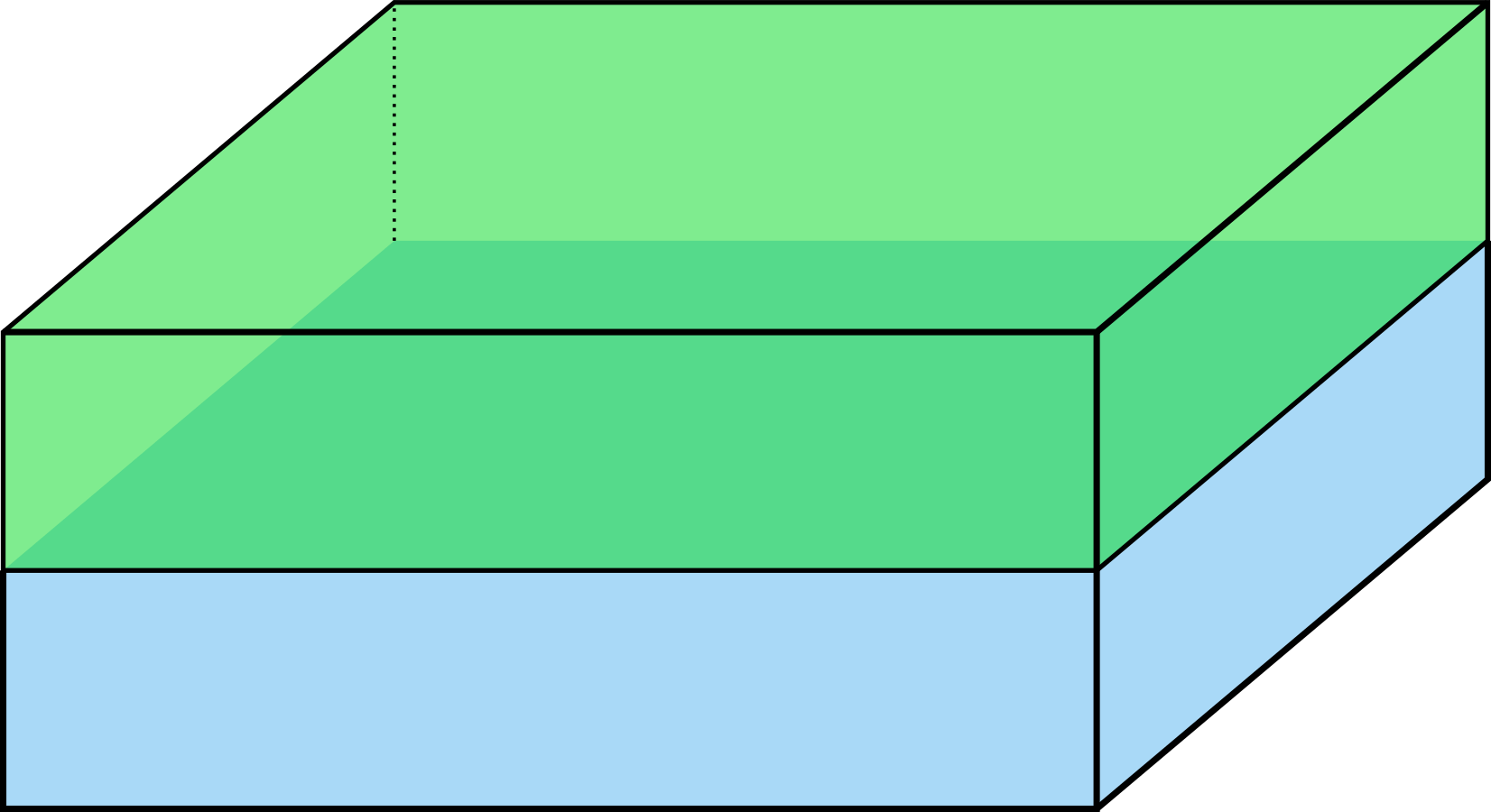}
    \end{tabular}
    \caption{The crystal of $SrTiO_3$ (left) and its 3D (middle) and 1D (right) necklace representations.}
    \label{ch1:fig:crystal}
\end{figure}
When discrete unit cells are represented by layers, then they directly correspond to classical combinatorial necklaces (cyclic words), see \cite{collins2017accelerated}.
Alternatively, 3D representations of unit cells are 3-dimensional necklaces, see
\cite{FUSE2018}.
Figure~\ref{ch1:fig:crystal} provides an illustration of the relationship between crystals and necklaces for both 1 and multiple dimensions, showing how the unit cell of $SrTiO_3$ can be represented as a necklace of size $2\times2\times2$ over an alphabet with four letters (blue, green, red, grey); there is one ion of strontium (green cubelet), one ion of titanium (blue cubelet), three ions of oxygen (red cubelets), and three empty (grey) cubelets. 
Moreover, recent work \cite{mi2021observation} has even shown the need for representing structures that are not only periodic in three spacial dimensions, but also in the fourth dimension of time. The algorithms for multidimensional necklaces can replace currently used 
random generation \cite{collins2017accelerated} of  unit cells leading to potentially identical crystal structures
in the process of   
configuration space exploration.
%exploration of the configuration space of crystal structures.

% \noindent
The paper generalises many results and provides efficient solutions for several problems on $q$-ary $d$-dimensional necklaces of size $\vectorise{n}=(n_1, n_2,\ldots n_d)$. Most notably:
\begin{itemize}
    \item  closed form formulas for the number of 
    %multidimensional
    necklaces, Lyndon  words,  and  atranslational  necklaces
    \item linear time (relative to the necklace size) algorithm for generating next  multidimensional necklaces
    \item  \RankingComplexity time algorithm for ranking a $d$-dimensional necklace, where $N =  \prod_{i = 1}^d n_i$
    \item  
    $O\left(N^{6(d + 1)} \cdot \log^d(q)\right)$ time unranking algorithms for generating
the $i^{th}$ necklace in $\mathcal{N}_q^{\vectorise{n}}$
\item
 $O(N^2 k)$ time  approximation algorithm with approximation factor of {\footnotesize $1 + \frac{
\log_q{(k N)}}{N-\log_q{(kN)}}-\frac{\log^2_q(kN)}{2N(N-\log_q{(kN))}}$} for $k$ center selection on necklaces based on the overlap distance function.
\end{itemize}

\noindent
Beyond classical necklaces we also look at \emph{fixed-content necklaces}.
A set of necklaces has fixed content if every necklace in the set has the same Parikh vector \cite{mateescu_salomaa_salomaa_yu_2001}.
As with general necklaces, there have been results for counting \cite{gilbert1961symmetry}, and generating \cite{Karim2013, Ruskey1999} both fixed content necklaces and bracelets.
Further, Hartman and Sawada provided a polynomial time algorithm to rank and unrank \emph{fixed density} necklaces, i.e. fixed content necklaces restricted to a binary alphabet \cite{Hartman2019}.
In this paper we design polynomial time ranking/unranking  algorithms for $q$-ary multidimensional necklaces  restricted   by a given Parikh vector.

The set of proposed
algorithms for 
dimensions two and three is a strong contribution to field as it fills a gap in the literature and has direct 
real-life applications in the context building algorithmic foundation for analysis of crystal structures. 
Moreover we feel that natural generalisation to any dimension strengthens the paper overall providing universal tools for building efficient algorithms on necklaces of any size.

%The ranking problem requires to compute the rank of the given object, according to some previously fixed order.
%Unranking is the inverse operation which asks for the generation of the $i$-th combinatorial object of size n in some combinatorial class.
%The generation problem requires to list all objects of a given combinatorial class and size, according to some previously fixed order.

%The ranking problem: Given a combinatorial class and an object from that class, compute the rank of the given object, according to some previously fixed order.

%The problem of unranking asks for the generation of the ith combinatorial object of size n in some combinatorial class A,
%according to some well defined order among the objects of size n of the class. 

%The unranking problem generates a combinatorial object
%whose rank and size are given, according to some fixed order.

%The iteration or exhaustive generation gives all objects
%of a given combinatorial class and size, according to some
%previously fixed order.
%The interest of this whole

%The k-center generation problem consists in generating $k$ 
%combinatorial necklaces of a given size which are...

\vspace{-0.3cm}

\section{Preliminaries}
\label{sec:prelims}

% \duncan{In red, things I think are fairly standard and can be assumed of the reader}

% {\color{red} A necklace can be conceptualised as an infinite word resented by a period.}
% Due to the infinite nature of necklaces, they are considered in terms of equivalence classes.
% To define necklaces, the following notions are introduced regarding words.
Let $\Sigma$ be a finite alphabet.
For the remainder of this work we assume $\Sigma$ to be made of symbols corresponding to the set $\{1,2,3,\hdots,q\}$, ordered such that $1 < 2 < 3 < \hdots < q$, and by extension $q = |\Sigma|$ % \duncan{can restate later, but a bit non standard}.
%{\color{red} 
We denote by $\Sigma^*$ the set of all words over $\Sigma$ and by $\Sigma^n$ the set of all words of length $n$.
The notation $\word{w}$ is used to clearly denote that the variable $\word{w}$ is a word.
The length of a word $\word{w} \in \Sigma^*$ is denoted $|\word{w}|$.
We use $\word{w}_i$, for any $i \in \{1,\hdots,|\word{w}|\}$ to denote the $i^{th}$ symbol of $\word{w}$.
Given two words $\word{w},\word{u} \in \Sigma^*$, the \emph{concatenation operation} is denoted $\word{w} : \word{u}$, returning the word of length $|\word{w}| + |\word{u}|$ where $(\word{w} : \word{u})_i$ equals either $\word{w}_i$, if $i \leq |\word{w}|$ or $\word{u}_{i - |\word{w}|}$ if $i > |\word{w}|$.
Given a word $\word{w}$, the \emph{Parikh vector} of $\word{w}$, denoted $\Parikh(\word{w})$ is a $q$ length vector such that $\Parikh(\word{w})_i$ contains the number of times that the $i^{th}$ symbol of $\Sigma$ appears in $\word{w}$.
% For example, given the alphabet $\Sigma = (1,2,3,4)$ and the word  $11124$, $\Parikh(11124) = (3,1,0,1)$.
%}

Let $[n]$ return the ordered set of integers from $1$ to $n$ inclusive.
More generally, let $[i,j]$ return the ordered set of integers from $i$ to $j$ inclusive.
% {\color{red}
Given 2 words $\word{u},\word{v} \in \Sigma^*$,
%where $|\word{u}| \leq |\word{v}|$, 
$\word{u} = \word{v}$ if and only if $|\word{u}| = |\word{v}|$ and $\word{u}_i = \word{v}_i$ for every $i \in [|\word{u}|]$.
A word $\word{u}$ is \emph{lexicographically smaller} than $\word{v}$ if there exists an $i \in [|\word{u}|]$ such that $\word{u}_1 \word{u}_2 \hdots \word{u}_{i-1} = \word{v}_1 \word{v}_2 \hdots \word{v}_{i-1}$ and $\word{u}_i < \word{v}_i$.
% For example, given the alphabet $\Sigma = \{a,b\}$ where $a < b$, the word $aaaba$ is smaller than $aabaa$ as the first 2 symbols are the same and $a$ is smaller than $b$.
% }
For a given set of words $\mathbf{S}$, the rank of $\word{v}$ with respect to $\mathbf{S}$ is the number of words in $\mathbf{S}$ that are smaller than $\word{v}$.

The {\em translation} of a word $\word{w}=\word{w}_1 \word{w}_2 \hdots \word{w}_n$ by $r \in [n-1]$ returns the word $\word{w}_{r + 1} \hdots \word{w}_n : \word{w}_1 \hdots \word{w}_r$, and is denoted by $\Angle{ \word{w} }_r$, i.e. $\Angle{ \word{w}_1 \word{w}_2 \hdots \word{w}_n }_r = \word{w}_{r + 1} \hdots \word{w}_n \word{w}_1 \hdots \word{w}_r$.
Under the translation operation, $\word{u}$ is equivalent to $\word{v}$ if $\word{v} = \Angle{ \word{w} }_r$ for some $r$.
The $t^{th}$ power of a word $\word{w} = \word{w}_1 \hdots \word{w}_n$, denoted $\word{w}^t$, equals $\word{w}$ repeated $t$ times.
% For example $(aab)^3 = aabaabaab$.
A word $\word{w}$ is \emph{periodic} if there is some word $\word{u}$ and integer $t \geq 2$ such that $\word{u}^t = \word{w}$.
% Equivalently, word $\word{w}$ is \emph{periodic} if there exists some translation $0 < r < |\word{w}|$ where $\word{w} = \Angle{ \word{w}}_r$.
A word is \emph{aperiodic} if it is not periodic.
The \emph{period} of a word $\word{w}$ is the aperiodic word $\word{u}$ such that $\word{w} = \word{u}^t$.

A \emph{necklace} is the equivalence class of words under the translation operation.
A word $\word{w}$ is written as $\necklace{w}$ when treated as a necklace.
Given a necklace $\necklace{w}$, the \emph{canonical representation} of $\necklace{w}$ is the lexicographically smallest element of the set of words in the equivalence class $\necklace{w}$.
The canonical representation of $\necklace{w}$ is denoted $\Angle{ \necklace{w} }$, and the $r^{th}$ shift of the canonical representation is denoted $\Angle{\necklace{w}}_{r}$.
Given a word $\word{w}$, $\Angle{\word{w}}$ denotes the canonical representation of the necklace containing $\word{w}$, i.e. the representative of $\necklace{u}$ where $\word{w} \in \necklace{u}$.
The set of necklaces of length $n$ over an alphabet of size $q$ is denoted $\mathcal{N}_q^n$, the size of which is given by $|\mathcal{N}_q^n |$.
Let $\word{w} \in \mathcal{N}_q^n$ denote that the word $\word{w}$ is the canonical representation of some necklace $\necklace{w} \in \mathcal{N}_q^n$.
An aperiodic necklace, known as a \emph{Lyndon word}, is a necklace representing the equivalence class of some aperiodic word.
% Note that the number of unique words represented by a Lyndon word of length $n$ is $n$.
The set of Lyndon words of length $n$ over an alphabet of size $q$ is denoted $\mathcal{L}_q^n$.
A necklace $\necklace{w}$ has \emph{fixed content} for some given Parikh vector $\vectorise{p}$ if $\Parikh(\necklace{w}) = \vectorise{p}$.
The set of fixed content necklaces for some vector $\vectorise{p}$ is denoted by $\mathcal{N}^{n}_{\vectorise{p}}$, and the set of fixed content Lyndon words by $\mathcal{L}^n_{\vectorise{p}}$.

%{\color{red}
The \emph{subword} of a word $\word{w}$ denoted $\word{w}_{[i,j]}$ is the word $\word{u}$ of length $|\word{w}| + j - i - 1 \bmod |\word{w}||)$ such that $\word{u}_{a} = \word{w}_{i - 1 + a \bmod |\word{w}|}$.
%, i.e. the word such that the $a^{th}$ symbol of $\word{u}$ corresponds to the symbol at position $i + a \bmod n$ of $\word{w}$.
% Similarly the \emph{subword} of the necklace $\necklace{w}$, denoted $\necklace{w}_{[i,j]}$ corresponds to the subword of the canonical representation of $\necklace{w}$ between the same positions, formally $\necklace{w}_{[i,j]} = \left(\Angle{\necklace{w}}\right)_{[i,j]}$.
For notation $\word{u} \sqsubseteq \word{w}$ denotes that $\word{u}$ is a subword of $\word{w}$. 
%and $\word{u} \sqsubseteq \necklace{w}$ denotes that $\word{u}$ is the subword of the necklace $\necklace{w}$.
Further, $\word{u} \sqsubseteq_{i} \word{w}$ denotes that $\word{u}$ is a subword of $\word{w}$ of length $i$.
If $\word{w} = \word{u} : \word{v}$, then $\word{u}$ is a prefix and $\word{v}$ is a suffix.
% A prefix or suffix of a word $\word{u}$ is \emph{proper} if its length is smaller than $|\word{u}|$.
%}

As both necklaces and Lyndon words are classical objects, there are many fundamental results regarding each objects.
The first results for these objects were equations determining the number of necklaces or Lyndon words of a given length.
The number of necklaces is given by the equation {\footnotesize $|\mathcal{N}_q^n| = \frac{1}{n} \sum\limits_{d | n} \phi\left( \frac{n}{d}\right)q^d$} where $\phi(n)$ is Euler's totient function.
Similarly the number of Lyndon words is given with the equation {\footnotesize$|\mathcal{L}_q^n| = \sum\limits_{d | n} \mu\left(\frac{n}{d} \right) |\mathcal{N}_q^d |$}, where $\mu(x)$ is the M\"{o}bius function.
A proof of these equations is provided in \cite{Graham1994}.
The problem of \emph{generating} every necklace in the set $\mathcal{N}_q^n$ for any $n,q \in \mathbb{N}$ in lexicographic order was solved first by Fredricksen and Maiorana \cite{Fredricksen1978}.
This algorithm was shown to run in constant amortised time (CAT) in \cite{Ruskey1992}.
A more direct CAT generation algorithm was introduced in \cite{Cattell2000}.

Recently the dual problems of \emph{ranking} and \emph{unranking} necklaces have been studied.
The \emph{rank} of a word $\word{w}$ in the set of necklaces $\mathcal{N}_q^n$ is in this work defined as the number of necklaces with a canonical representation smaller than $\word{w}$.
The \emph{unranking} process is effectively the reverse of this.
Given an integer $i \in [|\mathcal{N}_q^n|]$, the goal of the unranking process for $i$ is to determine the necklace $\mathcal{N}_q^n$ with a rank of $i$.
Lyndon words were first ranked by Kociumaka, Radoszewski, and Rytter \cite{Kociumaka2014} without tight complexity bounds.
The first algorithm to rank necklaces was given by Kopparty, Kumar, and Saks \cite{Kopparty2016}, also without tight bounds on the complexity.
A quadratic time algorithm for ranking both Lyndon necklaces was provided by Sawada and Williams \cite{Sawada2017}, who also provided a cubic time unranking algorithm.

% \subsection{Multidimensional Necklaces}

In order to establish multidimensional necklaces, notation for \emph{multidimensional words} must first be introduced.
A \emph{$d$-dimensional word} over $\Sigma$ is an array of size $\vectorise{n} = (n_1,n_2,\hdots,n_d)$ of elements from $\Sigma$.
In this work we tacitly assume that $n_1 \leq n_2 \leq \hdots \leq n_d$ unless otherwise stated.
Let $|\word{w}|$ be the size of $\word{w}$.
Given a size vector $\vectorise{n} = (n_1,n_2,\hdots,n_d)$, $\Sigma^{\vectorise{n}}$ is used to denote the set of all words of size $\vectorise{n}$ over $\Sigma$.
%Where it is clear from the context, $N$ is used to denote $n_1 \cdot n_2 \cdot \hdots \cdot n_d$ for a dimension vector $\vectorise{n}$.
% Let $N = n_1 \cdot n_2 \cdot \hdots \cdot n_d$ for a size vector $\vectorise{n}$.
For notation, given a vector $\vectorise{n} = (n_1, n_2, \hdots, n_d)$ where every $n_i \geq 0$, $[\vectorise{n}]$ is used to denote the set $\{(x_1, x_2, \hdots, x_d) \in \mathbb{N}^d | \forall i \in [d], x_i \leq n_i\}$.
Similarly $[\vectorise{m}, \vectorise{n}]$ is used to denote the set $\{(x_1, x_2, \hdots, x_d) \in \mathbb{N}^d | \forall i \in [d], m_i \leq x_i \leq n_i\}$.

For a $d$-dimensional word $\word{w}$, the notation $\word{w}_{(p_1,p_2,\hdots,p_d)}$ is used to refer to the symbol at position $(p_1,$ $p_2,$ $\hdots,$ $p_d)$ in the array.
Given 2 $d$-dimensional words $\word{w},\word{u}$ such that $|\word{w}| = (n_1,n_2,\hdots, n_{d - 1}, a)$ and $|\word{u}|$ $=$ $($ $n_1,$ $n_2,$ $\hdots,$ $n_{d - 1},$ $b)$, the concatenation $\word{w} : \word{u}$ is performed along the last coordinate, returning the word $\word{v}$ of size $(n_1, n_2, \hdots, n_{d - 1}, a + b)$ such that
$\word{v}_{\vectorise{p}} = \word{w}_{\vectorise{p}}$ if $p_d \leq a$ and $\word{v}_{\vectorise{p}} =  \word{u}_{(p_1,p_2,\hdots, p_{d - 1}, p_d - a)}$ if $p_d > a$.
For example given the words $\word{w} =$
{\footnotesize
    $\begin{bmatrix}
        a & a & a & b\\
        a & a & b & a%\\
        %b & a & a & a
    \end{bmatrix}$} and $\word{u} =$ 
    {\footnotesize $\begin{bmatrix}
        b & b & b & b\\
        b & b & b & b
    \end{bmatrix}$}, $\word{w} : \word{u} =$ {\footnotesize $\begin{bmatrix}
        a & a & a & b\\
        a & a & b & a\\
        % b & a & a & a\\
        b & b & b & b\\
        b & b & b & b
    \end{bmatrix}$}.
% See Figure \ref{fig:concatination_example}.

% \begin{figure}[ht]
%     \centering
%     \[
%     \word{w} =
%     \begin{bmatrix}
%         a & a & a & b\\
%         a & a & b & a\\
%         b & a & a & a\\
%         b & a & a & a
%     \end{bmatrix},
%     \word{u} = \begin{bmatrix}
%         b & b & b & b\\
%         b & b & b & b
%     \end{bmatrix},
%     \word{w} : \word{u} = \begin{bmatrix}
%         a & a & a & b\\
%         a & a & b & a\\
%         b & a & a & a\\
%         b & a & a & a\\
%         b & b & b & b\\
%         b & b & b & b
%     \end{bmatrix}
%     \]
%     \caption{An example of the connotation between the 2D words $\word{w} \in \Sigma^{4,4}$ and $\word{u} \in \Sigma^{4,2}$, returning the word $\word{w} : \word{u} \in \Sigma^{4,6}$.}
%     \label{fig:concatination_example}
% \end{figure}

A \emph{multidimensional cyclic subword} of $\word{w}$ of size $\vectorise{m}$ is denoted $\word{v} \sqsubseteq_{\vectorise{m}} \word{w}$.
As in the 1D case, a subword is defined by a starting position in the original word and set of size defining the size of the subword.
The subword $\word{v} \sqsubseteq \word{w}$ starting at position $\vectorise{p}$ with size $\vectorise{m}$ is the word $\word{v}$ such that $\word{v}_{\vectorise{i}} = \word{w}_{\vectorise{j}}$ for all $\vectorise{j}$ of the form $(p_1 + i_1 \bmod n_1, p_2 + i_2 \bmod n_2, \hdots, p_d + i_d \bmod n_d)$.
% \rev{$\bar w$ does not have the right font here.}
Such a subword $\word{v}$ is denoted by $\word{w}_{[\vectorise{p},\vectorise{m}]}$.
%For notation, $\word{w}_{\vectorise{p},\vectorise{m}}$ is the subword of $\word{w}$ starting at position $\vectorise{p}$ of size $\vectorise{m}$.
One important class of subwords are \emph{slices}, an example of which is given in Figure \ref{fig:2DPropertiesOverview}.
The $i^{th}$ slice of $\word{w}$, denoted by $\word{w}_{i}$, is the subword of size $(n_1,n_2,\hdots,n_{d - 1},1)$ starting at position $(i,1,\hdots,1,1)$ of $\word{w}$.
In the 2D case, the $i^{th}$ slice corresponds to the $i^{th}$ row of a word.
This work uses $\word{w}_{[i,j]}$ to denote $\word{w}_{i} : \word{w}_{i + 1} : \hdots : \word{w}_{j}$.
A \emph{prefix} of length $l$ for a multidimensional word $\word{w}$ is the first $l$ slices of $\word{w}$ in order.
A \emph{suffix} of length $l$ for a multidimensional word $\word{w}$ is the last $l$ slices of $\word{w}$ in order.
% In the 2D case, prefix and suffix of length $l$ corresponds to the first and last $l$ rows respectively.
% An overview of these properties are given in Figure \ref{fig:2DPropertiesOverview}.

\begin{figure}
    \label{fig:2DPropertiesOverview}
    \centering
    \includegraphics[scale=0.7]{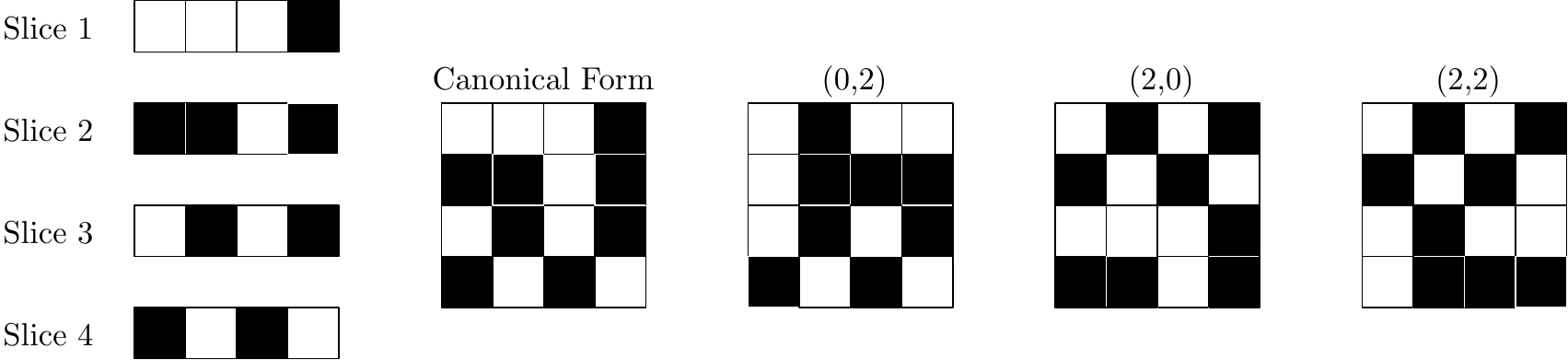}
    \caption{Example of a 2-dimensional word $\word{w}$ of size $(4,4)$ over a binary alphabet: the 4 slices of $\word{w}$; the canonical representation of $\word{w}$; and three translations of $\word{w}$.}
\end{figure}

A $d$-dimensional translation $r$ is defined by a vector $(r_1,r_2,\hdots,r_d)$.
The translation of the word $\word{w} \in \Sigma^{\vectorise{n}}$ by $r$, denoted $\Angle{\word{w}}_{r}$, returns the word $\word{v} \in \Sigma^{\vectorise{n}}$ such that $\word{v}_{\vectorise{p}} = \word{w}_{\vectorise{j}}$ for all $\vectorise{p} \in [\vectorise{n}]$ where $\vectorise{j} = (p_1 + r_1 \bmod n_1, p_2 + r_2 \bmod n_2, \hdots,p_d + r_d \bmod n_d)$.
It is assumed that $r_i \in [0, n_i - 1]$, so the set of translations is equivalent to the direct product of the cyclic groups $Z_{n_1} \times Z_{n_2} \times \hdots \times Z_{n_d}$.
% For notation let $Z_{\vectorise{n}} =  Z_{n_1} \times Z_{n_2} \times \hdots \times Z_{n_d}$.
Given two translations $r = (r_1, r_2, \hdots, r_d)$ and $t = (t_1, t_2, \hdots, t_d)$ in $Z_{\vectorise{n}} $, $t + r$ is used to denote the translation $(r_1 + t_1 \bmod n_1, r_2 + t_2 \bmod n_2, \hdots, r_d + t_d \bmod n_d)$.

\begin{definition}
A \textbf{multidimensional necklace} $\necklace{w}$ is an equivalence class of all multidimensional words under the translation operation.
\end{definition}

\noindent
Informally, given a necklace $\necklace{w}$ containing the word $\word{v}$, $\necklace{w}$ contains every word $\word{u}$ where there exists some translation $\vectorise{r}$ such that $\Angle{\word{v}}_{\vectorise{r}} = \word{u}$.
Let $\mathcal{N}_{q}^{\vectorise{n}}$ denote the set of necklaces of size $\vectorise{n}$ over an alphabet of size $q$.
As in the 1D case, a \emph{canonical representation} of a multidimensional necklace is defined as the smallest element in the equivalence class, denoted  $\Angle{\necklace{w}}$.
Similarly, given a word $\word{v} \in \necklace{w}$, $\Angle{\word{v}}$ denotes the canonical representation of the necklace $\necklace{w}$, i.e. $\Angle{\word{v}} = \Angle{\necklace{w}}$.
To determine the smallest element in the equivalence class, an ordering needs to be defined.
First, we introduce an ordering over translations.

\begin{definition}
\label{def:group_ordering}
Let $Z_{\vectorise{n}}$ be the direct product of the cyclic groups $Z_{n_1} \times Z_{n_2} \times \hdots \times Z_{n_d}$, i.e. the set of all translations of words of size $\vectorise{n}$.
The translation $g \in Z_{\vectorise{n}}$ is indexed by the injective function 
$\indexFunc(g)\rightarrow \sum\limits_{i = 1}^d \left ( g_i \cdot \prod\limits_{j = 1}^{i - 1} n_j \right)$.
Given two translations $g,t \in Z_{\vectorise{n}}$, $g < t$ if and only if $\indexFunc(g) < \indexFunc(t)$.
\end{definition}

\noindent
% The translation $g \in Z_{\vectorise{n}}$ is smaller than $t \in Z_{\vectorise{n}}$ if $\indexFunc(g) < \indexFunc(t)$.
Note that $I = (0, 0, \ldots, 0)$ is the smallest translation and $(n_1-1, n_2-1, \ldots, n_d - 1)$ is the largest.
Further, the translation $(n_1 + i_1, n_2 + i_2, \ldots, n_d + i_d)$ is equivalent to the translation $(i_1, i_2, \ldots, i_d)$.
Using this index an ordering on multidimensional words is defined recursively.
The key idea is to compare each slice based on the canonical representations.
For notation, given two words $\word{u},\word{s} \in \necklace{w}$, let $G(\word{u},\word{s})$ return the smallest translation $g$ where $\Angle{\word{u}}_g = \word{s}$.
% Note that $G$ can be computed in $O(N^2)$ time by simply checking each translation in $Z_{|\word{u}|}$.

\begin{definition}
\label{def:orderinging}
Let $\word{w},\word{u} \in \Sigma^{\vectorise{n}}$ and let $i \in [n_d]$ be the smallest index such that $\word{w}_i \neq \word{u}_i$.
Then $\word{w} < \word{u}$ if either $\Angle{\word{w}_i} < \Angle{ \word{u}_i }$, or $\Angle{\word{w}_i} = \Angle{ \word{u}_i }$ and $\indexFunc(G(\word{w}_i,\Angle{\word{w}_i})) <  \indexFunc(G(\word{u}_i,\Angle{\word{u}_i}))$.
Further, given necklaces $\necklace{w}$ and $\necklace{u}$, $\necklace{w} < \necklace{u}$ if and only if $\Angle{ \necklace{w} } < \Angle{ \necklace{u}}$.
\end{definition}

\noindent
Note that a 0-dimensional necklace is simply a symbol from $\Sigma$.
Hence for 1D necklaces this ordering is equivalent to the lexicographical ordering.
An example of the ordering is given in Figure \ref{fig:comparison_of_necklaces}.
Both $\mathcal{N}^{\vectorise{n}}_q$ and $\Sigma^{\vectorise{n}}$ are assumed to be ordered as in Definition \ref{def:orderinging}.
% When not clear from context, we refer to this ordering as \emph{Necklace recursive ordering}, in reference to both the recursive definition of the ordering and the use of necklaces as a key component of the ordering.
The {\emph{rank}} of a necklace $\necklace{w} \in \mathcal{N}^{\vectorise{n}}_q$ is defined as the number of necklaces smaller than $\necklace{w}$ in $\mathcal{N}^{\vectorise{n}}_q$.
In the other direction, the $i^{th}$ necklace in $\mathcal{N}^{\vectorise{n}}_q$ is the necklace $\necklace{w} \in \mathcal{N}^{\vectorise{n}}_q$ with the rank $i$, i.e. the necklace $\necklace{w}$ for which there are $i$ smaller necklaces.

\begin{figure}[ht]
{\small
    \centering
    \[
    \word{w} = \begin{bmatrix}
    \word{w}_1\\
    \word{w}_2\\
    \word{w}_3\\
    \word{w}_4
    \end{bmatrix} = 
    \begin{bmatrix}
        a & a & a & b\\
        a & a & b & a\\
        b & a & a & a\\
        b & a & a & a
    \end{bmatrix},  \word{u} = \begin{bmatrix}
        \word{u}_1\\
        \word{u}_2\\
        \word{u}_3\\
        \word{u}_4 
    \end{bmatrix} = 
    \begin{bmatrix}
        a & a & a & b\\
        a & a & b & a\\
        a & b & a & a\\
        b & a & a & a
    \end{bmatrix},
    \word{v} = \begin{bmatrix}
        \word{v}_1\\
        \word{v}_2\\
        \word{v}_3\\
        \word{v}_4
    \end{bmatrix} =
    \begin{bmatrix}
        a & a & a & b\\
        a & a & b & a\\
        a & a & b & b\\
        b & a & a & a
    \end{bmatrix}
    \]
    \caption{     
    An example of three words, $\word{w}, \word{u},$ and $\word{v}$, ordered as follows $\word{w} < \word{u} < \word{v}$.
    Note that $\word{w}_1:\word{w}_2 = \word{v}_1:\word{v}_2 = \word{u}_1:\word{u}_2$. However, $\Angle{ \word{w}_3 } = \Angle{ \word{u}_3 } = aaab$, which is smaller than $\Angle{ \word{v}_3 } = aabb$.
    Further, $\word{w}_3 < \word{u}_3$ as $G(\word{w}_3, \Angle{ \word{w}_3 }) = 1$ and $G(\word{u}_3, \Angle{ \word{u}_3 }) = 2$, which is larger than $1$.
    }
    \label{fig:comparison_of_necklaces}
    }
\end{figure}

% In order to answer some of the key questions regarding multidimensional necklaces, there are two further concepts that need to be defined for multidimensional necklaces.
% The first is the \emph{period} of a word.
One important concept for multidimensional words is that of the \emph{period} of a word.
Informally the period of $\word{w}$ of size $\vectorise{n}$ can be thought of as the smallest subword that can tile $d$-dimensional space equivalently to $\word{w}$.
To define the period of a word, it is easiest to first define the concept of \emph{aperiodicity}.

\begin{definition}
A word $\word{w}$ of size $\vectorise{n}$ is \textbf{aperiodic} if there exists no subword $\word{v} \sqsubseteq \word{w}$ of size $\vectorise{m} \neq \vectorise{n}$ such that $m_i \leq n_i$ for every $i \in [1, d]$, and $\word{w}_{\vectorise{j}} = \word{v}_{\vectorise{j}'}$ where $\vectorise{j}' = (j_1 \bmod m_1, j_2 \bmod m_2, \hdots, j_d \bmod m_d)$ for every position $\vectorise{j} \in [\vectorise{n}]$ in $\word{w}$.
\end{definition}

\begin{definition}
\label{def:periodic_multidimensional}
The \textbf{period} of a word $\word{a} \in \Sigma^{\vectorise{n}}$, denoted $\period(\word{a})$, is the aperiodic subword $\word{b} \sqsubseteq \word{a}$ of size $\vectorise{m}$ such that $\word{a}_{\vectorise{i}} = \word{b}_{\vectorise{i}'}$ for every position $\vectorise{i} \in [\vectorise{n}]$ and $\vectorise{i}' = (i_1 \bmod m_1, i_2 \bmod m_2,\hdots,i_d \bmod m_d)$.
\end{definition}

\noindent
By Definition \ref{def:periodic_multidimensional} every word, including aperiodic ones, has a unique period \cite{GAMARD201758}.
In the case of an aperiodic word $\word{w}$, the period is simply $\word{w}$.
A multidimensional necklace $\necklace{w}$ is aperiodic if every word $\word{v} \in \necklace{w}$ is aperiodic.
% Note that if some word in $\necklace{w}$ is aperiodic, then every word is.
An aperiodic necklace is called a \emph{Lyndon word}.
The set of Lyndon words of size $\vectorise{n}$ over an alphabet of size $q$ is denoted $\mathcal{L}_q^{\vectorise{n}}$. %in reference to Lyndon words, 1D aperiodic necklaces.
A related but distinct concept to aperiodic words are \emph{atranslational} words and necklaces.
A word $\word{w}$ is \emph{atranslational} if there exists no translation $g \neq (n_1,n_2,\hdots,n_d)$ such that $\word{w} = \Angle{\word{w}}_g$.
Equivalently, a necklace $\necklace{w}$ is atranslational if $\Angle{\necklace{w}}$ is atranslational.
The set of atranslational necklaces of size $\vectorise{n}$ over an alphabet of size $q$ is denoted $\mathcal{A}_q^{\vectorise{n}}$.

\begin{definition}
\label{def:atranslational}
A necklace $\necklace{w}$ if size $\vectorise{n}$ is \textbf{atranslational} if there exists no pair of translations $g,h \in Z_{\vectorise{n}}$ where $g \neq h$ and $\Angle{\necklace{w}}_{g} = \Angle{\necklace{w}}_{h}$.
\end{definition}

\noindent
In 1D every aperiodic necklace is atranslational, while in any higher dimension every atranslational word is aperiodic, although not every aperiodic word is atranslational.
By extension $\mathcal{A}_q^{\vectorise{n}} \subseteq \mathcal{L}_q^{\vectorise{n}} \subseteq \mathcal{N}_q^{\vectorise{n}}$.
A visual example of this relationship is given in Figure \ref{fig:class_relationship}.
For example {\footnotesize $\begin{bmatrix}
    a & b\\
    b & a
\end{bmatrix}$} is aperiodic but not atranslational, as there are only two unique representations of the necklace.
On the other hand {\footnotesize $\begin{bmatrix}
    a & a\\
    a & b
\end{bmatrix}$} is both atranslational and aperiodic.
For notation, $TR(\word{w})$ is used to denote the index of the smallest translation $g \in Z_{\vectorise{n}}$ where $\Angle{\word{w}}_g = \word{w}$.
The \emph{translational period} of a word $\word{w}$ is the subword $\word{u} \sqsubseteq_{\vectorise{g}} \word{w}$ where $\vectorise{g}$ is the smallest translations such that $\Angle{\word{w}}_{\vectorise{g}} = \word{w}$ and $\word{u}_{\vectorise{i}} = \word{w}_{\vectorise{i}}, \forall \vectorise{i} \in [\vectorise{g}]$.
% Similarly $TP(\word{w})$ is used to denote the index of the smallest translation $g \in Z_{\vectorise{n}}$ where $\Angle{\word{w}}_g = \word{w}$.
The following proposition (formally proven in Section \ref{sec:counting}) characterises the structure of any word that is aperiodic, but not atranslational.

\begin{proposition}
\label{prop:translational_symbolisation}
Every word $\word{w} \in \mathcal{L}^{\vectorise{n}}_q$ is either in $\mathcal{A}^{\vectorise{n}}_q$ or $\word{w} = \word{u}^p : \Angle{\word{u}^p}_g : \hdots : \Angle{\word{u}^p}_{g^{t - 1}}$ where:
\begin{itemize}
    \item $g$ is a translation where $g_d = p$ and there exists no translation $r < g$ where $\Angle{\word{u}^p}_r = \word{u}^p$.
    \item $\word{u} \in \mathcal{L}_q^{(n_1, \hdots n_{d - 1} r/p)}$.
    $t = \frac{n_d}{r}$ and is the smallest value greater than 0 such that $g^t = I$.
\end{itemize}
\end{proposition}

\begin{figure}
    \centering
    \includegraphics[scale=0.5]{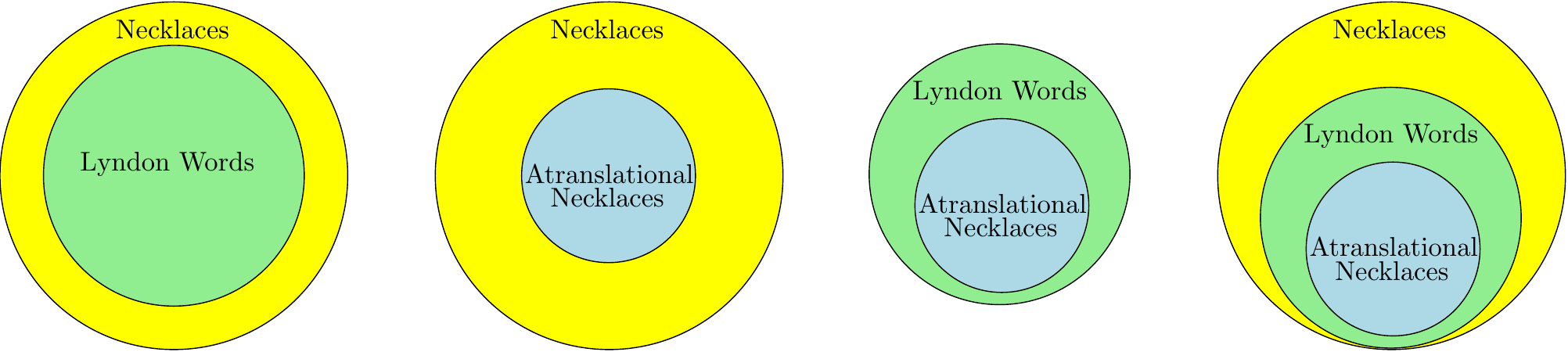}
    \caption{Visual representation of the relationships between Necklaces, Lyndon words and atranslational necklaces, namely that $\mathcal{A}_q^{\vectorise{n}} \subseteq \mathcal{L}_q^{\vectorise{n}} \subseteq \mathcal{N}_q^{\vectorise{n}}$.}
    \label{fig:class_relationship}
\end{figure}

\noindent
% Proposition \ref{prop:translational_symbolisation} is formally proven in Section \ref{sec:counting}.
As in the 1D case, the set of \emph{fixed-content multidimensional necklaces} is defined.
Given a Parikh vector $\vectorise{p}$, the set of multidimensional necklaces of size $\vectorise{n}$ with the Parikh vector $\vectorise{p}$ is denoted $\mathcal{N}_{\vectorise{p}}^{\vectorise{n}}$.

\begin{definition}
The set of necklaces $\mathcal{N}_{\vectorise{p}}^{\vectorise{n}} \subseteq \mathcal{N}_q^{\vectorise{n}}$ contains every necklace $\necklace{w} \in  \mathcal{N}_q^{\vectorise{n}}$ where the Parikh vector of $\necklace{w}$ equals $\vectorise{p}$, i.e. $\Parikh(\necklace{w}) = \vectorise{p}$. 
\end{definition}

\noindent
As in the unconstrained setting, fixed-content necklaces is further reduced to the set of fixed content Lyndon words, denoted $\mathcal{L}_{\vectorise{p}}^{\vectorise{n}}$, and the set of fixed content atranslational necklaces, denoted $\mathcal{A}_{\vectorise{p}}^{\vectorise{n}}$.

\vspace{-0.2cm}
\section{Overview of Results}
%
% Our results can be seen as a generalisation of the main operations of counting, ranking, generating, and unranking from the 1D setting to the multidimensional case.
% % Section \ref{sec:prelims} provides formal definitions of multidimensional necklaces, along with introducing the notation that we utilise in the remainder of the paper.
% This section gives a high level overview of our results, allowing a basic understanding while leaving the full proofs for later in the paper.
%
\vspace{-0.1cm}
\subsection{Counting}
% \vl{atranslational bit might appear earlier, only one or few formulae can be given}
%
Section \ref{sec:counting} provides results regarding counting the number of multidimensional necklaces, Lyndon words, and atranslational necklaces.
As well as being important results in their own right, Theorems \ref{thm:necklace_counting}, \ref{thm:Lyndon_counting} and \ref{thm:lyndon_to_atranslational} provide both closed form formulas to count the cardinality of these sets, along with relationships between the sets.
These relationships are particular use for our ranking techniques.

% \begin{theorem}
% \label{thm:necklace_counting}
% The number of necklaces of size $\vectorise{n}$ over an alphabet of size $q$ is given by the equation:
% $$|\mathcal{N}_q^{\vectorise{n}}| = \frac{1}{N} \sum\limits_{f_1 | n_1} \phi\left(f_1\right) \sum\limits_{f_2 | n_2} \phi\left(f_2\right) \hdots \sum\limits_{f_d | n_d} \phi\left(f_d \right) q^{(N/\lcm{(f_1, f_2, \hdots, f_d)})}$$
% Where $N = n_1 \cdot n_2 \cdot \hdots \cdot n_d$ and $\phi(x)$ is Euler's totient function.
% \end{theorem}

\begin{thm_necklace_counting}
The number of necklaces of size $\vectorise{n}$ over an alphabet of size $q$ is given by the equation:
$$|\mathcal{N}_q^{\vectorise{n}}| = \frac{1}{N} \sum\limits_{f_1 | n_1} \phi\left(f_1\right) \sum\limits_{f_2 | n_2} \phi\left(f_2\right) \hdots \sum\limits_{f_d | n_d} \phi\left(f_d \right) q^{(N/\lcm{(f_1, f_2, \hdots, f_d)})}$$
Where $N = \prod_{i = 1}^d n_i$ and $\phi(x)$ is Euler's totient function.
\end{thm_necklace_counting}

\noindent
Theorem \ref{thm:necklace_counting} is derived using the P\'{o}lya enumeration formula.
This set is used as the basis for our remaining counting equations.
Theorem \ref{thm:Lyndon_counting} shows how to use the number of necklaces as a subroutine in order to find the number of Lyndon words.

% \begin{theorem}
% \label{thm:Lyndon_counting}
% The number of Lyndon words of size $\vectorise{n}$ over an alphabet of size $q$ is given by the equation:
% $$|\mathcal{L}_q^{\vectorise{n}}| = \sum\limits_{f_1 | n_1} \mu\left(\frac{n_1}{f_1}\right) \sum\limits_{f_2 | n_2}\mu\left(\frac{n_2}{f_2}\right)\hdots \sum\limits_{f_d | n_d} \mu\left(\frac{n_d}{f_d}\right) |\mathcal{N}_q^{f_1,f_2\hdots f_d}|$$
% Where $\mu(x)$ is the M\"{o}bius function.
% \end{theorem}

\begin{thm_lynodn_counting}
The number of Lyndon words of size $\vectorise{n}$ over an alphabet of size $q$ is given by the equation:
\[
|\mathcal{L}_q^{\vectorise{n}}| = \sum\limits_{f_1 | n_1} \mu\left(\frac{n_1}{f_1}\right) \sum\limits_{f_2 | n_2}\mu\left(\frac{n_2}{f_2}\right)\hdots \sum\limits_{f_d | n_d} \mu\left(\frac{n_d}{f_d}\right) |\mathcal{N}_q^{f_1,f_2\hdots f_d}|
\]
Where $\mu(x)$ is the M\"{o}bius function.
\end{thm_lynodn_counting}

\noindent
Theorem \ref{thm:Lyndon_counting} is shown by first expressing the number of necklaces in terms of Lyndon words, then inverting this formula.
Lyndon words are in turn used as the basis for counting the number of atranslational necklaces.
% Finally Theorem \ref{thm:lyndon_to_atranslational} relates the number of Lyndon words to the number of atranslational necklaces.
The number of atranslational necklaces is determined by characterising and counting the number of \emph{translational Lyndon words}, the set of Lyndon words that are not atranslational, i.e. the size of $\mathcal{L}_q^{\vectorise{n}} \setminus \mathcal{A}_q^{\vectorise{n}}$.
By computing the number of such necklaces, the number of atranslational necklaces can be computed by simply subtracting the size of $\mathcal{L}_q^{\vectorise{n}} \setminus \mathcal{A}_q^{\vectorise{n}}$ from the size of $\mathcal{L}_q^{\vectorise{n}}$ as $|\mathcal{A}_q^{\vectorise{n}}| = |\mathcal{L}_q^{\vectorise{n}}| - |\mathcal{L}_q^{\vectorise{n}} \setminus \mathcal{A}_q^{\vectorise{n}}|$.
% The key observation is that every necklace in $\mathcal{L}_q^{\vectorise{n}} \setminus \mathcal{A}_q^{\vectorise{n}}$ is made from some atranslational necklace of smaller size that has been used to ``tile'' the space of size $\vectorise{n}$, by repeating these smaller words across $\vectorise{n}$ with some translation each time.
% For example the translational Lyndon word {\small $\begin{bmatrix}
%     a & a & b\\
%     a & b & a\\
%     b & a & a
% \end{bmatrix}$} is made by repeating the word $aab$ along dimension $2$ under the translation $(1)$ each time.

Proposition \ref{prop:translational_symbolisation} (in the preliminaries) establishes that the structure of every translational Lyndon word $\word{w}$ is recursively defined as $\word{w} = \word{u}^p : \Angle{\word{u}^p}_g : \hdots : \Angle{\word{u}^p}_{g^{t - 1}}$ where $\word{u}$ is a Lyndon word, and $g$ is the smallest translation such that $\word{w} = \Angle{\word{w}}_g$.
For example the translational Lyndon word {\footnotesize $\begin{bmatrix}
    a & a & b\\
    a & b & a\\
    b & a & a
\end{bmatrix}$} is made by repeating the word $aab$ along dimension $2$ under the translation $(1)$ each time.
This leaves the problem of counting the number of possible such translations.
To this end, the set $\mathbf{G}(l,\vectorise{n})$ is introduced as the set of translations $g \in Z_{n_1,n_2, \hdots, n_{d - 1}}$ such that repeating $g$ $\frac{n_d}{l}$ times returns the identity operation, and that repeating $g$ less than $\frac{n_d}{l}$ returns some distinct group operation.
Note that this corresponds to the number of possible translations that can be used to transform an atranslational word of size $(n_1,n_2,\hdots, n_{d - 1}, l)$ into a necklace of size $\vectorise{n}$.
The set $\mathbf{G}(l,\vectorise{n})$ can be expressed as $\mathbf{G}(l,\vectorise{n}) = \{x \in Z_{n_1,n_2, \hdots, n_{d - 1}} | x_i^{n_d/l} \bmod n_i \equiv 0, \exists i \in [d - 1]$ such that $\forall j \in [\frac{n_d}{l} - 1], x^j_i \bmod n_i \not\equiv 0 \}$.

While $\mathbf{G}(l,\vectorise{n})$ provides the basis for converting the number of ways of repeating some atranslational word to a translational Lyndon word, it is still necessary to account for translational Lyndon words made using a Lyndon word as a basis.
% To avoid over-counting, 
The functions $I(i,l,\vectorise{n})$ and $H(i,l,\vectorise{n},d)$ are introduced as means to count the number of combinations of translations that can be used to transform some atranslational word of size $(n_1, n_2, \hdots, n_{i-1}, l)$ into a translational Lyndon word of size $\vectorise{n}$.
% Without providing a full explanation, 
Let:
{% \footnotesize
$$H(i,l,\vectorise{n},d) = \prod\limits_{j = i}^d
\begin{cases}
1 & i = d\\
(|\mathbf{G}(1, \vectorise{n})| - (I(i,l,\vectorise{n}))) \cdot (H(i, l, (n_1, n_2, \hdots, n_{d - 1}), d - 1)) & i < d
\end{cases}$$
$$I(i,l,(n_1, n_2,\hdots,n_d)) = 
\begin{cases}
0 & i = d \text{ or } l > 1\\
1 + I(i,l,(n_1, n_2, \hdots, n_{d - 1})) & n_i = n_d\\
I(i,l,(n_1, n_2, \hdots, n_{d - 1})) & n_i \neq n_d
\end{cases}$$}

\begin{theorem_L_to_A}
% \label{thm:lyndon_to_atranslational}
The number of atranslational words of size $\vectorise{n}$ over an alphabet of size $q$ is given by:

$$
|\mathcal{A}^{\vectorise{n}}_q| = |L_q^{\vectorise{n}}| - \sum\limits_{i \in [d]} \sum\limits_{l | n_i} \begin{cases}
0 & l = n_i\\
\left(\prod\limits_{t = i + 1}^{d - 1} -\mu(n_t)\right)\left(-\mu\left(\frac{n_i}{l}\right)\right) |\mathcal{A}^{n_1,n_2,\hdots,n_{d - 1},l}_q| \cdot H(i,l,\vectorise{n},d) & 1 < l < n_d
\end{cases}
$$
\end{theorem_L_to_A}

\subsection{Generation}
% \vl{I would put generation before ranking as it is easier} \duncan{I agree}
Section \ref{subsec:generation} covers the problems of generating necklaces.
String generation in lexicographic order is easy.
We find the last character which is not equal to $q$ (the largest symbol in the alphabet $\Sigma$) and increase it by 1.
Similar methods can be used for the generation of necklaces, such as the classical generation algorithm by Fredricksen and Maiorana \cite{Fredricksen1978}.
% \vl{Similar method can be actually expanded to necklaces. Caveats.}
The key tool used by both our algorithm and the algorithm for generation of 1D necklaces are \emph{prenecklaces}.
% words which are the prefix of the canonical representation of a necklace.
Formally, a prenecklace of size $\vectorise{n}$ is a word of size $\necklace{n}=
(n_1, n_2, \hdots, n_{d - 1}, n_d)$
that is the prefix of the canonical form of some necklace of size $(n_1, n_2, \hdots, n_{d - 1}, n_d + m)$ for some arbitrary $m \in \mathbb{N}$.
In other words, a word $\word{w} \in \Sigma^{\vectorise{n}}$ is a prenecklace of size $\necklace{n}$ if and only if there exists some integer $m \in \mathbb{N}$ and necklace $\necklace{v} \in \mathcal{N}_q^{n_1,n_2, \hdots, n_{d - 1}, n_d + m}$ where $(\Angle{\necklace{v}})_{[1,n_d]} = \word{w}$.
For example, the word $ababa$ is a prenecklace of size $(5)$, as it can be extended by concatenating the symbol $b$ to the end, giving the word $ababab$ which is the canonical representative of the corresponding necklace class.
However, the word $abaa$ is not a prenecklace as $aa : \word{w} : ab < abaa : \word{w}$ for every $\word{w} \in \Sigma^*$.
Note that the canonical form of every necklace is itself a prenecklace.

\begin{figure}[h]
    \centering
    \includegraphics{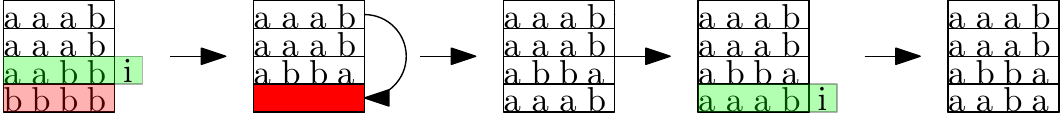}
    \caption{Example of the generation algorithm being using to move from one necklace to the next, via an intermediary prenecklace.
    In the first case, the second last slice $a a b b$ is incremented as the last slice $bbbb$ is maximal.
    On this slice has been incremented, the first slice is copied to the now vacant last position.
    As $i = 3$ does not satisfy the property that $4 \bmod i \equiv 0$, this process needs to be repeated to generate a new necklace.
    In this case, only the last slice needs to be incremented, leading to a new necklace.
    }
    \label{fig:next_prenecklace_example}
\end{figure}

The main idea behind our algorithm is to generate the set of all prenecklaces of size $\vectorise{n}$ over the alphabet $\Sigma$ in order.
By extension, this process generates each necklace in order.
The notation $\mathcal{P}_q^{\vectorise{n}}$ is used for the set of all prenecklaces of size $\vectorise{n}$ over an alphabet of size $q$.
Given a word $\word{w} \in \mathcal{P}_q^{\vectorise{n}}$, our algorithm generates the word $\word{u}$ that is subsequent to $\word{w}$ in the ordering.
This is done as follows.
Starting with $\word{w}$, the largest index $i$ such that $\word{w}_i \neq \word{Q}$ is determined, where $\word{Q}$ is the word of size $(n_1, n_2, \hdots, n_{d - 1})$ where every position in $\word{Q}$ is filled with the largest symbol $q \in \Sigma$.
% \vl{$q$ is undef}
% \duncan{Add clearer notation.}
The word $\word{u}$ is created from $\word{w}$ by first incrementing the value of the $i^{th}$ slice of $\word{w}$.
The incrimination of $\word{w}_i$ is done by either translating $\word{w}_i$, if $\word{w}_i$ has not already been translated as much as possible without returning to the canonical form $\Angle{\word{w}_i}$, 
%(formally the translation following $TR(\word{w}_i)$ in $Z_{n_1,n_2,\hdots,n_{d - 1}}$), recalling that $TR(\word{w}_i)$ returns the smallest translation $t \in Z_{n_1,n_2,\hdots,n_{d - 1}}$ where $\Angle{\Angle{\word{w}_i}}_t = \word{w}_i$
or by setting $\word{u}_i$ to $NextNecklace(\Angle{\word{w}_i})$ if no such translation exists.
% $TR(\word{w}_i) = TP(\word{w}_i)$, recalling that $TP(\word{w}_i)$ returns the largest translation $t \in Z_{n_1,n_2,\hdots,n_{d-1}}$ such that for every translation $r \in Z_{n_1,n_2,\hdots,n_{d-1}} r < t, \Angle{\word{w}_i}_t \neq \Angle{\word{w}_i}_r$.
After incrementing slice $i$, the remainder of $\word{u}$ is made by repeating the first $i$ slices.
Formally, $\word{u}_j = \word{u}_{j \bmod i}$ for every $j \in [i + 1, n_d]$.
A high level overview of this process is shown in Figure \ref{fig:next_prenecklace_example}.
It is shown that $\word{u}$ is a necklace if and only if $n_d \bmod i \equiv 0$.
Repeating prenecklace generation at most $n_d$ times guarantees that a necklace is generated.
% Theorem \ref{chp8:thm:generation_complexity} shows that our approach allows the set of multidimensional necklaces in linear time per necklace relative to the size of the necklaces.
% Notably, this corresponds to the lower bound required to output every necklace in order.

% \begin{theorem}
% \label{chp8:thm:generation_complexity}
% Let $\word{w}$ be a word of size $\vectorise{n}$.
% $NextNecklace(\word{w})$ returns the smallest word $\word{u} > \word{w}$ such that $\word{u} = \Angle{\word{u}}$ in $O(N)$ time.
% \end{theorem}

\begin{theorem4}
Let $\word{w}$ be a word of size $\vectorise{n}$.
$NextNecklace(\word{w})$ returns the smallest word $\word{u} > \word{w}$ such that $\word{u} = \Angle{\word{u}}$ in $O(N)$ time.
\end{theorem4}

\noindent
Theorem \ref{chp8:thm:generation_complexity} is proven by first showing that our algorithm generates every prenecklace in order.
This is shown in a combinatorial manner, by first providing a key characterisation of prenecklaces, then showing how to generate the subsequent prenecklace from a given prenecklace.
This generation process works in a recursive manner, with each prenecklace of size $\vectorise{n}$ requiring a $d-1$ necklace of size $(n_1,n_2,\hdots,n_{d - 1})$ to be generated.
From this characterisation, the efficiency of the generation algorithm is shown by proving that to generate the next necklace of size $\vectorise{n}$, a total of $n_d$ prenecklaces of size $\vectorise{n}$ must be generated.
The complexity comes from the recursive process.
Each prenecklace requires a necklace of size $(n_1,n_2,\hdots,n_{d - 1})$ to be generated, in turn requiring $n_{d - 1}$ prenecklaces of size $(n_1,n_2,\hdots,n_{d - 2})$ to be generated.
Repeating this recursive process yields a total of $N$ operations to generate the next necklace.
%, giving our complexity bound.

\subsection{Ranking}

% \duncan{full explanation}
Section \ref{sec:ranking} provides our algorithm for ranking multidimensional necklaces.
Recall that the rank of a word $\word{w}$ within the set of necklaces $\mathcal{N}_q^{\vectorise{n}}$ is the number of necklaces in $\mathcal{N}_q^{\vectorise{n}}$ that are smaller than $\word{w}$.
% Our ranking algorithm uses similar mechanisms to the work of Kociumaka, Radoszewski, and Rytter \cite{Kociumaka2014}.
At a high level, our ranking algorithm operates by transforming the number of words belonging to a necklace class with a canonical representation smaller than $\word{w}$ into the rank of $\word{w}$ within the set of necklaces.
This transformation is performed via the rank within the sets of atranslational and Lyndon words.
Our ranking algorithm is split between a set of theoretical tools, and a set of computational tools.
The theoretical tools establish a relationship between the number of such words
% belonging to a necklace class with a canonical representation smaller than $\word{w}$, to 
and the rank of $\word{w}$.
This motivates our computational tools that are focused on counting the number of such words.
% The main technical contribution of this algorithm is a technique to compute the size of the initial set, while the theoretical tools provide the means to transform this set to the rank within the sets of necklaces.

\noindent
\textbf{Theoretical Tools.}
For notation $T(\word{w},\vectorise{f})$ is used to denote the set of words $\word{v} \in \Sigma^{\vectorise{f}}$ where the canonical representation of the necklace class including $\word{v}$ is smaller than $\word{w}$, i.e. $T(\word{w},\vectorise{f}) = \{\word{v} \in \Sigma^{\vectorise{f}}, \Angle{\word{v}} < \word{w}\}$.
Our main computational challenge is in computing the size of $T(\word{w},\vectorise{f})$.
Treating the process of computing the size of $T(\word{w},\vectorise{f})$ as a black box for the moment, it is shown how to compute the rank of $\word{w}$ from the size of $T(\word{w},\vectorise{n})$ through two auxiliary classes of sets, the sets of aperiodic words of size $\vectorise{f} \in [\vectorise{n}]$ belonging to a necklace smaller than $\word{w}$ denoted $L(\word{w}, \vectorise{f})$, and the sets of atranslational words of size $\vectorise{f}\in [\vectorise{n}]$ belonging to a necklace smaller than $\word{w}$ denoted $A(\word{w}, \vectorise{f})$.
The sizes of $L(\word{w}, \vectorise{f})$ and $A(\word{w}, \vectorise{f})$ are determined using the same relationships established by our counting formulae in Section \ref{sec:counting}.
In terms of this notation:

{
% \footnotesize
\[
|L(\word{w},\vectorise{n})| = \sum\limits_{f_1 | n_1} \mu\left(\frac{n_1}{f_1}\right) \sum\limits_{f_2 | n_2} \mu\left(\frac{n_2}{f_2}\right) \hdots \sum\limits_{f_d | n_d} \mu\left(\frac{n_d}{f_d}\right) |T(\word{w},\vectorise{f})|
\]
\[
|A(\word{w},\vectorise{n})| = |L(\word{w},\vectorise{n})| -  \sum\limits_{i \in [d]} \sum\limits_{l | n_i} \begin{cases}
0 & l = n_i\\
\left(\prod\limits_{t = i + 1}^{d - 1} -\mu(n_t)\right)\left(-\mu\left(\frac{n_i}{l}\right)\right)|A(\word{w},n_1,\hdots,n_{i - 1},l)| \cdot H(i,l,\vectorise{n},d) & l < n_i
\end{cases}
\]
}

\noindent
Here $\mu(n)$ is the M\"{o}bius function.
%, returning $1$ if $n$ is a square free integer with an even number of factors, -1 if $n$ is a square free integer with an even number of factors, and $0$ if $n$ has a square factor.
Note that computing the sizes of $L(\word{w},\vectorise{n})$ and $A(\word{w},\vectorise{n})$ requires computing the size of $T(\word{w},\vectorise{f})$ to be computed for every $\vectorise{f}$ where $n_i \bmod f_i \equiv 0$.
By observing that every atranslational necklace in $\mathcal{A}_q^{\vectorise{n}}$ corresponds to $n_1 \cdot n_2 \cdot \hdots \cdot n_d$ words in $A(\word{w},\vectorise{n})$, the rank of $\word{w}$ within $\mathcal{A}_q^{\vectorise{n}}$ can be computed by dividing the size of $A(\word{w},\vectorise{n})$ by $n_1 \cdot n_2 \cdot \hdots \cdot n_d$.
For notation let $RA(\word{w}, \vectorise{f})$ be the rank of $\word{w}$ within the set $\mathcal{A}_q^{\vectorise{f}}$, $RL(\word{w}, \vectorise{f})$ be the rank of $\word{w}$ within the set $\mathcal{L}_q^{\vectorise{f}}$, and $RN(\word{w}, \vectorise{f})$ be the rank of $\word{w}$ within the set $\mathcal{N}_q^{\vectorise{f}}$.
Using the above observation, $RA(\word{w}, \vectorise{f})$ is given by the equation {\footnotesize $RA(\word{w}, \vectorise{f}) = \frac{|A(\word{w}, \vectorise{f})|}{f_1 \cdot f_2 \cdot \hdots \cdot f_d}$}.

Computing $RN(\word{w}, \vectorise{n})$ using $RL(\word{w}, \vectorise{f})$ and $RA(\word{w}, \vectorise{f})$ is conceptually the reverse of the process for computing the size of $|A(\word{w},\vectorise{n})|$ from using the sizes of $L(\word{w},\vectorise{f})$ and $T(\word{w},\vectorise{f})$.
Before showing how to transform $RA(\word{w},\vectorise{f})$ to $RL(\word{w},\vectorise{f})$, two helper functions are needed for the cases that $\word{w}$ is a translational, aperiodic word, i.e. $\word{w} \in \mathcal{L}_q^{\vectorise{n}}$, $\word{w} \notin \mathcal{A}_q^{\vectorise{n}}$.
Let $g \in Z_{\vectorise{n}}$ be the smallest translation such that $\word{w} = \Angle{\word{w}}_g$ and let $\word{u} \in \mathcal{A}_q^{g}$ be the translational period of $\word{w}$.
As the rank of $\word{w}$ within the set $\mathcal{A}_q^{g}$ does not count the word $\word{u}$,
% as $\word{u}$ is not strictly smaller than $\word{u}$, 
it is necessary to account for the possible translational words with a translational period of $\word{u}$, using translations that are smaller than $g$.
% Recall that $\mathbf{G}(l,\vectorise{n})$ returns the number of translations in the set.
The function $S(g,l,\vectorise{n})$ returns the number of translations in $\mathbf{G}(l,\vectorise{n})$ that are smaller than $g$.
In the case that there exists some index $i \in [d]$ such that $g_i > 1$ and $g_t = 1$ for every $t \in [i + 1, d]$ the number of possible translations corresponds to the sum of $S(r_j,l_j,(n_1,n_2,\hdots,n_j))$ for $j \in [i + 1, d]$ where $l_i = g_i$ and $l_j = 1$ for every $j > i$, and $r_j \in \mathbf{G}(l_j,(n_1,n_2,\hdots,n_j))$ is the smallest such translation for which $\word{w} = \Angle{\word{w}}_{r_j}$.
For notational convince, the function $U(\word{w})$ is defined as:%introduced as returning either:
{ %\footnotesize
\[
U(\word{w}) = \begin{cases}
0 & \text{$\word{w}$ is either atranslational or periodic}\\
\sum\limits_{j = i}^d \begin{cases}
        S(r_j, l, (n_1, n_2, \hdots, n_j)) & j = i\\
        S(r_j, 1, (n_1, n_2, \hdots, n_j)) & j > i
\end{cases} & \text{$\word{w}$ is a Lyndon word with a translational period of $g$}
\end{cases}
\]
}

% \begin{itemize}
%     \item $0$ if $\word{w}$ is either atranslational or periodic.
%     \item $\sum\limits_{j = i}^d \begin{cases}
%         S(r_j, l, (n_1, n_2, \hdots, n_j)) & j = i\\
%         S(r_j, 1, (n_1, n_2, \hdots, n_j)) & otherwise.
%     \end{cases}$ if $\word{w}$ is a Lyndon word with a translational period of $g$.
% \end{itemize}

\noindent
The rank $\word{w}$ within the set $\mathcal{L}_q^{\vectorise{n}}$ can be expressed as the sum of $RA(\word{w},\vectorise{n})$, $U(\word{w})$ and $C(\word{w},\vectorise{n})$, where $C(\word{w},\vectorise{n})$ is the sum {\footnotesize $\sum\limits_{i \in [d]} \sum\limits_{l | n_i} \begin{cases}
0 & l = n_i\\
\left(\prod\limits_{t = i + 1}^{d - 1} -\mu(n_t)\right)\left(-\mu\left(\frac{n_i}{l}\right)\right) |RA(\word{w}_{[1,l]},n_1,n_2,\hdots,n_{i - 1},l)| \cdot H(i,l,\vectorise{n},d) & 1 < l < n_d
\end{cases}$}.

% Using $U(\word{w})$, the rank of $\word{w}$ within the set $\mathcal{L}_q^{\vectorise{n}}$ can be computed from $RA(\word{w},n_1,n_2,\hdots,n_{i - 1},l)$ for every $i \in [d]$ and factor $l$ of $n_i$ as:
% $RL(\word{w},\vectorise{n}) = RA(\word{w},\vectorise{n}) + U(\word{w})+
% %\begin{align*}
% %&RL(\word{w},\vectorise{n}) = RA(\word{w},\vectorise{n}) + U(\word{w}) + 
% %\\
% &\sum\limits_{i \in [d]} \sum\limits_{l | n_i} \begin{cases}
% 0 & l = n_i\\
% \left(\prod\limits_{t = i + 1}^{d - 1} -\mu(n_t)\right)\left(-\mu\left(\frac{n_i}{l}\right)\right) |RA(\word{w}_{[1,l]},n_1,n_2,\hdots,n_{i - 1},l)| \cdot H(i,l,\vectorise{n},d) & 1 < l < n_d
% \end{cases}
% $%\end{align*}

% \noindent
The rank of $\word{w}$ within the set $\mathcal{N}_q^{\vectorise{n}}$ can be computed by taking the sum over $RL(\word{w},\vectorise{f})$ for every $\vectorise{f}$ such that for all $i \in [d]$, $f_i$ is a factor of $n_i$.
% Using $RL(\word{w},\vectorise{f})$, it is possible to compute the rank of $\word{w}$ within the set $\mathcal{N}_q^{\vectorise{n}}$.
The key observation is that every necklace in $\mathcal{N}_q^{\vectorise{n}}$ must have a period of size $\vectorise{f}$ where $n_i \bmod f_i \equiv 0$ for every $i \in [d]$.
Further, the number of necklaces with a period of size $\vectorise{f}$ smaller than $\word{w}$ is equivalent to $RL(\word{w},\vectorise{f})$.
Therefore, the number of necklaces smaller than $\word{w}$ can be computed by summing the number of Lyndon words smaller than $\word{w}$ for every such $\vectorise{f}$.
% Therefore by ranking $\word{w}$ with $RL(\word{w},\vectorise{f})$ for every such value of $\vectorise{f}$, $RN(\word{w})$ can be computed as:
Hence the rank of $\word{w}$ within the set $\mathcal{N}_q^{\vectorise{n}}$ is given by the equation {\footnotesize $RN(\word{w},\vectorise{n}) = \sum\limits_{f_1 | n_1} \sum\limits_{f_2 |n_2} \hdots \sum\limits_{f_d | n_d} RL(\word{w},\vectorise{f})$}.

\noindent
\textbf{Computational Tools.}
The theoretical tools above show how to transform the size of the sets $T(\word{w},\vectorise{f})$ to the rank of $\word{w}$ among necklaces, and by extension Lyndon words and atranslational necklaces.
This leaves the problem of computing the size of $T(\word{w},\vectorise{f})$.
The size of $T(\word{w},\vectorise{f})$ is computed by partitioning $T(\word{w},\vectorise{f})$ into $f_d^2$ subsets, denoted $\mathbf{B}(\word{w}, g_d, j, \vectorise{f})$ where $\mathbf{B}(\word{w}, g_d, j, \vectorise{f})$ contains every word $\word{v} \in \Sigma^{\vectorise{f}}$ where:
\begin{itemize}
    \item The smallest translation $t \in Z_{\vectorise{f}}$ such that $\word{w} > \Angle{\word{v}}_t$ is of the form $t = (t_1, t_2, \hdots, t_{d - 1}, g_d)$.
    \item $j$ is the length of the longest shared prefix of both $\word{w}$ and $\Angle{\word{v}}_t$, i.e. $\word{w}_{[1,j]} = (\Angle{\word{v}}_t)_{[1,j]}$.
\end{itemize}

\noindent
The size of $T(\word{w},\vectorise{f})$ is given by {\footnotesize $\sum\limits_{j \in [f_d]} \sum\limits_{g_d \in [f_d]} |\mathbf{B}(\word{w}, g_d, j, \vectorise{f})|$}.
The size of $\mathbf{B}(\word{w}, g_d, j, \vectorise{f})$ is computed based on two cases determined by the values of $j$ and $g_d$.
In the case that $j + g_d < f_d$:
\[
|\mathbf{B}(\word{w}, g_d, j, \vectorise{f})| = |\beta(\word{w},g_d,0,\vectorise{f})| \cdot (q^{f_1\cdot f_2\cdot \hdots \cdot f_{d - 1}} - |\beta(\word{w}_{j + 1},1,0,\vectorise{f})| - 1) \cdot |\mathbf{\Theta}|\cdot q^{f_1\cdot f_2\cdot \hdots\cdot f_{d - 1} \cdot (f_d - (g_d + j + 1))}
\]

\noindent
And in the case that $j + g_d \geq f_d$, the size of $\mathbf{B}(\word{w}, g_d, j, \vectorise{f})$ is given by the equation:
\[
|\mathbf{B}(\word{w}, g_d, j, \vectorise{f})| = |\beta(\word{w},f_d + t - j,t + 1, \vectorise{f})| + \left(|\beta(\word{w}_{t + 1},1,0,\vectorise{f})| - |\beta(\word{w}_{j + 1},1,0,\vectorise{f})|\right)\cdot |\beta(\word{w},f_d- j - 1,0,\vectorise{f})|\cdot|\mathbf{\Theta}|
\]

\noindent
Where $\beta(\word{w},g_d,j,\vectorise{f})$ is a set containing every word $\word{u} \in \Sigma^{f_1, f_2,\hdots,f_{d - 1}, g_d}$ with the properties that $\word{w}_{[1,j]} = \word{u}_{[1,j]}$ and that every suffix of $\word{u}$ under any translation from $Z_{(f_1,f_2,\hdots, f_{d - 1})}$ is strictly greater than $\word{w}$, i.e. for every $i \in [g_d]$ and $h \in Z_{(f_1,f_2,\hdots, f_{d - 1})}$, $\word{w}_{[1,i]} < \Angle{\word{u}_{[1,i]}}_h$.
Further the set $\mathbf{\Theta}$ contains the set of unique translations of $\word{w}_{[1,j]}$, i.e. the set $\{r \in Z_{\vectorise{f}} : \nexists s \in Z_{\vectorise{f}}$ where $s < r, \Angle{\word{w}}_r = \Angle{\word{w}}_s\}$.
This leaves the problem of computing the size of $\beta(\word{w},g_d,j,\vectorise{f})$.
Observe that when $i = j$ then either the size of $\beta(\word{w},g_d,j,\vectorise{f})$ is $1$, corresponding to the empty word in the case that $i = j = 0$, or $0$ when $i > 0$ due to the suffix of length $j$ every word of $\beta(\word{w},g_d,j,\vectorise{f})$ in this case being equal to $\word{w}_{[1,j]}$.
Otherwise when $i \neq j$, as every suffix of $\word{u} \in \beta(\word{w},g_d,j,\vectorise{f})$ must belong to $\beta(\word{w},g_d',j',\vectorise{f})$ for some $g_d' \leq g_d$ and $j' \in \{0, j + 1\}$. Therefore the size of $\beta(\word{w},g_d,j,\vectorise{f})$ can be computed in a recursive manner.
Explicitly, the size of $\beta(\word{w},g_d,j,\vectorise{f})$ equals:
% through the recursive formula:
\vspace{-0.2cm}
{ % \footnotesize
\[
|\beta(\word{w},g_d,j,\vectorise{f})| = \begin{cases}
    0 & g_d = j, j > 0\\
    1 & g_d = j = 0\\
    NS(\word{w},j,\vectorise{f}) \cdot |\beta(\word{w},g_d-j-1,0,\vectorise{f})| + |\beta(\word{w},g_d, j + 1,\vectorise{f})| & Otherwise.
\end{cases}
\]
}

\noindent
Where $NS(\word{w},j,\vectorise{f})$ returns the number of slices greater than $\word{w}_j$, defined as:

{ %\footnotesize
\[
NS(\word{w},j,\vectorise{f}) = (TP(\word{w}_{j + 1}) - TR(\word{w}_{j + 1})) +  \sum\limits_{i \in [d - 1]} \sum\limits_{h_i | f_i} RA(\word{w}_{j}, \vectorise{h[i]})) \cdot |\vectorise{h[i]}| \cdot H(i,h,\vectorise{f},d)
\]
}
\begin{theorem2}
The rank of a $d$-dimensional necklace in the set $\mathcal{N}_q^{\vectorise{n}}$ can be computed in \RankingComplexity time, where $N = \prod_{i = 1}^d n_i$.
\end{theorem2}

% \begin{theorem}
% \label{thm:ranking_complexity}
% The rank of a $d$-dimensional necklace in the set $\mathcal{N}_q^{\vectorise{n}}$ can be computed in \RankingComplexity time, where $N = \prod_{i = 1}^d n_i$.
% \end{theorem}

\noindent
% The complexity and correctness of this algorithm is provided in Section \ref{sec:ranking}.
The correctness of this algorithm is shown by first establishing the relationships between the classes of $T(\word{w},\vectorise{f}),L(\word{w},\vectorise{f})$, and $A(\word{w},\vectorise{f})$, and the rank of $\word{w}$ within the sets $\mathcal{N}_q^{\vectorise{f}},\mathcal{L}_q^{\vectorise{f}}$ and $\mathcal{A}_q^{\vectorise{f}}$.
The complexity largely comes from the recursive nature of the algorithm.
In general, the cost of determining the size of $T(\word{w},\vectorise{f})$ for every $\vectorise{f} \in \{(x_1, x_2, \hdots, x_d) \in [\vectorise{n}]: \forall i \in [d], n_i \bmod x_i \equiv 0\}$.
Each of these sets requires a set of $n_d^2$ values of $\mathbf{B}(\word{w}, g_d, j, \vectorise{f})$, in turn requiring $n_d$ words of size $(n_1, n_2, \hdots, n_{d - 1})$ to be ranked.
As this must be repeated for each dimension, the function $NS$ must be called a total of $O(N^3)$ times.
%complexity.
The additional factor of $O(N^2)$ is due to the cost of evaluating $NS$, requiring $O(n_1 + n_2 + \hdots + n_d) \approx O(N)$ calls to $H(i,h,\vectorise{f},d)$, itself requiring $O(N)$ time to evaluate.
%and of computing $RA(\word{w},)$ partitioning $T(\word{w},\vectorise{n})$ in each dimension.
% The total complexity is determined by repeating these arguments for each dimension.

These results are extended to the fixed content case, where every necklace shares the same Parikh vector $\vectorise{p}$.
% This is particularly noteworthy as no ranking algorithm exists for fixed content necklaces in the 1D case for any alphabet of size $q > 2$.
The same theoretical tools as unconstrained necklaces are used in the fixed content case.
The main difference between these settings, accounting for the increased complexity in the fixed content case, comes from the computational tools.
Primarily, when computing the size of $\mathbf{B}(\word{w}, g_d, j, \vectorise{f})$, it is necessary to subdivide the set $\mathbf{B}(\word{w}, g_d, j, \vectorise{f})$ based on the Parikh vector of the prefixes.
This results in an exponential cost in the size of the alphabet from the $O(N^q)$ potential number of prefix Parikh vectors.

\begin{thm_fixed_content_ranking}
The rank of a $d$-dimensional necklace in the set $\mathcal{N}_{\vectorise{p}}^{\vectorise{n}}$ can be computed in $O(N^{6 + q})$ time, where $N = \prod_{i = 1}^d n_i$ and $\vectorise{p}$ is some given Parikh vector of length $q$.
\end{thm_fixed_content_ranking}

\subsection{Unranking}

% Section \ref{chap8:subsec:unranking} provides algorithms for unranking both necklaces, and fixed content necklaces.
Recall that the unranking problem asks for the necklace $\necklace{w}$ with rank $i$ within set $\mathcal{N}_q^{\vectorise{n}}$.
Let $\word{w} = \Angle{\necklace{w}}$.
% It is assumed that the word $\word{w}$ is the canonical representation of $\necklace{w}$.
Our unranking algorithm works by iteratively determining the prefix of $\word{w}$, starting with the empty word.
At the $j^{th}$ step of the unranking process, the prefix of $\Angle{\necklace{w}}$ of length $j$ has been determined, with the goal being to determine the value of $\word{w}_{j + 1}$.
The value of $\word{w}_{j + 1}$ is determined by searching the space of $d - 1$ dimensional necklaces for the necklace $\necklace{u}$ such that $\word{w}_{j + 1} \in \necklace{u}$.
The value of $\necklace{u}$ is determined using the words $\word{a},\word{b} \in \necklace{u}$, where $\word{a}$ is the canonical representation of $\necklace{u}$, and $\word{b}$ is the largest word in $\necklace{u}$.
Two words $\word{A}, \word{B} \in \Sigma^{\vectorise{n}}$ are generated where $\word{A}$ is the smallest necklace with the prefix $\word{w}_{[1,j]} : \word{a}$, and $\word{B}$ the largest necklace with the prefix $\word{w}_{[1,j]} : \word{b}$.
Using the $NextNecklace$ algorithm given in Section \ref{subsec:generation}, it is possible to find each word in $O(N)$ time.
Observe that $\word{w}_{j + 1}$ belongs to necklace class $\necklace{u}$ if and only if $RN(\word{A},\vectorise{n}) \leq i \leq RN(\word{B},\vectorise{n})$.
Using this as a basis, a binary search is performed over the set $\mathcal{N}_q^{n_1,n_2,\hdots,n_{d - 1}}$ to determine the necklace class of $\word{w}_{j + 1}$, starting with the necklace with rank {\footnotesize $\frac{|\mathcal{N}_q^{n_1,n_2,\hdots,n_{d - 1}}|}{2}$}, and navigating through the set based on the value of $i$ relative to the ranks of $\word{A}$ and $\word{B}$ for each necklace.
This requires a $d - 1$-dimensional necklace to be unranked at each step.
% Note that the same approach can be applied to the fixed content case, substituting the appropriate ranking algorithm.

% $\word{A}_i = (\word{w}_{[1,j]} : \word{a})_{i \bmod j + 1}$ and $\word{B} = \word{w}_{[1,j]} : \word{b} : \word{q}$
% Our unranking approach is based on performing binary search to determine the symbol at each position in the necklace in order.
% Our algorithm works by determining the rank of both the smallest and largest necklaces with a given prefix, using the ranking process to do so.\vl{explanation of smallest and largest continuation}
% Theorem \ref{chp8:thm:unranking} compliments both Theorems \ref{thm:ranking_complexity} and \ref{chp8:thm:generation_complexity} by showing how to unrank multidimensional necklaces in polynomial time.
% Corollary \ref{chp8:col:fixed_content_unranking} extends this result to the fixed content setting, showing that a fixed content necklace can be unranked in $O\left(N^{(q + 7)(d + 1)} \log^d(q)\right)$.

\begin{theorem3}
The $i^{th}$ necklace in $\mathcal{N}_q^{\vectorise{n}}$ can be generated (unranked) in $O\left(N^{6(d + 1)} \cdot \log^d(q)\right)$ time.
\end{theorem3}
%
% \begin{theorem}
% \label{chp8:thm:unranking}
% The $i^{th}$ necklace in $\mathcal{N}_q^{\vectorise{n}}$ can generated (unranked) in $O\left(N^{6(d + 1)} \cdot \log^d(q)\right)$ time.
% \end{theorem}
%
%
\vspace{-0.4cm}
\begin{col_fixed_content_unranking}
The $i^{th}$ necklace in $\mathcal{N}_{\vectorise{p}}^\vectorise{n}$ can be generated (unranked) in $O(N^{(q + 7)(d + 1)}) \log^d(q)$ time.
\end{col_fixed_content_unranking}
%
% \begin{corollary}
% \label{chp8:col:fixed_content_unranking}
% The $i^{th}$ necklace in $\mathcal{N}_{\vectorise{p}}^\vectorise{n}$ can be unranked in $O(N^{(q + 7)(d + 1)}) \log^d(q)$ time. 
% \end{corollary}
%
\noindent
The main complexity of this algorithm comes from the recursive process.
Observe that to compute the value of $\word{w}_{j + 1}$, it is necessary to rank at most $\log(|\mathcal{N}_q^{n_1,n_2,\hdots,n_{d - 1}}|) \approx N$ necklaces, each requiring a necklace in the set $\mathcal{N}_q^{n_1,n_2,\hdots,n_{d - 1}}$ to be unranked.
A full proof of both is provided in Section \ref{chap8:subsec:unranking}.
% This recursion leads to the factor of $d$ within the exponent, while the repeated ranking process gives the factor of $6$

\subsection{The \texorpdfstring{$k$}{k}-centre problem}

\noindent
The last operation this paper presents for the set of multidimensional necklaces is that of the $k$-centre problem.
%This operation is not only new to the set of multidimensional necklaces, but we believe the first definition of the $k$-centre problem for necklaces. 
The difficulty is that we need to
select equally spaced centres in implicitly represented sets of objects.
For our setting, the $k$-centre problem for a set of necklaces $\mathcal{N}_q^{\necklace{n}}$ ({\it Problem~\ref{prob:k_sample}}) asks for a set $\mathbf{S}$ of $k$ necklaces minimising the objective function $\max_{\necklace{w} \in \mathcal{N}_q^{\necklace{n}}} \left(\min_{\necklace{u} \in \mathbf{S}} dist(\necklace{w}, \necklace{u})\right)$, where $dist(\necklace{w}, \necklace{u})$ is some distance function.
% The first challenge for this work was to choose an appropriate distance function.
% In order to construct a graph from the set of necklaces for the $k$-centre problem, it is necessary to determine a similarity metric for necklaces.
Motivated by the problem of choosing a distinct set of crystals, we use subwords as a notion of similarity for defining the distance function.
% The idea behind this approach is that local interactions between ions have much higher energy than long range interactions, and therefore are much more impactful on the overall structure of a crystal.
% Therefore, by choosing a set of necklaces with a diverse set of subwords, the corresponding unit cells should provide a good sample from which to find the optimal solution.
To this end, we turn to the {\sl the overlap coefficient} \cite{cohen2003comparison, piskorski2007comparison, recchia2013comparison}.
Informally, the overlap coefficient measures of the number of common subwords between necklaces, normalised by the total number of subwords in each necklace.
We use $\mathfrak{O}(\necklace{s},\necklace{v})$ to denote our overlap based distance between two necklaces.

The graph in this setting corresponds to the set of all necklaces of some given length $n$ over an alphabet $\Sigma$ of size $q$.
This setting has some unique properties.
While the graph can be completely represented, it is of exponential size relative to the description in terms of $n$ and $q$.
Despite this, the graph has a highly symmetric structures due to the nature of necklaces.
We show that verifying a solution to the $k$-centre problem for necklaces can not be done in polynomial time relative to $n$ and $q$ unless $P = NP$, indicating that the $k$-centre problem itself is likely to be at least NP-hard.

\newtheorem*{thm_sampling_np}{Theorem \ref{thm:sampling_descion_np}}

\begin{thm_sampling_np}
Given a set of $k$ necklaces $\mathbf{S} \in \mathcal{N}_q^{\vectorise{n}}$ and a distance $\ell$, it is NP-hard to determine if there exists some necklace $\necklace{v} \in \mathcal{N}_q^\vectorise{n}$ such that $\mathfrak{O}(\necklace{s},\necklace{v}) > \ell$ for every $\necklace{s} \in \mathbf{S}$ for any dimension $d$.
\end{thm_sampling_np}

\noindent
Despite this challenge, we provide two approximation algorithms for solving the $k$-centre problem on necklaces, both using \emph{de-Bruijn sequences} as a basis.
A de-Bruijn sequences  of order $n$ over the alphabet $\Sigma$ is a cyclic word of length $q^n$ containing every word in $\Sigma^n$ exactly once \cite{Annexstein1997}.
In 1D our algorithm splits such a sequence into a set of $k$ centres, requiring some overlap between centres to preserve the property that every word in $\Sigma^n$ appears in some centre at least once.
The disconnect between our algorithm and the derived theoretical lower bound is due to some words in $\Sigma^n$ appearing more than once in our set of centres.
Figure \ref{fig:deBrujinExample} provides a sketch of the process of dividing such a sequence between a set of $k$ centres.
\vspace{-0.4cm}
\begin{figure}[h]
    \begin{tabular}{l|l}
        Sequence: & 0000001000011000101000111001001011001101001111010101110110111111\\
        Centre & Word \\
        1 & {\color{red} 000000}1000011000{\color{blue} 10100}\\
        2 & \hspace{3.1cm}{\color{blue} 10100}01110010010{\color{darkgreen} 11001}\\
        3 & \hspace{6.2cm}{\color{darkgreen} 11001}10100111101{\color{purple}01011}\\
        4 & {\color{red}000000}\hspace{8.1cm}{\color{purple}01011}10110111111
    \end{tabular}
    \caption{Example of how to split the de Bruijn sequence of order 6 between 4 centres.
    Highlighted parts are the shared subwords between two centres.
    }
    \label{fig:deBrujinExample}
\end{figure}
\vspace{-0.4cm}
\begin{thm_db1} 
% \label{thm:de_bruijn_1}
The $k$-centre problem for $\mathcal{N}_q^n$
% Problem~\ref{prob:k_sample} in 1D
can be approximated in $O(n \cdot k)$ time with an approximation factor of $1 + f(n,k)$ where $f(n,k) = \frac{\log_q{(k \cdot n)}}{n-\log_q{(k \cdot n)}}-\frac{\log^2_q(k\cdot n)}{2n(n-\log_q{(k\cdot n))}}$ and $f(n,k) \rightarrow 0$ for $n \rightarrow \infty$.
\end{thm_db1}

% \begin{theorem} 
% \label{thm:de_bruijn_1}
% Problem~\ref{prob:k_sample} in 1D can be approximated in $O(n \cdot k)$ time with an approximation factor of $1 + f(n,k)$ where $f(n,k) = \frac{\log_q{(k \cdot n)}}{n-\log_q{(k \cdot n)}}-\frac{\log^2_q(k\cdot n)}{2n(n-\log_q{(k\cdot n))}}$ and $f(n,k) \rightarrow 0$ for $n \rightarrow \infty$.
% \end{theorem}

\noindent
Our second algorithm extends this to the multidimensional setting.
% \igor{the idea of covering some approximate multidimensional de Bruijn sequence is missing here. Can you say something about it here? }
At a high level, the idea of this algorithm is to construct an approximation of the \emph{de Bruijn torus}, the multidimensional equivalent of a de Bruijn sequence, splitting this word between centres analogously to how the de Bruijn sequence is partitioned in the 1D setting.
This is achieved by taking the alphabet $\Sigma$ and constructing a new alphabet $\Sigma'$ where each symbol in $\Sigma'$ corresponds to a word in $\Sigma^{\vectorise{f}}$ for some size vector $\vectorise{f}$, where we assume $n_i \bmod f_i \equiv 0$ and $f_d = 1$.
Any de Bruijn sequence of order $t$ on this new alphabet can be converted to a word $\word{w}$ of size $(f_1, f_2, \hdots,f_{d - 1}, q^{F \cdot t})$ where $F = f_1 \hdots f_2 \cdot \hdots \cdot f_{d - 1}$, with the property that $\word{w}$ contains every word in $\Sigma^{f_1,f_2,\hdots,f_{d - 1},t}$ as a subword at least once.

This word $\word{w}$ is converted into a set of $k$ centres by first partitioning $\word{w}$ into a set $\mathbf{s}$ of {\footnotesize $k' = k \cdot \frac{n_1 \cdot n_2 \cdot \hdots \cdot n_{d - 1}}{f_1 \cdot f_2 \cdot \hdots \cdot f_{d - 1}}$} centres of size $(f_1,f_2,\hdots,f_{d - 1}, n_d)$ in the same manner as in the 1D case.
%reduce the problem of finding a set of $k$-centres in the multidimensional setting to the problem %of finding a set of $k'$ centres in the 1D setting.
%This is achieved by taking the alphabet $\Sigma$ and constructing a new alphabet $\Sigma'$ where each symbol in $\Sigma'$ corresponds to a word in $\Sigma^{\vectorise{f}}$ for some size vector $\vectorise{f}$ - where we assume $n_i \bmod f_i \equiv 0$ and $f_d = 1$.
%In this new setting, rather than constructing $k$ samples of size $\left(\frac{n_1}{f_1}, \frac{n_2}{f_2}, \hdots, \frac{n_{d - 1}}{f_{d - 1}}, n_d \right)$, we instead construct a set $\mathbf{s}$ of $k' = k \cdot \frac{n_1 \cdot n_2 \cdot \hdots \cdot n_{d - 1}}{f_1 \cdot f_2 \cdot \hdots \cdot f_{d - 1}}$ 1D centres of length $n_d$.
% These centres are constructed from $\Sigma'$ in the same manner as in the 1D algorithm.
Once each centre has been generated, a set of $k$ centres can be constructed by partitioning $\mathbf{s}$ into $k$ arbitrary disjoint subsets.
Each subset is made into a necklace by concatenating the constitute words into a word of the appropriate size.
%giving a set of $k$ samples for alphabet $\Sigma'$ which may in turn be turned into a set of samples for the original alphabet $\Sigma$ by replacing each symbol in $\Sigma'$ with the corresponding word from $\Sigma^{\vectorise{f}}$.

\begin{thm_alg_3}
% \label{thm:alg_3}
The $k$-centre problem for $\mathcal{N}_q^{\vectorise{n}}$
% Problem \ref{prob:k_sample}
can be approximated in $O(N^2 k)$ time within an approximation factor of $1 +  \frac{
\log_q{(k N)}}{N-\log_q{(kN)}}-\frac{\log^2_q(kN)}{2N(N-\log_q{(kN))}}$, where $N = \prod_{i = 1}^d n_i$.
\end{thm_alg_3}

\section{Counting of Multidimensional Necklaces}
\label{sec:counting}

This section provides a comprehensive overview of the equations for counting the number of necklaces, Lyndon words, and atranslational necklaces.
For both necklaces and Lyndon words, explicit counting is done by application of the P\'{o}lya enumeration theorem to the group operations defined in Section \ref{sec:prelims}.
Equations \ref{eq:necklace_1d} and \ref{eq:lyndon_1d} are classical formulas for counting the number of 1D necklaces and 1D Lyndon words respectively.
A classical proof for the Necklace Equation is provided by Graham, Knuth and Patashnik \cite{Graham1994}, while Perrin \cite{perrin_1997} provides a proof of the Lyndon word Equation.

\begin{align}
|\mathcal{N}_q^n| = \frac{1}{n} \sum\limits_{d|n} \phi\left(\frac{n}{d}\right) q^{d}.\label{eq:necklace_1d}
\end{align}

\begin{align}
|\mathcal{L}_q^n| = \frac{1}{n}\sum\limits_{d|n} \mu\left(\frac{n}{d}\right) q^{d}\label{eq:lyndon_1d}.
\end{align}

\noindent
In Equation \ref{eq:necklace_1d} $\phi(n)$ is Euler's totient function and in Equation \ref{eq:lyndon_1d} $\mu(n)$ is the M\"{o}bius function.
Formally, $\phi(n)$ gives the number of natural numbers smaller than $n$ which are co-prime to $n$, and $\mu(n)$ returns $1$ if $n$ is a square free integer with an even number of factors, -1 if $n$ is a square free integer with an even number of factors, and $0$ if $n$ has a square factor.
These equations form the starting point for counting multidimensional necklaces.

\begin{theorem}
\label{thm:necklace_counting}
The number of necklaces of size $\vectorise{n}$ over an alphabet of size $q$ is given by the equation:
$$|\mathcal{N}_q^{\vectorise{n}}| = \frac{1}{N} \sum\limits_{f_1 | n_1} \phi\left(f_1\right) \sum\limits_{f_2 | n_2} \phi\left(f_2\right) \hdots \sum\limits_{f_d | n_d} \phi\left(f_d \right) q^{(N/\lcm{(f_1, f_2, \hdots, f_d)})}$$
Where $N = \prod_{i = 1}^d$ and $\phi(x)$ is Euler's totient function.
\end{theorem}

% \begin{thm_necklace_counting}
% The number of necklaces of size $\vectorise{n}$ over an alphabet of size $q$ is given by the equation:

% $$|\mathcal{N}_q^{\vectorise{n}}| = \frac{1}{N} \sum\limits_{f_1 | n_1} \phi\left(f_1\right) \sum\limits_{f_2 | n_2} \phi\left(f_2\right) \hdots \sum\limits_{f_d | n_d} \phi\left(f_d \right) q^{(N/\lcm{(f_1, f_2, \hdots, f_d)})}$$

% Where $N = n_1 \cdot n_2 \cdot \hdots \cdot n_d$ and $\phi(x)$ is Euler's totient function.
% \end{thm_necklace_counting}

\begin{proof}
Recall from the preliminaries that multidimensional necklaces of size $\vectorise{n}$ are equivalence classes of words in $\Sigma^{\vectorise{n}}$ under the group $Z_{\vectorise{n}} = Z_{n_1} \times Z_{n_2} \times \hdots \times Z_{n_d}$ where $\times$ denotes the direct product and $Z_x$ the cyclic group of order $x$.
A straightforward way to compute the number of necklaces of size $\vectorise{n}$ is by using the \emph{P\'{o}lya enumeration formula}, giving:

\[
|\mathcal{N}_q^{\vectorise{n}}| = \frac{1}{N}\sum\limits_{g \in Z_{\vectorise{n}}} q^{c(g)}.
\]
Where $g = (g_1, g_2, \hdots, g_d)$ is some group action in $Z_{\vectorise{n}}$ and $c(g)$ returns the number of cycles from the group action $g$.
Since $Z_{\vectorise{n}}$ is formed by the direct product of the cyclic groups, for each group action $g = (g_1, g_2, \hdots, g_d)$, where $1 \leq i_j \leq n_j$.
Therefore, the number of necklaces, $|\mathcal{N}_q^{\vectorise{n}}|$, is rewritten as:

\[
|\mathcal{N}_q^{\vectorise{n}}| = \frac{1}{N} \sum\limits_{g_1 = 1}^{n_1} \sum\limits_{g_2 = 1}^{n_2} \hdots \sum\limits_{g_d = 1}^{n_d} q^{c((g_1, g_2, \hdots g_d))}
\]

\noindent
In order to determine the value of $c(g)$, consider the permutation induced by $g$.
Given some position $\mathbf{j} = (j_1, \hdots, j_d)$, let $\mathbf{j}'$ be the position following $\mathbf{j}$ in the cycle induced by $g$, i.e. $\mathbf{j}' = \mathbf{j} + g$.
The coordinate of $\mathbf{j}'$ in the $i^{th}$ dimension is equal to the coordinate in the $i^{th}$ dimension of $\mathbf{j}$ shifted by $g_i$.
Since this is a cyclic operation, this shift is done modulo the length of dimension $i$, $n_i$.
This gives $j'_i = (j_i + g_i)\bmod n_i$.

Let $g^t$ denote the group action made by applying $t$ times operation $g$ to the identity operation $I = (0,0,\hdots,0)$, i.e. $I + g + g + \hdots + g$.
The length of the cycle induced by some cyclic shift $g$ is the smallest value $t > 0$ such that $\mathbf{j} + g^t = \mathbf{j}$.
In other words, the length of the cycle equals the number of times $g$ must be applied to itself to become the identity operation.
The length of this cycle is therefore the smallest $t$ such that for every $i$, $(\mathbf{j}_i + t \cdot g_i) \bmod n_i \equiv \mathbf{j}_i$.
To compute this, note that $t$ must be divisible by the smallest value $l_i$ for each dimension such that $(\mathbf{j}_i + l_i \cdot g_i) \bmod n_i \equiv \mathbf{j}_i$.
As such, the smallest value $t$ may have is the least common multiple of every $l_i$.
For any smaller non-zero value, there is some dimension $i$ for which $(\mathbf{j}_i + t \cdot g_i) \bmod n_i \not\equiv \mathbf{j}_i$.
By the properties of modular addition, it is clear that every cycle has the same length.
Therefore, the number of cycles of length $t$ is $\frac{N}{t}$.

This is rewritten as follows.
Observe that the only possible values for $l_i$ are divisors of $n_i$.
For each divisor $f_i$ of $n_i$, there are $\phi(\frac{n_i}{f_i})$ values for which $f_i = l_i$.
As this is independent in each dimension, this is used to derive the following equation for the number of necklaces:

\begin{align*}
    |\mathcal{N}_q^{\vectorise{n}}| = \frac{1}{N} \sum\limits_{f_1 | n_1} \phi\left(\frac{n_1}{f_1}\right) \sum\limits_{f_2 | n_2} \phi\left(\frac{n_2}{f_2}\right) \hdots \sum\limits_{f_d | n_d} \phi\left(\frac{n_d}{f_d}\right) q^{\frac{N}{\lcm{(f_1, f_2, \hdots, f_d)}}}. \label{eq:multidimensional_necklaces}
\end{align*}
\end{proof}

\noindent
Using the set of necklaces as a basis, the next goal is to count the number of Lyndon words.

\begin{theorem}
\label{thm:Lyndon_counting}
The number of Lyndon words of size $\vectorise{n}$ over an alphabet of size $q$ is given by the equation:
$$|\mathcal{L}_q^{\vectorise{n}}| = \sum\limits_{f_1 | n_1} \mu\left(\frac{n_1}{f_1}\right) \sum\limits_{f_2 | n_2}\mu\left(\frac{n_2}{f_2}\right)\hdots \sum\limits_{f_d | n_d} \mu\left(\frac{n_d}{f_d}\right) |\mathcal{N}_q^{f_1,f_2\hdots f_d}|$$
Where $\mu(x)$ is the M\"{o}bius function.
\end{theorem}

% \begin{thm_lynodn_counting}
% The number of Lyndon words of size $\vectorise{n}$ over an alphabet of size $q$ is given by the equation:

% \[
% |\mathcal{L}_q^{\vectorise{n}}| = \sum\limits_{f_1 | n_1} \mu\left(\frac{n_1}{f_1}\right) \sum\limits_{f_2 | n_2}\mu\left(\frac{n_2}{f_2}\right)\hdots \sum\limits_{f_d | n_d} \mu\left(\frac{n_d}{f_d}\right) |\mathcal{N}_q^{f_1,f_2\hdots f_d}|
% \]

% Where $\mu$ is the M\"{o}bius function.
% \end{thm_lynodn_counting}

\begin{proof}
In order to derive this algorithm, it is useful to first rewrite the number of necklaces in terms of Lyndon words.
Consider the set of necklaces in $\mathcal{N}_q^{\vectorise{n}}$ with a period of size $\vectorise{f}$.
Note that the period of each necklace corresponds to the minimal word under the translation operation for the necklace class corresponding to the period.
More explicitly, given a necklace $\necklace{w} \in \mathcal{N}_q^{\vectorise{n}}$ with a period $\word{u}$, for $\word{u}^{\vectorise{t}}$ to be the canonical form of $\necklace{w}$, $\word{u}$ must be the canonical form of $\Angle{\word{u}}$, as otherwise there would be some translation of $\word{u}^{\vectorise{t}}$ that is smaller than $\word{u}^{\vectorise{t}}$.
Therefore, the number of necklaces with a period of size $\vectorise{f}$ directly corresponds to the number of Lyndon words of size $\vectorise{f}$.
% Observe that the number of such necklaces must equal $|\mathcal{L}_q^{\vectorise{f}}|$ as this corresponds to the number of words that are minimal under the translational
Further, any necklace with a period in $\mathcal{L}_q^{\vectorise{f}}$ can not also have a period in $\mathcal{L}_q^{\vectorise{f}'}$ for any $\vectorise{f}' \neq \vectorise{f}$ without contradiction.
Therefore the size of the set of necklaces can be rewritten in terms in terms of the number of Lyndon words as:

\begin{align*}
    |\mathcal{N}_q^{\vectorise{n}}| = \sum\limits_{f_1 | n_1} \sum\limits_{f_2 | n_2} \hdots \sum\limits_{f_d | n_d} |\mathcal{L}_q^{f_1, f_2, \hdots, f_d}|. %\label{eq:multidimensional_necklaces_from_lyndon}
\end{align*}

\noindent
This equation is used to to derive an equation to count the number of Lyndon words using the number of necklaces as a basis.
The necklace counting formula is used to compute the number of Lyndon words through repeated application of the M\"{o}bius inversion formula, giving:

\begin{align*}
    |\mathcal{L}_q^{\vectorise{n}}| = \sum\limits_{f_1 | n_1} \mu\left(\frac{n_1}{f_1}\right) \sum\limits_{f_2 | n_2}\mu\left(\frac{n_2}{f_2}\right)\hdots \sum\limits_{f_d | n_d} \mu\left(\frac{n_d}{f_d}\right) |\mathcal{N}_q^{f_1,f_2\hdots f_d}|%\label{eq:lyndon_counting}
\end{align*}
\end{proof}

% \noindent
% From these equations, an upper and lower bound on the number of necklaces is derived.

% \begin{lemma}
% \label{lem:neck_bounds}
% The number of necklaces is bounded by $\frac{q^{N}}{N} \leq |\mathcal{N}_q^{\vectorise{n}}| \leq q^N$ where $\vectorise{n}$ is the dimension vector and $q$ is the size of the alphabet.
% \end{lemma}

% \begin{proof}
% The upper bound comes directly as the number of possible words.
% Using the above equations, observe that for every word $n_i$, $1$ is a factor.
% As $\phi(1) = 1$, this gives the number of necklaces as at least $\frac{q^{N}}{N}$.
% \end{proof}

\subsection{Counting Atranslational necklaces}

% \noindent
Related to the concept of aperiodic necklaces are \emph{atranslational necklaces}.
Recall that a necklace $\necklace{w}$ is atranslational if there exists no cyclic shift $g \in Z_{\vectorise{n}}$ such that $g \neq (n_1, n_2, \hdots, n_d)$ and $\Angle{\necklace{w}}_g = \Angle{\necklace{w}}$.
Note that while every atranslational word is aperiodic, not every aperiodic word is atranslational.
As this work is the first to formally characterise these objects, this section provides several key results regarding the structure of atranslational words.
The main result of this section is an equation for counting the number of atranslational words using Lyndon words, and by extension necklaces, as a basis.
Before providing our counting algorithms, we formally prove Proposition \ref{prop:translational_symbolisation}, formally characterising translational Lyndon words.

\begin{prop_trans_symb}
Every word $\word{w} \in \mathcal{L}^{\vectorise{n}}_q$ is either in $\mathcal{A}^{\vectorise{n}}_q$ or $\word{w} = \word{u}^p : \Angle{\word{u}^p}_g : \hdots : \Angle{\word{u}^p}_{g^{t - 1}}$ where:
\begin{itemize}
    \item $g$ is a translation where $g_d = p$ and there exists no translation $r < g$ where $\Angle{\word{u}^p}_r = \word{u}^p$.
    \item $\word{u} \in \mathcal{L}_q^{(n_1, \hdots n_{d - 1} r/p)}$.
    $t = \frac{n_d}{r}$ and is the smallest value greater than 0 such that $g^t = I$.
\end{itemize}
\end{prop_trans_symb}

\begin{proof}
Recall that $\word{w} \in \mathcal{L}_q^{\vectorise{n}}$ is used to denote that $\word{w} = \Angle{\word{w}}$ where $\Angle{\word{w}} \in \mathcal{L}_q^{\vectorise{n}}$.
For the sake of contradiction let $\word{w} \in \mathcal{L}^{\vectorise{n}}_q$ be an aperiodic word that is neither atranslational nor of the form $\word{u}^p : \Angle{\word{u}^p}_g : \hdots : \Angle{\word{u}^p}_{g^{t - 1}}$ for $\word{u} \in \mathcal{L}_q^{(r/p,n_d - 1,\hdots,n_1)}$.
As $\word{w}$ is not atranslational, let $g$ be the translation such that $\word{w} = \Angle{\word{w}}_g$.
Further let $\word{u}$ be the prefix of $\word{w}$ corresponding to the first $g_d$ slices.
If $\word{u} \notin \mathcal{L}_q^{(g_d/p,n_d - 1,\hdots,n_1)}$ then $\word{u}$ has some period which that is also the period of $\word{w}$.
Otherwise note that $\Angle{\word{w}} = \word{w}$.
Therefore as $\Angle{\word{w}}_g = \word{w}$, $\Angle{\word{w}_{[g_d + 1, 2 \cdot g_d]}}_{(g_1, g_2, \hdots, g_{d - 1})} = \word{u}^p$.
More generally, $\Angle{\word{w}_{[(l - 1)\cdot g_d + 1, l \cdot g_d]}}_{(g_1, g_2, \hdots, g_{d - 1})^l} = \word{u}^p$.
This allows $\word{w}$ to be written as $\word{u}^p : \Angle{\word{u}^p}_{(g_1, g_2, \hdots, g_{d - 1})} : \hdots :  \Angle{\word{u}^p}_{(g_1, g_2, \hdots, g_{d - 1})^{t - 1}}$.
Note that if $t < \frac{n_d}{g_d}$ then $\Angle{\word{w}}_g = \Angle{\word{u}}_{(g_1, g_2, \hdots, g_{d - 1})} : \Angle{\word{u}}_{(g_1, g_2, \hdots, g_{d - 1})^2} : \hdots :  \Angle{\word{u}}_{(g_1, g_2, \hdots, g_{d - 1})^{t - 1}}$, therefore $\Angle{\word{w}}_g = \word{w}$ if and only if $\word{u} = \Angle{\word{u}}_{(g_1,g_2,\hdots,g_{d - 1})}$.
If $\word{u} = \Angle{\word{u}}_{(g_1,g_2,\hdots,g_{d - 1})}$, then $\word{w} = \word{u}^p : \Angle{\word{u}^p}_{(g_1,g_2,\hdots,g_{d - 1})} : \hdots : \Angle{\word{u}^p}_{(g_1,g_2,\hdots,g_{d - 1})^{t - 1}} = \word{u}^{p \cdot t}$.
Hence $\word{w}$ would be periodic.
Similarly, if $t > \frac{n_d}{g_d}$ and $t \bmod \frac{n_d}{g_d} \not\equiv 0$ then for $\word{w} = \Angle{\word{w}}_g, \word{u} = \Angle{\word{u}}_{(g_1,g_2,\hdots,g_{d - 1})}$ meaning $\word{w} = \word{u}^{p \cdot t}$.
Further, if $t > \frac{n_d}{g_d}$ and $t \bmod \frac{n_d}{g_d} \equiv 0$ then $\word{w}$ has a period of size $\left(n_1,n_2,\hdots,n_{d - 1}, \frac{n_d}{t}\right)$.
Therefore for $\word{w}$ to be aperiodic and not a translational it must be of the form $\word{u}: \Angle{\word{u}}_{(g_1, \hdots, g_d)} : \hdots : \Angle{\word{u}}_{(g_1, \hdots, g_d)^{t - 1}}$ where $t = \frac{n_d}{g_d}$.
% In the other direction, if the stated conditions are met then $\word{w} = \word{u} : \Angle{\word{u}}_{(g_1,g_2,\hdots,g_{d - 1})} : \hdots : \Angle{\word{u}}_{(g_1,g_2,\hdots,g_{d - 1})^{t - 1}}$ and $\word{u} \in \mathcal{L}^{(n_1,n_2,\hdots,n_{d - 1},r)}_q$ then $\word{w} \in \mathcal{L}^{\vectorise{n}}_q$ and $\word{w} \notin \mathcal{A}_q^{\vectorise{n}}$.
% Similarly if $\word{u} \notin \mathcal{A}^{(n_1,n_2,\hdots,n_{d - 1},r)}_q$ it must be in $\mathcal{L}^{(n_1,n_2,\hdots,n_{d - 1},r)}_q$.
\end{proof}

% This section covers the problem counting the number of aperiodic necklaces in terms of Lyndon words.

\noindent
Our techniques for counting atranslational necklaces operates by computing the number of translational (non-atranslational) Lyndon words, corresponding to the size of the set $\mathcal{L}_q^{\vectorise{n}} \setminus \mathcal{A}_q^{\vectorise{n}}$.
This is achieved by using the recursive structure given in Proposition \ref{prop:translational_symbolisation}, with atranslational necklaces as a basis.
By showing that any translational Lyndon word can be written in the form of some atranslational necklaces under a set of translations, it becomes natural to formulate the number of Lyndon words as an equation in terms of atranslational necklaces.
Theorem \ref{thm:lyndon_to_atranslational} inverts this formulation to give the number of atranslational necklaces in terms of Lyndon words and atranslational necklaces of strictly smaller size.
As Lyndon words of any size and dimensions can be counted, and 1D atranslational necklaces are equivalent to Lyndon words, this equation in terms of Lyndon word and atranslational words with smaller dimension can be evaluated recursively.
% This structure is used as the basis for counting first the number of Lyndon words in terms of the number of atranslational necklaces, then the number of atranslational necklaces in terms of the number of Lyndon words.

% \newtheorem*{theorem_L_to_A}{Theorem \ref{thm:lyndon_to_atranslational}}

\begin{theorem_L_to_A}
% \label{thm:lyndon_to_atranslational}
The number of atranslational words of size $\vectorise{n}$ over an alphabet of size $q$ is given by:

$$
|\mathcal{A}^{\vectorise{n}}_q| = |L_q^{\vectorise{n}}| - \sum\limits_{i \in [d]} \sum\limits_{l | n_i} \begin{cases}
0 & l = n_i\\
\left(\prod\limits_{t = i + 1}^{d - 1} -\mu(n_t)\right)\left(-\mu\left(\frac{n_i}{l}\right)\right) |\mathcal{A}^{n_1,n_2,\hdots,n_{d - 1},l}_q| \cdot H(i,l,\vectorise{n},d) & 1 < l < n_d
\end{cases}
$$
\end{theorem_L_to_A}

\begin{theorem}
\label{thm:lyndon_to_atranslational}
The number of atranslational necklaces of size $\vectorise{n}$ over an alphabet of size $q$ is given by:
$$
|\mathcal{A}^{\vectorise{n}}_q| = |L_q^{\vectorise{n}}| - \sum\limits_{i \in [d]} \sum\limits_{l | n_i} \begin{cases}
0 & l = n_i\\
\left(\prod\limits_{t = i + 1}^{d - 1} -\mu(n_t)\right)\left(-\mu\left(\frac{n_i}{l}\right)\right) |\mathcal{A}^{n_1,n_2,\hdots,n_{d - 1},l}_q| \cdot H(i,l,\vectorise{n},d) & 1 < l < n_d
\end{cases}
$$
\end{theorem}

\noindent
This section is laid out as follows.
% Proposition \ref{prop:translational_symbolisation} provides a formal characterisation of the set of aperiodic words that are not atranslational.
Lemmas \ref{lem:dominating_factor}, and \ref{lem:G_set} provide key combinatorial results that are used to build the equation presented in Lemma \ref{lem:atransaltional_to_lyndon} to count the number of Lyndon words in terms of atranslational necklaces.
These lemmas take advantage of Proposition \ref{prop:translational_symbolisation} to build the foundational structure of the translational words.
Finally Theorem \ref{thm:lyndon_to_atranslational} is restated and formally proven.

Following the characterisation of translational Lyndon words given by Proposition \ref{prop:translational_symbolisation}, the next obvious question is how to count the number of atranslational necklaces.
% To do so two further results are needed to reduce the complexity of the counting problem.
% Lemma \ref{lem:translational_cartesian} establishes the key relationship between Lyndon words in $\mathcal{L}_q^{\vectorise{n}}$ and the translational, aperiodic words in $\mathcal{L}_q^{n_1, n_2, \hdots,n_{d - 1}, m}$ for $m = n_d \cdot c$ that use words in $\mathcal{L}_q^{\vectorise{n}}$ as a translational period.
Lemma \ref{lem:dominating_factor} shows which translational Lyndon words can be represented in the form outlined by Proposition \ref{prop:translational_symbolisation} using as the translational period both some member of $\mathcal{L}_q^{n_1,n_2,\hdots,n_{d - 1}, c}$ and some member of $\mathcal{L}_q^{n_1,n_2,\hdots,n_{d - 1}, c\cdot d}$ for some pair of integers $c,d \in \mathbb{Z}$.
This relationships form the basis of our counting technique used in Lemma \ref{lem:atransaltional_to_lyndon} to count the number of Lyndon words in terms of atranslational necklaces.

\begin{lemma}
\label{lem:dominating_factor}
Let $\vectorise{n}$ be a vector of size.
Given some value $f$ which is a factor of $n_d$, and value $c$ which is a factor of $f$, for any word $\word{a} \in \mathcal{L}^{n_1,n_2,\hdots,n_{d - 1},c}_q$ such that $\word{a}^r : \Angle{\word{a}^r}_g : \hdots : \Angle{\word{a}^r}_{g^{t - 1}} \in \mathcal{L}^{\vectorise{n}}_q$ there exists some word $\word{b} \in \mathcal{L}^{n_1,n_2,\hdots,n_{d - 1},f}_q$ such that $\word{a}^r : \Angle{\word{a}^r}_g : \hdots : \Angle{\word{a}^r}_{g^{t - 1}} = \word{b} : \Angle{\word{b}}_{g'} : \hdots : \Angle{\word{b}}_{g'^{t' - 1}}$ where $r\cdot c \leq f$.
\end{lemma}

\begin{proof}
This claim is shown by considering two cases based on the value of $r$ relative to $f$.
The first case is when $r = \frac{f}{c}$.
In this case let $g' = g$ and $\word{b} = \word{a}^{r - 1} : \Angle{\word{a}}_g$.
Clearly the Lyndon word $\word{a}^r : \Angle{\word{a}^r}_g : \hdots : \Angle{\word{a}^r}_{g^{t - 1}}$ is equivalent to $\word{b} : \Angle{\word{b}}_g :  \hdots : \Angle{\word{b}}_{g^{t - 1}}$.
In the second case $r < \frac{f}{c}$.
If $c \cdot r$ is a factor of $f$, then either the word $\word{a}^r : \Angle{\word{a}^r}_g : \hdots : \Angle{\word{a}}_{g^{f/(r\cdot c)}} \in \mathcal{A}^{n_1,n_2,\hdots,n_{d - 1}, f}_q$ or $\word{a}^r : \Angle{\word{a}^r}_g: \hdots : \Angle{\word{a}^r}_{g^{t - 1}}$ is periodic, contradicting the initial assumption.
If $c \cdot r$ is not a factor of $f$, then let $r' = \frac{f}{c} \bmod r$ and $t' = \floor{\frac{f}{c \dot r}}$.
If $\word{a}^r : \Angle{\word{a}^r}_g : \Angle{\word{a}^{r'}}_{g^{t'}}$ is not atranslational then $\word{a}^r : \Angle{\word{a}^r}_g : \hdots : \Angle{\word{a}^r}_{g^{t - 1}}$ must be periodic with a period in dimension $d$ of at least $f$.
Hence $\word{a}^r : \Angle{\word{a}^r}_g : \Angle{\word{a}^{r'}}_{g^{t'}} \in A^{n_1,n_2,\hdots,n_{d - 1},f}_q$.
\end{proof}

\noindent
% In order to use these Lemma \ref{lem:translational_cartesian} to relate the number of Lyndon words to the number of atranslational necklaces it is important to count the number of possible translations.
The main challenge is to account for $d$-dimensional translational Lyndon words made of $(d - 1)$-dimensional translational Lyndon words.
To this end the set $\mathbf{G}(l,\vectorise{n}) = \{(x_1,x_2,\hdots,x_{d - 1}) \in [\vectorise{n}] : x_i^{n_d/l} \bmod n_i \equiv 0,$ and for some dimension $i$, there exists no value of $j \in [\frac{n_d}{l} - 1]$ such that $x_i^j \bmod n_i \equiv 0 \}$ is introduced.
This set counts the number of possible translations of a $d$-dimensional atranslational word of size $(n_1,n_2, \hdots, n_{d - 1}, l)$ that may be used to build a $d$-dimensional Lyndon word of size $\vectorise{n}$.
The following Lemma provides an important step in the computation of the number of $d - 1$-dimensional atranslational necklaces that can be used to build a $d$-dimensional Lyndon word.

\begin{lemma}
\label{lem:G_set}
Let $\mathbf{G}(l,\vectorise{n}) = \{(x_1,x_2,\hdots,x_{d - 1}) \in [\vectorise{n}] : x_i^{n_d/l} \bmod n_i \equiv 0,$ and for some dimension $i$, there exists no value of $j \in [\frac{n_d}{l} - 1]$ such that $x_i^j \bmod n_i \equiv 0 \}$.
Given some translation $t \in \mathbf{G}(l,(n_1,n_2,\hdots, n_{d - 1}))$, $(t_1, t_2, \hdots, t_{d - 2}, \frac{n_{d - 1}}{l}) \in  \mathbf{G}(l,\vectorise{n})$ if and only if $l = 1$ and $n_{d - 1} = n_d$.
\end{lemma}

\begin{proof}
% For $g_{i + 1}$ observe that any translation of the form $(g_{i,1} , g_{i,2}, \hdots, g_{i,i - 1}, \frac{n_i}{l})$ applied to $\word{w}[i + 1]$ would simply return $\word{w}[i + 1]$.
% In order to count the number of possible translations for $g_{i + 1}$, we claim that for any pair of translations $t,s \in \mathbf{G}(l, (n_1,n_2,\hdots,n_i))$ if $(t_1,t_2, \hdots, t_{i - 1} \frac{n_i}{l}) \in G(1, (n_1,n_2,\hdots,n_{i + 1}))$ then $(s_1,s_2, \hdots, s_{i - 1} \frac{n_i}{l}) \in G(1, (n_1,n_2,\hdots,n_{i + 1}))$.
% For the sake of contradiction, assume that $(t_1,t_2, \hdots, t_{i - 1} \frac{n_i}{l}) \in G(1, (n_1,n_2,\hdots,n_{i + 1}))$ and $(s_1,s_2, \hdots, s_{i - 1} \frac{n_i}{l}) \notin \mathbf{G}(1, (n_1,n_2,\hdots,n_{i + 1}))$.
% There are two possible cases to consider.
% Either, for every $a \in [i - 1]$, there exists some $j < n_{i + 1}$ such that $s_a \cdot j \bmod n_a \equiv 0$ or, for some $a \in [i - 1]$ $s_a \cdot {n_{i + 1}} \bmod n_{a} \not\equiv 0$.

% In the first case,
Observe that $\left(\frac{n_i}{l}\right)\cdot n_{i + 1} \bmod n_i \equiv 0$.
Further note that $\frac{n_i}{l}$ must be the smallest translation such that $t_a \cdot \frac{n_i}{l} \bmod n_a \equiv 0$ for every $a \in [i - 1]$ and hence $\frac{n_i}{l}$ must be a factor of $n_{i + 1}$.
Additionally, if $n_{i + 1} > \frac{n_i}{l}$, then $\frac{n_i}{l}$ exists as some value smaller than $n_{i + 1}$ such that $t_a \cdot \frac{n_i}{l} \bmod n_a \equiv 0$.
Hence the only possible value of $n_{i + 1}$ is $\frac{n_i}{l}$ and further for $n_{i + 1}$ to be greater than or equal to $n_i$, $l$ must be equal to $1$ and therefore $n_{i + 1} = n_i$.
Therefore, given some translation $t \in \mathbf{G}(l,(n_1,n_2,\hdots, n_{d - 1}))$, $(t_1, t_2, \hdots, t_{d - 2}, \frac{n_{d - 1}}{l}) \in  \mathbf{G}(l,\vectorise{n})$ if and only if $l = 1$ and $n_{d - 1} = n_d$.
\end{proof}

\noindent
Lemma \ref{lem:G_set} provides the basis for generalising the set $\mathbf{G}(l,\vectorise{n})$ to count the number of ways a $d - i$-dimensional atranslational word can be used to form a $d$-dimensional Lyndon word.
More explicitly, consider the $i$-dimensional atranslational word $\word{w}$.
To use $\word{w}$ as the translational base of some $d$-dimensional Lyndon word, note that there must be some translation applied to $\word{w}$ at every dimension from $i$ to $d$.
Let $\word{u} = (\word{w} : \Angle{\word{w}}_{g} : \hdots : \Angle{\word{w}}_{g^t}):$ $\Angle{(\word{w} : \Angle{\word{w}}_{g} : \hdots : \Angle{\word{w}}_{g^t})}_{h}$ $: \hdots :$ $\Angle{(\word{w} : \Angle{\word{w}}_{g} : \hdots : \Angle{\word{w}}_{g^t})}_{h^s}$.
For $\word{u}$ to be a Lyndon word, $h$ must not be $(g_1, g_2, \hdots, g_i , n_d/l)$ as $(\word{w} : \Angle{\word{w}}_{g} : \hdots : \Angle{\word{w}}_{g^t}) = \Angle{(\word{w} : \Angle{\word{w}}_{g} : \hdots : \Angle{\word{w}}_{g^t})}_{(g_1, g_2, \hdots, g_i, n_d/l)}$.

Using this observation, the following two functions are needed to count the number possible ways an $i$-dimensional atranslational word can be used to build a $d$-dimensional word.
Let $I(i,l,\vectorise{n})$ return the number of dimensions $j \in [i + 1,d]$ where there exists some translation $g \in \mathbf{G}(1,(n_1,n_2, \hdots, n_j))$ such that $(g_1, g_2, \hdots, g_{i - 1}, \frac{n_i}{l},1,1,\hdots,1) \in \mathbf{G}(1,\vectorise{n})$.
% , where $l_j$ equals $1$ if $j > i$ and $l$ otherwise.
The value of $I(i,l,\vectorise{n})$ can be computed using Lemma \ref{lem:G_set} as:

$$I(i,l,(n_1, n_2,\hdots,n_d)) = 
\begin{cases}
0 & i = d \text{ or } l > 1\\
1 + I(i,l,(n_1, n_2, \hdots, n_{d - 1})) & n_i = n_d\\
I(i,l,(n_1, n_2, \hdots, n_{d - 1})) & n_i \neq n_d
\end{cases}$$

\noindent
The function $H(i,l,\vectorise{n},d)$ is used to return the number of possible sets of translations that can be used to build a $d$-dimensional Lyndon word from $\word{w}$.
Note that each such set requires $d - i$ translations if $l = n_{i}$, or $d - i + 1$ translations if $l < n_i$.
If $i = d$ then the value of $H(i,l,\vectorise{n},d)$ is either $1$, if $l = n_d$, or $|\mathbf{G}(l,\vectorise{n})|$ otherwise.
If $i < d$, the number of possible translations of dimensions $d$ equals the size of $\mathbf{G}(1,\vectorise{n})$ minus the number of dimensions where the translation in the lower dimension can be cancelled out by some translation in a higher dimension.
Note that if any translation in dimension $i$ can be cancelled out by some translation in dimensions $j > i$, then following Lemma \ref{lem:G_set} every translation can be.
Therefore the value of $H(i,l,\vectorise{n},d)$ is given by the equation 

$$H(i,l,\vectorise{n},d) = \prod\limits_{j = i}^d
\begin{cases}
1 & i = d\\
(|\mathbf{G}(1, \vectorise{n})| - (I(i,l,\vectorise{n}))) \cdot (H(i, l, (n_1, n_2, \hdots, n_{d - 1}), d - 1)) & i < d
\end{cases}
$$
% \argy{do we indeed have one case?}

\noindent
Using the functions $H(i,l,\vectorise{n},d)$ and $I(i,l,\vectorise{n})$, the number atranslational necklaces of size $\vectorise{n}$ are counted in terms of atranslational necklaces of smaller size and Lyndon words of size $\vectorise{n}$.
Lemma \ref{lem:atransaltional_to_lyndon} shows how to express the number of Lyndon words in terms of atranslational necklaces.
Theorem \ref{thm:lyndon_to_atranslational} builds on this to show how to count the number of atranslational necklaces using Lemma \ref{lem:atransaltional_to_lyndon}.

\begin{lemma}
\label{lem:atransaltional_to_lyndon}
The number of $d$-dimensional Lyndon words of size $\vectorise{n}$ over an alphabet of size $q$ is given in terms of atranslational necklaces as:

$$
|\mathcal{L}^{\vectorise{n}}_q| = |\mathcal{A}_q^{\vectorise{n}}| + \sum\limits_{i \in [d]} \sum\limits_{l | n_i} \begin{cases}
0 & l = n_i\\
\left(\prod\limits_{t = i + 1}^{d - 1} -\mu(n_t)\right)\left(-\mu\left(\frac{n_i}{l}\right)\right) |\mathcal{A}^{n_1,n_2,\hdots,n_{d - 1},l}_q| \cdot H(i,l,\vectorise{n},d) & 1 < l < n_d
\end{cases}
$$
% $$
% \mathcal{L}^{\vectorise{n}}_q = \sum\limits_{l | n_d} \begin{cases}
% |A_q^{\vectorise{n}}| & l = n_d\\
% -\mu\left(\frac{n_d}{l}\right) |\mathcal{A}^{n_1,n_2,\hdots,n_{d - 1},l}_q| \cdot |\mathbf{G}(l,\vectorise{n})| & 1 < l < n_d\\
% -\mu\left(\frac{n_d}{l}\right) |\mathcal{L}^{n_1,n_2,\hdots,n_{d - 1},l}_q| \cdot |\mathbf{G}(l,\vectorise{n})| & l = 1
% \end{cases}
% $$
% Where $\mathbf{G}(l,\vectorise{n}) = \{(x_1,x_2,\hdots,x_{d - 1}) \in [\vectorise{n}] : x_i^{n_d/l} \bmod n_i \equiv 0,$ and for some dimension $i$, there exists no value of $j \in [\frac{n_d}{l} - 1]$ such that $x_i^j \bmod n_i \equiv 0 \}$.
\end{lemma}

\begin{proof}
Note that every Lyndon word is either atranslational itself, or of the form $\word{a}^r : \Angle{\word{a}^r}_g : \hdots : \Angle{\word{a}^r}_{g^{t - 1}}$ for some $\word{a} \in \mathcal{L}^{n_1,n_2,\hdots,n_{d - 1},f}$.
Following Lemma \ref{lem:dominating_factor}, every Lyndon word of the form $\word{a}^r : \Angle{\word{a}^r}_g : \hdots : \Angle{\word{a}^r}_{g^{t - 1}}$ can be rewritten as $\word{b}: \Angle{\word{b}}_g : \hdots : \Angle{\word{v}}_{g^{t - 1}}$ for some $\word{b} \in \mathcal{A}^{n_1,n_2,\hdots,n_{d - 1},l\cdot r}_q$.
Let $\word{a}$ be than canonical representation of an atranslational necklace of size $(n_1, n_2,\hdots,n_{d - 1},l)$.
For Lyndon words with a $d$-dimensional translational period there are three cases to consider.
If $l = n_d$, then $\word{a} \in \mathcal{A}^{n_1,n_2, \hdots,n_d}_q$.
If $\frac{n_d}{l}$ is prime then for every cyclic shift of $X = (x_1,x_2,\hdots,x_{d - 1})$ where $x_i \in 1\hdots n_i - 1$ such that $x_i^{n_d/l} \bmod n_i \equiv 0$ and for some $i$ $\nexists j \in 1 \hdots \frac{n_d}{l} - 1$, the word $\word{a} : \Angle{\word{a}}_X : \hdots : \Angle{\word{a}^r}_{X^{(n_2/l) - 1}} \in \mathcal{L}^{\vectorise{n}}_q$.
% For notation let $\mathbf{G}(l,\vectorise{n}) = \{(x_1,x_2,\hdots,x_{d - 1}) \in [\vectorise{n}] : x_i^{n_d/l} \bmod n_i \equiv 0,$ and for some dimension $i$, there exists no value of $j \in [\frac{n_d}{l} - 1]$ such that $x_i^j \bmod n_i \equiv 0 \}$.
The number of words of the form $\word{a} : \Angle{\word{a}}_g: \hdots : \Angle{\word{a}}_{g^{(n_d/l) - 1}} \in \mathcal{L}^{\vectorise{n}}_q$ is $|\mathcal{G}(l,\vectorise{n})| \cdot |\mathcal{A}^{n_1,n_2,\hdots,n_{d - 1},l}_q|$.

In the case that $\frac{n_d}{l}$ is not prime, following Lemma \ref{lem:dominating_factor} there exists some $d'$ such that $\word{b} = \word{a} : \Angle{\word{a}}_g : \hdots : \Angle{\word{a}}_{g^{t'}}$ where $\word{b}$ has size $(n_1, n_2, \hdots, l')$.
If there are at least two distinct prime factors of $\frac{n_d}{l}$, then note that $\word{a} : \Angle{\word{a}}_g : \hdots : \Angle{\word{a}}_{g^t}$ is counted for each prime factor.
Let $p$ be the number of distinct prime factors.
To avoid over counting, every word of size $(n_1, n_2, \hdots, n_{d - 1}, l)$ needs to be subtracted $p - 1$ times.
To this end, a new function $P(t)$ is introduced to act as a correction factor.

If $p = 2$ then by setting $P(2) = -1$ the over counting is avoided.
If $p = 3$, then as these words were counted three times for each prime factor, then subtracted three times $\frac{n_2}{d \cdot i}$ for each $i$ in the set of prime factors, to avoid under counting these words $P(3)$ must return 1.
One special case is when $\frac{n_d}{l}$ has a square prime factor, $i^2$.
In this case as $\frac{n_d}{l \cdot i}$ has the same number of distinct primes, $P(\frac{n_d}{l})$ must return 0.
Repeating this argument, $P(s)$ is $-1$ if $s$ has an even number of prime factors, $1$ if $s$ has an odd number of prime factors, and $0$ otherwise.
Note that this corresponds to $-1(\mu\left(\frac{n_d}{l}\right))$ where $\mu\left(\frac{n_d}{l}\right)$ is the m\"{o}bius function.
Further, as $P(1) = 1$, both the prime and non-prime cases can be combined into one case.

The same arguments are applied to the lower dimensional case.
Note that the number of possible translations in this case is given by $H(i,l,\vectorise{n},d)$.
This gives the number of Lyndon words with a translational period of size $(n_1,n_2,\hdots,n_{i - 1},l,1,1\hdots,1)$ as $|\mathcal{A}^{(n_1,n_2,\hdots,n_{i - 1},l,1,1\hdots,1)}_q| \cdot H(i,l,\vectorise{n},d)$,
where $l$ is a factor of $n_i$.
In order to account for over counting, the number of possible Lyndon words is multiplied by $\left(\prod\limits_{t = i + 1}^{d - 1} -\mu(n_t)\right)\left(-\mu\left(\frac{n_i}{l}\right)\right)$.
Therefore the total number of Lyndon words of size $\vectorise{n}$ is equal to:

$$
|\mathcal{L}^{\vectorise{n}}_q| = |A_q^{\vectorise{n}}| + \sum\limits_{i \in [d]} \sum\limits_{l | n_i} \begin{cases}
0 & l = n_i\\
\left(\prod\limits_{t = i + 1}^{d - 1} -\mu(n_t)\right)\left(-\mu\left(\frac{n_i}{l}\right)\right) |\mathcal{A}^{n_1,n_2,\hdots,n_{i - 1},l}_q| \cdot H(i,l,\vectorise{n},d) & 1 < l < n_d
\end{cases}
$$

\end{proof}

\begin{theorem_L_to_A}
The number of atranslational necklaces of size $\vectorise{n}$ over an alphabet of size $q$ is given by:

$$
|\mathcal{A}^{\vectorise{n}}_q| = |L_q^{\vectorise{n}}| - \sum\limits_{i \in [d]} \sum\limits_{l | n_i} \begin{cases}
0 & l = n_i\\
\left(\prod\limits_{t = i + 1}^{d - 1} -\mu(n_t)\right)\left(-\mu\left(\frac{n_i}{l}\right)\right) |\mathcal{A}^{n_1,n_2,\hdots,n_{i - 1},l}_q| \cdot H(i,l,\vectorise{n},d) & 1 < l < n_d
\end{cases}
$$
\end{theorem_L_to_A}

\begin{proof}
It follows from Lemma \ref{lem:atransaltional_to_lyndon} that the number of translational words in $\mathcal{L}_q^{\vectorise{n}}$ is given by the equation $$
|\mathcal{L}^{\vectorise{n}}_q \setminus \mathcal{A}_q^{\vectorise{n}}| = \sum\limits_{i \in [d]} \sum\limits_{l | n_i} \begin{cases}
0 & l = n_i\\
\left(\prod\limits_{t = i + 1}^{d - 1} -\mu(n_t)\right)\left(-\mu\left(\frac{n_i}{l}\right)\right) |\mathcal{A}^{n_1,n_2,\hdots,n_{i - 1},l}_q| \cdot H(i,l,\vectorise{n},d) & 1 < l < n_d
\end{cases}
$$

Hence the number of atranslational necklaces is 
$$|\mathcal{A}^{\vectorise{n}}_q| = |\mathcal{L}^{\vectorise{n}}_q| -  \sum\limits_{i \in [d]} \sum\limits_{l | n_i} \begin{cases}
0 & l = n_i\\
\left(\prod\limits_{t = i + 1}^{d - 1} -\mu(n_t)\right)\left(-\mu\left(\frac{n_i}{l}\right)\right) |\mathcal{A}^{n_1,n_2,\hdots,n_{i - 1},l}_q| \cdot H(i,l,\vectorise{n},d) & 1 < l < n_d
\end{cases}
$$
\end{proof}

\subsection{Counting Fixed Content Multidimensional Necklaces}
\label{subsec:fixed_content_counting}

Following the results for the unconstrained case, the natural question to ask is if there exists similar formulae for the number of fixed content Necklaces, Lyndon words, and atranslational necklaces.
% Here, along with the size of each dimension, the Parikh vector $\vectorise{P}$ defining the number of occurrences of each character will be given.
Starting with $\mathcal{N}_\vectorise{p}^\vectorise{n}$, using the arguments from Graham, Knuth and Patashnik \cite{Graham1994} the number of necklaces can be computed by considering the possible periodic sub-words.
It follows from above that to split along the $i^{th}$ dimension with a period of $t_i$, $\vectorise{P}_j \mod t_i \equiv 0$ for each letter $j$.
For notation, let $\frac{\vectorise{P}}{k} = \left(\frac{P_1}{k}, \frac{P_2}{k}, \hdots, \frac{P_q}{k} \right)$.
Further let $\genfrac(){0pt}{3}{N}{\vectorise{P}}$ denote the multinomial $\genfrac(){0pt}{2}{N}{P_1, P_2, \hdots, P_q}$.
For each subword with periods $t_1$ to $t_D$ there are $\left(\genfrac{}{}{0pt}{2}{\frac{N}{\lcm(t_1, t_2 \hdots t_D)}}{\frac{\vectorise{P}}{\lcm(t_1, t_2 \hdots t_D)}}\right)$ possible fixed-content periods.
% of size $(t_1, t_2, \hdots, t_d)$.
Therefore the total number of fixed content necklaces is:

\begin{align}
    |\mathcal{N}_{\vectorise{P}}^{\vectorise{n}}| = \frac{1}{N} \sum\limits_{t_1 | \gcd(n_1, \vectorise{P})} \phi\left(\frac{n_1}{d_1}\right) \hdots \sum\limits_{t_D | \gcd(n_D, \frac{\vectorise{P}}{d_1 \hdots d_{D - 1}})} \phi\left(\frac{n_D}{d_D}\right) \left(\genfrac{}{}{0pt}{2}{\frac{N}{\lcm(t_1, t_2 \hdots t_D)}}{\frac{\vectorise{P}}{\lcm(t_1, t_2 \hdots t_D)}}\right)
\end{align}

\noindent
Where $\gcd(n,\vectorise{P})$ is the greatest common denominator of both $n$ and every value in the vector $\vectorise{P}$, i.e. $\gcd(n,P_1,P_2,\hdots,P_q)$.
The number of fixed content Lyndon words can be counted though repeated application of the M\'{o}bius inversion formula using the previous arguments as:

\begin{align}
    L_{\vectorise{P}}^{\vectorise{n}} = \sum\limits_{d_1 | \gcd(n_1, \vectorise{P})} \mu\left(\frac{n_1}{d_1}\right) \hdots \sum\limits_{d_D | \gcd(n_D, \frac{\vectorise{P}}{d_1 d_2 \hdots d_{D - 1}})} \mu\left(\frac{n_D}{d_D}\right) |\mathcal{N}_{\vectorise{P}}^{d_1, d_2 \hdots d_D}|
\end{align}

\noindent
Finally the number of atranslational fixed content necklaces is derived using the same arguments as in the unconstrained case.
More specifically, the number of atranslational necklaces of size $\vectorise{v}$ is given by:

\begin{align}
    |\mathcal{A}_{\vectorise{p}}^{\vectorise{n}}| =
    |\mathcal{L}_{\vectorise{p}}^{\vectorise{n}}| -  \sum\limits_{i \in [d]} \sum\limits_{l | n_i} \begin{cases}
        0 & l = n_i\\
        \left(\prod\limits_{t = i + 1}^{d - 1} -\mu(n_t)\right)\left(-\mu\left(\frac{n_i}{l}\right)\right) |\mathcal{A}^{n_1,n_2,\hdots,n_{i - 1},l}_{\vectorise{p}/(n_d \cdot n_{d - 1} \cdot \cdot n_i/l)} | \cdot H(i,l,\vectorise{n},d) & 1 < l < n_d
    \end{cases}
\end{align}
 
%    |\mathcal{L}_{\vectorise{p}}^{\vectorise{n}}| - \sum\limits_{l | \gcd(n_d, \vectorise{{P}})} \begin{cases}
%        0 & l = 1\\
%        - \mu\left(\frac{n_d}{l}\right)|\mathcal{A}_{\vectorise{P}/l}^{n_1,n_2,\hdots,n_{d - 1}, n_d / l}| \cdot |\mathbf{G}(l,\vectorise{n})| & 1 < l < n_d\\
%        - \mu\left(\frac{n_d}{l}\right) |\mathcal{L}_{\vectorise{P}/l}^{n_1,n_2,\hdots,n_{d - 1}}| \cdot |\mathbf{G}(l,\vectorise{n})| & l = n_d
%    \end{cases}

% Similarly, $\vectorise{q}$-weighted multidimensional necklaces may be computed in a similar manner.

% \subsection{Counting Multidimensional Bracelets}

% The final result in this section is 

\section{Generating Necklaces}
\label{subsec:generation}

The idea presented here is based on generation of lower dimensional necklaces, generalising the 1D techniques to the higher dimensional setting.
For the 1D setting, there have been several approaches for the generation of necklaces in constant amortised time, notably those of Cattell, Ruskey, Sawada, Serra, and Miers \cite{Cattell2000} and of Fredricksen and Maiorana \cite{Fredricksen1978}.

Before generating the set of necklace, the idea of a \emph{multidimensional prenecklace} must be established.
Informally, a word is a prenecklace if it is the prefix of the canonical representation of at least one necklace.
A prenecklace is a word $\word{w}$ of size $(n_1,n_2,\hdots,n_d)$ such that there exists some necklace of size $(n_1,n_2,\hdots,n_{d - 1}, n_d + m)$, for some arbitrary $m \in \mathbb{N}$ represented by a word $\word{u}$ such that $\word{u}_{[1,n_d]} = \word{w}$.
The set of prenecklaces of size $\vectorise{n}$ over an alphabet of size $q$ is denoted $\mathcal{P}_q^{\vectorise{n}}$, and is assumed to be ordered as in Definition \ref{def:orderinging}.
Note that the canonical representation of every necklace $\necklace{w}$ is a prenecklace as $\Angle{\necklace{w}} : \Angle{\necklace{w}}$ is the canonical representation of the necklace $\Angle{\Angle{\necklace{w}}:\Angle{\necklace{w}}}$.

Prenecklaces form the basis for the constant amortised time algorithm due to Cattell, Ruskey, Sawada, Serra, and Miers \cite{Cattell2000}.
Before describing our algorithm, we first provide a reminder of the 1D algorithm.
Given a word $\word{w}$, let $\lyn(\word{w})$ return the longest prefix of $\word{w}$ that is the canonical representation of a Lyndon word.
For example, given the word $\word{w} = aaabaaab$, $\lyn(\word{w}) = aaab$.
The 1D algorithm uses these Lyndon prefixes as a means to iterate over the set of all prenecklaces, and by extension necklaces.

\begin{theorem21}
(\cite{Cattell2000}) Let $\word{w} \in \mathcal{P}_q^{n - 1}$ and let $p = |\lyn(\word{w})|$. The word $\word{w} : b$ is in $\mathcal{P}_q^{n}$ if and only if $\word{w}_{n - p} \leq b \leq q$. Furthermore,

$$|\lyn(\word{w} : b)| = \begin{cases}
p & b = \word{w}_{n - p}\\
n & b > \word{w}_{n - p}.
\end{cases}$$
\end{theorem21}

\noindent
Theorem 2.1 is used as the basis for a simple branching algorithm to generate the set of prenecklaces.
The idea is to start with the prenecklace corresponding to the empty word, and to branch on the set of possible symbols to extend it.
This is repeated in a depth first manner, evaluating the lexicographically smallest branch first at each step, until a depth of $n$ is reached.
Figure \ref{fig:1D_GenerationExample} provides a visual illustration.

\begin{figure}
    \centering
    \includegraphics{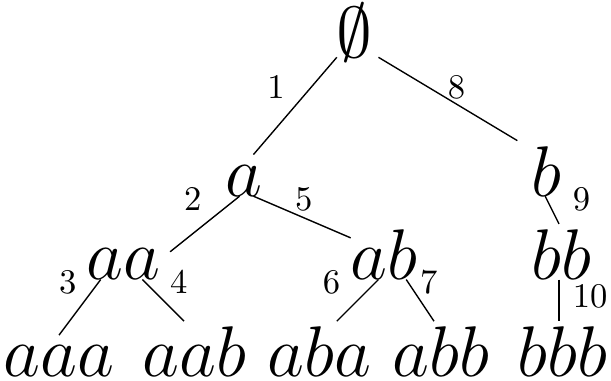}
    \caption{An example of the 1D generation algorithm over the binary alphabet for necklaces of length 3.
    Each edge is labelled with the order that it is traversed in.
    Observe that starting with the empty word, at each step the lexicographically lowest child vertex is first visited until the tree is completely explored.}
    \label{fig:1D_GenerationExample}
\end{figure}

A tempting approach would be to make an alphabet of size equal to the number of necklaces with size $(n_1, \hdots, n_{d - 1})$ and to generate the 1D necklaces from that.
While this approach would generate a set of necklaces, as each $d$-dimensional necklace is comprised of a set of $d-1$-dimensional necklaces, it would also miss any in which one or more slices are translated by any degree.
Similarly, representing every slice under each translation would generate words that are not necklaces.
Let us illustrate it for a set of necklaces over a binary alphabet with size $(2,2)$.
The complete set of necklaces is given in Figure \ref{fig:generation_example}.
Of particular interest is the necklace represented by ${\footnotesize \begin{bmatrix} 
A & B\\
B & A
\end{bmatrix}}$.
While the first row, $A B$, is the canonical representation of a 1D necklace, $B A$ is not as it is equal to $A B$ after a cyclic shift.
Despite $A B$ occurring as the necklace representation multiple times prior to this, $B A$ only occurs at this point.
As such, the situations where some slice may or may not be translated must be understood and taken into account in order to generate the set of necklaces.
\begin{figure}
    \centering
%    \begin{tabular}{l l}
%        Necklace & Code \\
%        $A A$ & 1\\
%        $A B$ & 2\\
%        $B B$ & 3
%    \end{tabular}
   { \footnotesize
    \begin{align*}
    &\begin{bmatrix}
    A & A\\
    A & A
    \end{bmatrix} &\rightarrow &\begin{bmatrix}
    A & A\\
    A & B
    \end{bmatrix} &\rightarrow &\begin{bmatrix}
    A & A\\
    B & B
    \end{bmatrix} &\rightarrow &\begin{bmatrix}
    A & B\\
    A & B
    \end{bmatrix} &\rightarrow &\begin{bmatrix}
    A & B\\
    B & A
    \end{bmatrix} &\rightarrow &\begin{bmatrix}
    A & B\\
    B & B
    \end{bmatrix} &\rightarrow &\begin{bmatrix}
    B & B\\
    B & B
    \end{bmatrix} \\
    &\begin{bmatrix}
    1\\
    1
    \end{bmatrix} &\rightarrow &\begin{bmatrix}
    1\\
    2
    \end{bmatrix} &\rightarrow &\begin{bmatrix}
    1\\
    3
    \end{bmatrix} &\rightarrow &\begin{bmatrix}
    2\\
    2
    \end{bmatrix} &\rightarrow &\begin{bmatrix}
    2\\
    translated(2)
    \end{bmatrix} &\rightarrow &\begin{bmatrix}
    2\\
    3
    \end{bmatrix} &\rightarrow &\begin{bmatrix}
    3\\
    3
    \end{bmatrix}     
    \end{align*}
    \vspace{-0.5cm}
    }
    \caption{An example of generation of $(2,2)$ necklaces, over the alphabet $(A,B)$.  The following mapping from necklace to code has been used: 
        $A A \rightarrow  1$,
        $A B \rightarrow  2$,
        $B B \rightarrow  3$.
    }
    \label{fig:generation_example}
\end{figure}

\noindent
\textbf{Our Algorithm in a nutshell.}
At a high level, the main idea behind our algorithm is to generate the set of all prenecklaces of size $\vectorise{n}$ over the alphabet $\Sigma$ in order.
By extension, this process generates each necklace in order.
Given a word $\word{w} \in \mathcal{P}_q^{\vectorise{n}}$, our algorithm generates the word $\word{u}$ that is subsequent to $\word{w}$ in the ordering.
This is done as follows.
Starting with $\word{w}$, the largest index $i$ such that $\word{w}_i \neq q^{n_1,n_2, \hdots, n_d}$ is determined.
The word $\word{u}$ is created from $\word{w}$ by first incrementing the value of the $i^{th}$ slice of $\word{w}$.
The incrimination of $\word{w}_i$ is done by either translating $\word{w}_i$ by the translation following $TR(\word{w}_i)$ in $Z_{n_1,n_2,\hdots,n_{d - 1}}$, recalling that $TR(\word{w}_i)$ returns the smallest translation $t \in Z_{n_1,n_2,\hdots,n_{d - 1}}$ where $\Angle{\Angle{\word{w}_i}}_t = \word{w}_i$, or by setting $\word{u}_i$ to $NextNecklace(\Angle{\word{w}_i})$ if $TR(\word{w}_i) = TP(\word{w}_i)$, recalling that $TP(\word{w}_i)$ returns the largest translation $t \in Z_{n_1,n_2,\hdots,n_{d-1}}$ such that for every translation $r \in Z_{n_1,n_2,\hdots,n_{d-1}} r < t, \Angle{\word{w}_i}_t \neq \Angle{\word{w}_i}_r$.
After incrementing slice $i$, the remainder of $\word{u}$ is made by repeating the first $i$ slices.
More formally, $\word{u}_j = \word{u}_{j \bmod i}$ for every $j \in [i + 1, n_d]$.
A high level overview of this process is shown in Figure \ref{fig:next_prenecklace_example}.
It is shown that $\word{u}$ is a necklace if and only if $n_d \bmod i \equiv 0$.
By repeating this prenecklace generation at most $n_d$ times, this algorithm guarantees that a necklace is generated.

% \newtheorem*{theorem4}{Theorem \ref{chp8:thm:generation_complexity}}

% \begin{theorem4}
% Let $\word{w}$ be a word of size $\vectorise{n}$.
% $NextNecklace(\word{w})$ returns the smallest word $\word{u} > \word{w}$ such that $\word{u} = \Angle{\word{u}}$ in $O(N)$ time.
% \end{theorem4}

\begin{theorem}
\label{chp8:thm:generation_complexity}
Let $\word{w}$ be a word of size $\vectorise{n}$.
$NextNecklace(\word{w})$ returns the smallest word $\word{u} > \word{w}$ such that $\word{u} = \Angle{\word{u}}$ in $O(N)$ time.
\end{theorem}

\noindent
The remainder of this section proves Theorem \ref{chp8:thm:generation_complexity}.
First, Lemma \ref{lem:prenecklace_property} provides a key characterisation of prenecklaces.
Lemma \ref{lem:prenecklace_property} is strengthened by Lemma \ref{lem:next_prenecklace}, which provides the key structural results used as the basis for generating prenecklaces.
Lemma \ref{lem:prenecklace_counting} is used as the foundation for proving the complexity of Theorem \ref{chp8:thm:generation_complexity}, showing the number of prenecklaces that need to be generated to move from one necklace to the next.
Finally, Theorem \ref{chp8:thm:generation_complexity} is restated and formally proven.

Before presenting our results on prenecklaces, a set of auxiliary functions are introduced.
First given some word $\word{v} \in \Sigma^{\vectorise{n}}$ let $translate(\word{v})$ return the translation in $g \in Z_{\vectorise{n}}$ such that $g_1 = TR(\word{v})_1 + 1 \bmod TP(\word{v})_1$, and $g_i$ is either $TR(\word{v})_i + 1 \bmod TP(\word{v})_1$, if $ g_{i - 1} = 0$ and $0 \neq TR(\word{v})_{i - 1}$, or $g_i = TR(\word{v})_i$ if either $g_{i - 1} \neq 0$ or $0 = TR(\word{v})_i$.
Informally, this can be thought of as choosing the next translation in the ordering defined by the index function, while accounting for the periodicity of $\word{v}$.
Secondly, given some necklace $\necklace{u} \in \mathcal{N}_q^{\vectorise{n}}$ let $ NextNecklace(\necklace{u})$ be a black box function that returns the necklace subsequent to $\necklace{u}$ in $\mathcal{N}_q^{\vectorise{n}}$.
Finally using these functions as a basis let: $$NextSlice(\word{w},i) = \begin{cases}
translate(\word{w}_i) & \begin{split}
&TR(\word{w}_i) < TP(\word{w}_i) \text{ and }\\ 
&TR(\word{w}_i) < \left(TP(\word{w}_{[1,i - 1]})_1, TP(\word{w}_{[1,i - 1]})_2, \hdots, TP(\word{w}_{[1,i - 1]})_{d - 1}\right)
\end{split}\\
NextNecklace(\word{w}_i) & TR(\word{w}_i) = TP(\word{w}_i)
\end{cases}$$
Informally, $NextSlice$ can be thought of as returning the next possible value for the $i^{th}$ slice of the word $\word{w}$, such that $\word{w}_{[1,i - 1]} : NextSlice(\word{w},i)$ remains a prenecklace.

\begin{lemma}
\label{lem:prenecklace_property}
A word $\word{w}$ is a prenecklace if and only if $\word{w}_1 = \Angle{\word{w}}$ and $\word{w}_{[1,i]} \leq \Angle{\word{w}_{[n_d - i,n_d]}}_{g}$ for every $i \in [n_d]$ and $g \in Z_{n_1, n_2, \hdots, n_{d - 1}}$.
\end{lemma}

\begin{proof}
Observe first that if $\word{w}_1 \neq \Angle{\word{w}}$, then $\Angle{\word{w}: \word{u}}_{TR(\word{w}_1)} < \word{w} : \word{u}$ for any arbirtary suffix $\word{u}$, thus $\word{w}$ can not be a prenecklace.
Similarly if $\word{w}_{[1,i]} > \Angle{\word{w}_{[n_d - i,n_d]}}_g$ for some $i \in [n_d]$ and $g \in Z_{n_1, n_2, \hdots, n_{d - 1}}$ then for any $t \in \mathbb{N}$ and word $\word{u} \in \Sigma^{(n_1, n_2,\hdots,n_{d - 1}, t)}$, $\Angle{\word{w}_{[n_d - i,n_d]} : \word{u} : \word{w}_{[1,n_d - i - 1]}}_{g} < \word{w} : \word{u}$.
Hence, there exists no word for which $\word{w}$ is a prenecklace.

In the other direction, let $\word{u} = \word{w} : q^{(n_1,n_2,\hdots, n_{d})}$ for some $\word{w}$ where $\word{w}_1 = \Angle{\word{w}}$ and $\word{w}_{[1,i]} \leq \Angle{\word{w}_{[n_d - i,n_d]}}_{g}$ for every $i \in [n_d]$ and $g \in Z_{n_1, n_2, \hdots, n_{d - 1}}$.
Note that $\word{u} = \Angle{\word{u}}$ if and only if $\word{u} \leq \Angle{\word{u}}_g$ for every $g \in Z_{n_1, n_2, \hdots, 2 \cdot n_{d}}$.
If $\word{w} = q^{(n_1,n_2,\hdots, n_{d})}$ then this condition is satisfied.
Alternatively, if $\word{w} > q^{(n_1,n_2,\hdots, n_{d})}$, then $\word{w}_1 < q^{(n_1,n_2,\hdots, n_{d - 1})}$.
Let $g \in Z_{n_1, n_2, \hdots, 2 \cdot n_{d}}$ be a translation of the form $g  = (g_1, g_2, \hdots, g_{d - 1}, n_d + t)$ for some $t \in [n_d]$.
Clearly $\word{u} < \Angle{\word{u}}_g$ as $\word{u}_1 < (\Angle{\word{u}}_g)_1$.
Similarly, let $r \in Z_{n_1, n_2, \hdots, 2 \cdot n_{d}}$ be a translation of the form $g  = (r_1, r_2, \hdots, r_{d - 1}, t)$ for some $t \in [n_d]$ and let $r' = (r_1, r_2, \hdots, r_{d - 1})$.
Either $\word{w}_{[1,t]} < \Angle{\word{w}_{n_d - t, n_d}}_{r'}$, in which case $\word{u} < \Angle{\word{u}}_{r}$, or $\word{w}_{[1,t]} = \Angle{\word{w}_{n_d - t, n_d}}_{r'}$.
In the second case, as $\word{w}_{[1,t]} = \Angle{\word{w}_{n_d - t, n_d}}_{r'}$ and $\word{w} < q^{(n_1,n_2,\hdots, n_{d})}$, ${\word{w}_{n_d - t, n_d}} < q^{(n_1,n_2,\hdots, n_{d - 1}, t)}$.
Therefore $\word{u} < \Angle{\word{u}}_r$, and subsequently $\word{u} = \Angle{\word{u}}$.
Hence $\word{w}$ is a prenecklace.
\end{proof}

\begin{lemma}
\label{lem:next_prenecklace}
Let $\word{w} \in \Sigma^{\vectorise{n}}$ be a the $j^{th}$ prenecklace in $\mathcal{P}_q^{\vectorise{n}}$ and let $i \in [n_d]$ be the largest index such that $\word{w}_i \neq q^{n_1,n_2,\hdots,n_{d - 1}}$.
Then the $(j + 1)^{th}$ prenecklace in $\mathcal{P}_q^{\vectorise{n}},\word{u}$ has the structure 
$$\word{u}_l = \left(\word{w}_{[1,i - 1]} : NextSlice(\word{w},i)\right)_{l \bmod i}$$
and further $\word{u}$ is the canonical representation of a necklace if and only if $n_d \bmod i \equiv 0$.
\end{lemma}

\begin{proof}
This proof is structured as follows.
First, it is shown that $\word{w}_{[1,i - 1]} : NextSlice(\word{w},i)$ is a necklace.
It is then shown that no smaller prenecklace that $\word{u}$ can exist.
Finally, it is shown that $\word{u}$ is the canonical representation of a necklace if and only if $n_d \bmod i \equiv 0$.

Observe first that as $\word{w}$ is a prenecklace then every prefix of $\word{w}$ must be a prenecklace.
As such $\word{w}' = \word{w}_{[1,i - 1]}$ must be a prenecklace.
Therefore, there can exist no translation $g \in Z_{n_1,n_2,\hdots,n_{d - 1}}$ for which $\word{w}'_{[1,j]} > \Angle{\word{w}'_{[i - j - 1, i - 1]}}_g$, for any $j \in [i - 1]$.
Let $\word{s} = NextSlice(\word{w},i))_{l \bmod i}$.
For $\word{w}' : \word{s}$ to be a necklace, there must be no translation $g \in Z_{n_1,n_2,\hdots,n_{d - 1},i}$ where $\word{w}' : \word{s} > \Angle{\word{w}' : \word{s}}_{g}$.

Consider the case where $\word{s} = translate(\word{w}_i)$.
For the sake of contradiction, assume that under the translation $g \in Z_{n_1,n_2,\hdots,n_{d - 1},i}$, $\word{w}' : \word{s} > \Angle{\word{w}' : \word{s}}_{g}$.
In this case there must be some suffix of $\word{w}'$ such that $\word{w}'_{[1,g_d]} = \Angle{\word{w}'_{[i - 1 - g_d, i - 1]}}_{(g_1, g_2, \hdots, g_{d - 1})}$.
However, this leads to a contradiction as $g$ would have to be greater than or equal to $(TP(\word{w}_{[1,i - 1]})_1, TP(\word{w}_{[1,i - 1]})_2, \hdots, TP(\word{w}_{[1,i - 1]})_{d - 1})$.
Further, note that for any translation $g \geq (TP(\word{w}_{[1,i - 1]})_1, TP(\word{w}_{[1,i - 1]})_2, \hdots, TP(\word{w}_{[1,i - 1]})_{d - 1})$, there exists some translation $g' = (g_1 \bmod TP(\word{w}_{[1,i - 1]})_1, g_2 \bmod TP(\word{w}_{[1,i - 1]})_2, \hdots, g_{d - 1} \bmod TP(\word{w}_{[1,i - 1]})_{d - 1})$ such that $\Angle{\word{w}'}_{(g_1, g_2,\hdots,g_{d - 1})} = \Angle{\word{w}'}_{g'}$.
Therefore, translating $\word{s}$ by any translation greater than $$(TP(\word{w}_{[1,i - 1]})_1, TP(\word{w}_{[1,i - 1]})_2, \hdots, TP(\word{w}_{[1,i - 1]})_{d - 1})$$ leads to representing a word that has previously been looked at.

Consider now the case where $\word{s} = NextNecklace(\word{w}_i)$.
For the sake of contradiction, assume again that under the translation $g \in Z_{n_1,n_2,\hdots,n_{d - 1},i}$ $\word{w}' : \word{s} > \Angle{\word{w}' : \word{s}}_{g}$.
Then there must exist some suffix of $\word{w}'$ such that $\word{w}'_{[1,g_d]} = \Angle{\word{w}'_{[i - 1 - g_d, i - 1]}}_{(g_1, g_2, \hdots, g_{d - 1})}$ and where $\Angle{\word{s}}_{(g_1, g_2, \hdots, g_{d - 1})} < \word{w}'_{g_d + 1}$.
However as $\word{s}$ belongs to a larger necklace class than $\word{w}_i, \word{w}_i < \Angle{\word{s}}_t$ for every translation $t \in Z_{n_1, n_2, \hdots, n_{d - 1}}$.
Further, as $\word{w}_{[1,i]}$ is a prenecklace, $\Angle{\word{w}_i}_{(g_1, g_2, \hdots, g_{d - 1})} \geq \word{w}'_{g_d + 1}$.
Therefore, $\Angle{\word{s}}_{(g_1, g_2, \hdots, g_{d - 1})} > \word{w}'_{g_d + 1}$, and hence $\word{w}' : \word{s}$ must be the canonical representation of a necklace.

Observe that as $\word{w}' : \word{s}$ is the canonical representation of a necklace, any word made by repeating $\word{w}' : \word{s}$ must also be a necklace, and by extension any prefix there of must be a prenecklace.
Therefore $\word{u}$ must be a prenecklace.
For the sake of contradiction, let $\word{v} \in \Sigma^{\vectorise{n}}$ be the canonical representation of some prenecklace such that $\word{w} < \word{v} < \word{u}$.
Following the above arguments, the prefix of $\word{v}$ of length $i$ must equal the prefix of $\word{u}$ of length $i$, i.e. $\word{v}_{[1,i]} = \word{u}_{[1,i]}$.
Therefore, if $\word{v} < \word{u}$ there must exist some index $j \in [i + 1, n_d]$ such that $\word{v}_j < \word{u}_j$.
Starting with the case where $j = i + 1$, if $\word{v}_j < \word{u}_j$, then $\word{v}_j < \word{v}_1$ leading to a contradiction as the suffix starting at position $i + 1$ of $\word{v}$ would be smaller than $\word{v}$.
Similarly, if $j = i + 2$ then if $\word{v}_j < \word{u}_j, \word{v}_j < \word{v}_2$ and by extension $\word{v}_{[1,n_d - j]} < \word{v}_{[j,n_d]}$ contradicting the assumption that $\word{v}$ is a prenecklace.
More generally, for any arbitrary $j \in [i + 1,n_d]$ if $\word{v}_j < \word{u}_j$ then $\word{v}_j < \word{v}_{j - i}$, implying that  $\word{v}_j <  \word{v}_{j \bmod i}$ and by extension $\word{v}_{[1,n_d - j]} < \word{v}_{[j,n_d]}$.
Therefore $\word{u}$ must be the prenecklace with rank $j + 1$.

In order to show that $\word{u}$ is the canonical representation of a necklace if and only if $n_d \bmod i \equiv 0$, it is sufficient to show that the translational period of $\word{w}' : \word{s}$ in dimension $d$ is $i$.
For the sake of contradiction, let there exist some translation $g \in Z_{n_1,n_2,\hdots,n_{d - 1}}$ and translation $r \in Z_i$ such that $\word{w}' : \word{s} = \Angle{\Angle{\word{w}' : \word{s}}_{g}}_{r}$.
In this case, $\word{s}$ must equal $\Angle{\word{w}_{i - r}}_{g}$.
However, as $\word{s} > \word{w}_i, \Angle{\word{w}_{i - r}}_{g} > \word{w}_i$. 
Therefore prefix $\word{w}_{[1,i]}$ can not be a prenecklace as the suffix starting at $r + 1$ would be smaller than the corresponding prefix, contradicting the assumption that $\word{w}$ is a prenecklace.
Therefore the translational period of $\word{w}' : \word{s}$ must be $i$, and hence $\word{u}$ can be a necklace if and only if $n_d \bmod i \equiv 0$.
\end{proof}

\noindent
Lemma \ref{lem:next_prenecklace} provides the basic tool to determine the next prenecklace from a given prenecklace.
From a theoretical stand point, this is all that is needed to describe an algorithm in order to generate the next necklace.
Formally, by repeatedly applying Lemma \ref{lem:next_prenecklace} to some necklace, the next necklace in the ordering is generated.
Lemma \ref{lem:prenecklace_counting} formalises the number of times this process needs to be repeated in order to generate the next necklace.

\begin{lemma}
\label{lem:prenecklace_counting}
Given $\necklace{w},\necklace{u} \in \mathcal{N}_q^{\vectorise{n}}$ such that $rank(\necklace{u}) = rank(\necklace{w}) + 1$, let $Pre(\word{w},\word{u}) = \{\word{v} \in \Sigma^{\vectorise{n}}: \word{u} > \word{v} > \word{w}, \word{v}$ is a prenecklace$\}$.
The size of $Pre(\word{w},\word{u})$ is at most $n_d$.
\end{lemma}

\begin{proof}
This is statement is proven constructively.
Let $NextPrenecklace(\word{u})$ return the smallest prenecklace greater than $\word{u}$, using the techniques outlined in Lemmas \ref{lem:next_prenecklace}.
Let $\word{u}^1 = \word{u}$ and let $\word{u}^t = NextPrenecklace(\word{u}^{t - 1})$.
Similarly let $i^t \in [n_d]$ be the largest index such that $\word{u}^{t}_{i} \neq q^{(n_1,n_2,\hdots,n_{d - 1})}$ and let $\word{s}^t = \word{u}^t_{[1,i]}$.
Following the arguments given in Lemma \ref{lem:next_prenecklace}, every suffix of $\word{s}^t$ under any translation $g \in Z_{n_1,n_2,\hdots, n_{d - 1}}$ where $g < TP(\word{s}_i)$ and $g < TP(\word{s}_{[1,i]}) $ must be strictly greater than $\word{s}$.
Therefore, for any translation $g \in Z_{n_1, n_2, \hdots, n_{d - 1}, i},\word{s} \geq \Angle{\word{s}}_g$ and hence $\word{s}$ is the canonical representation of the necklace $\Angle{\word{s}}$.
Following the construction given in Lemma \ref{lem:next_prenecklace}, if $n_d \bmod i^t \equiv 0$, then $\word{u}^t = \Angle{\word{u}^t}$.
Therefore the number of prenecklaces between $\necklace{w}$ and $\necklace{u}$ equates to the largest value of $t$ to guarantee that every  $n_d \bmod i^t \equiv 0$.
To this end observe that following the construction in \ref{lem:next_prenecklace}, $\word{u}^t_{i^{t - 1} + 1} = \word{u}^{t - 1}_1$ and further $\word{u}^{t - 1} \neq q^{n_1,n_2,\hdots,n_{d - 1}}$.
Therefore $i_t > i_{t - 1}$.
Hence the size of $Pre(\word{w},\word{u})$ is at most $n_d$.
\end{proof}

\noindent
Lemma \ref{lem:prenecklace_counting} is used as the basis for determining the complexity of our generation algorithm.
At a high level, the $O(N)$ bound is due to the number of times $NextPrenecklace$ needs to be recursively called.
Following \ref{lem:prenecklace_counting}, to transform $\word{w}$ representing necklace $\necklace{w}$ to $\word{u}$ representing $\necklace{u}$, $NextPrenecklace$ needs to be called at most $n_d$ times.
However, for each of these calls, it may be necessary to generate a $d - 1$ dimensional necklace, requiring $n_{d - 1}$ calls to $NextPrenecklace$.
Repeating this logic shows that $NextPrenecklace$ can be called no more than $n_1 \cdot n_2 \cdot \hdots \cdot n_d$ times.
Theorem \ref{chp8:thm:generation_complexity} formalises this argument.

\begin{theorem4}
Let $\word{w}$ be a word of size $\vectorise{n}$.
$NextNecklace(\word{w})$ returns the smallest word $\word{u} > \word{w}$ such that $\word{u} = \Angle{\word{u}}$ in $O(N)$ time.
\end{theorem4}

\begin{proof}
Following Lemma \ref{lem:prenecklace_counting}, note that by applying the function $NextPrenecklace$ at most $n_d$ times, the smallest necklace greater than $\word{w}$ can be determined.
As each call to $NextPrenecklace$ requires $NextNecklace$ as a subroutine, to determine the next prenecklace of dimensions $d - 1$, $n_{d - 1}$ prenecklaces of dimensions $d - 2$ must be determined.
Following this logic, to determine the next prenecklace of dimensions $d$ at most $\frac{N}{n_d \cdot n_{d - 1} \cdot \hdots \cdot n_{d - i + 1}}$ prenecklaces of dimensions $i$ must be considered.
Therefore a total of $O(N)$ time is needed to compute all $n_d$ prenecklaces.
As it takes at most $O(N)$ time to determine if a word is a necklace, this process takes at most $O(N)$ time.
\end{proof}

\begin{figure}
    \centering
    \includegraphics{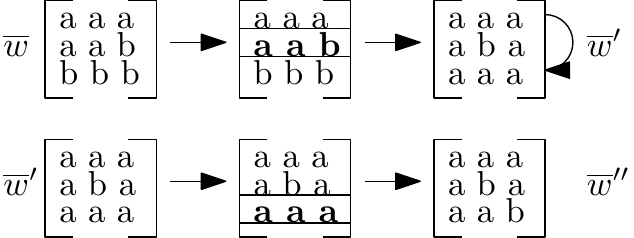}
    \caption{An example of generation algorithms for converting the word $\word{w}$ representing necklace $\necklace{w}$ to the word $\word{w}''$ representing the necklace $\necklace{w}''$, the necklace following $\necklace{w}$ in the ordering.
    In the first iteration $\word{w}_{3}$ is maximal, therefore the slice $\word{w}_2$ is incremented, producing the prenecklace $\word{w}'$.
    As $\word{w}'$ is not the canonical representation of a necklace, $NextPrenecklace$ must be applied again.
    In the second iteration the slice $\word{w}'_{3}$ is incremented, giving $\word{w}''$ which is the canonical representation of the necklace $\necklace{w}''$, terminating the algorithm.}
    \label{fig:generation_algorithm_overview}
\end{figure}

\section{Ranking Multidimensional Necklaces}
\label{sec:ranking}

% \igor{
Informally, the ranking problem, also known as the indexing problem, asks for the number of members of some given ordered set smaller than some element.
Unranking is the reverse process, asking for the element of some ordered set with a given rank.
Ranking has been studied for various objects including partitions~\cite{RankPartition}, permutations~\cite{RankPerm1,RankPerm2}, combinations~\cite{RankComb}, etc.
Unranking has similarly been studied for objects such as permutations~\cite{RankPerm2} and trees~\cite{Gupta1983,Pallo1986}.
The first class of cyclic words to be ranked were \emph{Lyndon words} by Kociumaka, Radoszewski, and Rytter \cite{Kociumaka2014} who provided an $O(n^3)$ time algorithm, where $n$ is the length of the word.
An algorithm for ranking necklaces was given by Kopparty, Kumar, and Saks \cite{Kopparty2016}, without tight bounds on the complexity.
A $O(n^2)$ time algorithm for ranking necklaces was provided by Sawada and Williams \cite{Sawada2017}.
More recently, we have provided an $O(q^2 \cdot n^4)$ time algorithm for ranking the closely related set of cyclic words known as bracelets \cite{Adamson2021}.
%}

% This section covers the ranking algorithm.
Within the setting of multidimensional necklaces $\mathcal{N}_q^{\vectorise{n}}$, the rank of a necklace $\necklace{w}$ is the number of necklaces smaller than $\necklace{w}$ under the ordering given in Definition \ref{def:orderinging}.
More broadly, we can take any word $\word{v}$ and determine the number of necklaces with a canonical representation smaller than $\word{v}$ using the same ordering.
In this case, the smallest necklace greater than or equal to $\word{v}$ is determined using the $NextNecklace$ algorithm given in Theorem \ref{chp8:thm:generation_complexity}.

\begin{theorem}
\label{thm:ranking_complexity}
The rank of a $d$-dimensional necklace in the set $\mathcal{N}_q^{\vectorise{n}}$ can be computed in \RankingComplexity time, where $N = \prod_{i = 1}^d n_i$.
\end{theorem}

% \begin{theorem2}
% The rank of a $d$-dimensional necklace in the set $\mathcal{N}_q^{\vectorise{n}}$ can be computed in \RankingComplexity time, where $N = \prod_{i = 1}^d n_i$.
% \end{theorem2}

\noindent
\textbf{Algorithm Outline}
Our ranking algorithm uses similar mechanisms to the work of Kociumaka, Radoszewski, and Rytter \cite{Kociumaka2014}.
At a high level, our ranking technique for  is based on transforming the number of words belonging to a necklace class smaller than $\word{w}$ into the rank of $\word{w}$ via the rank among Lyndon words and atranslational necklaces.
The relationships established in Section \ref{sec:counting} are used as a basis for converting the size of the sets of words belonging to necklace classes smaller than $\word{w}$, to the size of the set of words belonging to a Lyndon word smaller than $\word{w}$, then to the number of words belonging to an atranslational necklace smaller than $\word{w}$.
Observe that any atranslational necklace of size $\vectorise{n}$ contains exactly $N = n_1 \cdot n_2 \cdot \hdots \cdot n_d$ words, therefore given the number of words belonging to an atranslational necklace smaller than $\word{w}$, the rank of $\word{w}$ within the set of atranslational necklaces can be directly computed.
From the rank of $\word{w}$ within the set of atranslational necklaces, the rank of $\word{w}$ within the sets of Lyndon words and Necklaces are computed.
An overview of the ranking process if given in Figure \ref{fig:ranking_overview}.

\begin{figure}
    \centering
    \includegraphics[scale=0.5]{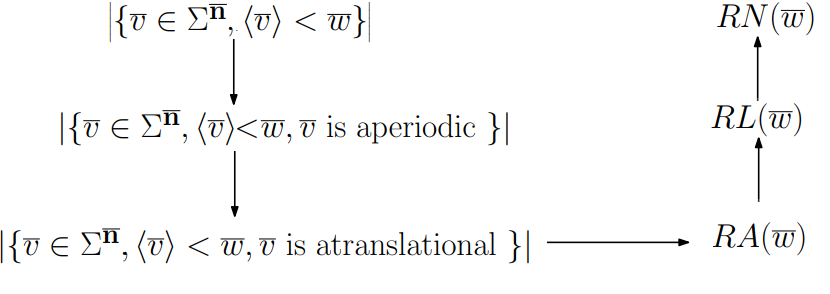}
    \caption{Outline of our ranking technique for some word $\word{w}$ in the set of multidimensional necklaces, along with the associated theoretical tools used.
    This process starts with the set of words belonging to some necklace class smaller than $\word{w}$.
    The size of this set is used to determine the number of aperiodic words belonging to a necklace class smaller than $\word{w}$ (shown in Lemma \ref{lem:t_words_to_l_words}) which in turn is used to determine the number of words belonging to atranslational necklaces smaller than $\word{w}$ (Lemma \ref{lem:a_words_to_l_words}).
    From the number of words belonging to atranslational necklaces smaller than $\word{w}$, the rank of $\word{w}$ is computed, first in terms of atranslational necklaces (Lemma \ref{lem:a_words_to_a_rank}), then Lyndon words (Lemma \ref{lem:a_rank_to_l_rank}) and finally necklaces (Lemma \ref{lem:l_rank_to_T_rank}).
    % \duncan{CHANGE - fix typo, give more space.}
    }
    \label{fig:ranking_overview}
\end{figure}

This leaves the problem of computing the number of words belonging to a necklace class smaller than $\word{w}$.
A recursive approach similar manner to the technique presented by Sawada and Williams \cite{Sawada2017} is used.
At a high level, this set of words is partitioned based on two properties; the smallest translation $g \in Z_{\vectorise{n}}$ such that $\Angle{\word{u}}_g < \word{w}$, and the length $j$ of the longest common prefix between $\Angle{\word{u}}_g$ and $\word{w}$, i.e. the largest value such that $\left(\Angle{\word{u}}_g\right)_{[1,j]} = \word{w}_{[1,j]}$.
The size of each of these subsets is computed in a combinatorial manner, by providing a characterisation of words based on the values of $j$ and $g$.
In each case, the main computational cost is due to counting the number of words of some length determined by $j$ and $g$ such that each suffix of these words under any translation is strictly greater than the prefix of $\word{w}$.
The number of such words is computed using a recursive formula, observing that if $\word{v}$ is a word where every suffix is greater than $\word{w}$ then $\word{v}_{[i,|\word{v}|]}$ must itself be a word such every that suffix is greater than $\word{w}$.

Before presenting the further technical details of our algorithm, so notation and definitions must be established.
For the remainder of this section, it is assumed that the word being ranked is the canonical representation of a necklace.
First, it is necessary to define a method of comparing two words of different sizes.
In this section, two words $\word{w} \in \Sigma^{\vectorise{n}}$ and $\word{u} \in \Sigma^{\vectorise{f}}$ are compared if and only if $n_i \bmod f_i \equiv 0$ for every $i \in [d]$.
As such, given such a pair of words $\word{u}^{\vectorise{n}/\vectorise{f}}$ is used to denote the word $\word{u}'$ where $\word{u}'_{(i_1,i_2,\hdots,i_d)} = \word{u}_{(i_1 \bmod f_1, i_2 \bmod f_2, \hdots, i_d \bmod f_d)}$.
Using this notation, a comparison between word $\word{w}$ and $\word{u}$ is given as:

\begin{definition}
\label{def:different_size_comparison}
Let $\word{u} \in \Sigma^{\vectorise{f}}$, and $\word{v}\in \Sigma^{\vectorise{n}}$ where $n_i \bmod f_i \equiv 0$.
$\word{u} < \word{v}$ if and only if $\word{u}^{\vectorise{n}/\vectorise{f}} < \word{v}$ following Definition \ref{def:orderinging}.
Similarly, $\word{u} > \word{v}$ if and only if $\word{u}^{\vectorise{n}/\vectorise{f}} > \word{v}$.
\end{definition}

\noindent
At a high level, the ranking algorithm for a word $\word{w}$ works by first determining the number of words of size $\vectorise{f} = (f_1, f_2, \hdots, f_d)$ smaller than $\word{w}$, denoted $T(\word{w}, \vectorise{f})$, for every $f_i$ that is factor of $n_i$.
This value is transformed, first from $T(\word{w}, \vectorise{f})$ to the number of aperiodic words smaller than $\word{w}$, denoted $L(\word{w}, \vectorise{f})$, and finally to the number of atranslational words smaller than $\word{w}$, $A(\word{w}, \vectorise{f})$.
The set $A(\word{w}, \vectorise{f})$ is then translated into the rank of $\word{w}$ within the set of atranslational necklaces $A_q^{\vectorise{f}}$, denoted $RA(\word{w}, \vectorise{f})$.
This rank is than used to calculate the rank within the set of Lyndon words $RL(\word{w}, \vectorise{f})$.
Finally, this rank is translated to the necklace rank $RN(\word{w}, \vectorise{f})$.
Lemmas \ref{lem:t_words_to_l_words}, and \ref{lem:a_words_to_l_words} show how to transform the size of the sets $T_{\word{w}, \vectorise{f}}$ into the size of $A(\word{w}, \vectorise{n})$.
Lemmas \ref{lem:a_words_to_a_rank}, \ref{lem:a_rank_to_l_rank} and \ref{lem:l_rank_to_T_rank} show how to transform the size of the sets $A(\word{w}, \vectorise{f})$ into the value $RN(\word{w}, \vectorise{n})$.

In order to compute the size of $T(\word{w},\vectorise{f})$, $T(\word{w},\vectorise{f})$ is partitioned into the subsets $\mathbf{B}(\word{w},g,j,\vectorise{f})$.
Here $\mathbf{B}(\word{w},g,j,\vectorise{f})$ contains the set of words $\word{v} \in T(\word{w},\vectorise{f})$ where: (1) $g$ is the smallest translation such that $\Angle{\word{v}}_g < \word{w}$ and (2) $j$ is the length of the longest shared prefix between $\Angle{\word{v}^{\vectorise{n}/\vectorise{f}}}_g$ and $\word{w}'$, i.e. the largest value such that $\left(\Angle{\word{v}^{\vectorise{n}/\vectorise{f}}}_g\right)_{[1,j]} = \word{w}_{[1,j]}$.
The size of each set $\mathbf{B}(\word{w},g,j,\vectorise{f})$ is computed by considering the structure of the words in $\mathbf{B}(\word{w},g,j,\vectorise{f})$.
This requires the size of two further sets to be computed, the number of non-cyclic words where every suffix is greater than $\word{w}$, and the number of words of size $(f_1, f_2, \hdots, f_{d - 1})$ that are smaller than $\word{w}_{j + 1}$.
The first of these sets is the more technical, requiring a new recursive technique to be built which is provided in Subsection \ref{subsec:beta}.

The remainder of this section proves Theorem \ref{thm:ranking_complexity}.
For ease of reading, it has been subdivided as follows.
Section \ref{sec:ranking_tool} covers the theoretical tools needed to transform the size of the set $T(\word{w},\vectorise{f})$ into the rank of $\necklace{w}$.
Section \ref{subsec:comp_tool} provides the main tools used to compute the size of $T(\word{w},\vectorise{f})$.
% Finally, Subsection \ref{subsec:beta} covers the main sub method used in the ranking process.
Finally Theorem \ref{thm:ranking_complexity} is restated and formally proven.

\subsection{Theoretical Tools}
\label{sec:ranking_tool}

This section covers the theoretical tools that are used to rank necklaces.
At a high level, the goal is to start with the set $T(\word{w},\vectorise{f})$ and show how to convert it to the rank of $\necklace{w}$, via the sets $L(\word{w}, \vectorise{f})$ (Lemma \ref{lem:t_words_to_l_words}) and $A(\word{w}, \vectorise{f})$ (Lemma \ref{lem:a_words_to_l_words}).
From the set $A(\word{w}, \vectorise{f})$, the rank of $\necklace{w}$ is computed, first in the set of atranslational necklaces $\mathcal{A}_q^{\vectorise{n}}$ (Lemma \ref{lem:a_words_to_a_rank}), then the set of Lyndon words $\mathcal{L}_q^{\vectorise{n}}$ (Lemma \ref{lem:a_rank_to_l_rank}) and finally within the set of necklaces $\mathcal{N}_q^{\vectorise{n}}$ (Lemma \ref{lem:l_rank_to_T_rank}).
This section utilises many of the relationships between the sets of necklaces, Lyndon words, and atranslational necklaces established in Section \ref{sec:counting}.

\begin{lemma}
\label{lem:t_words_to_l_words}
The size of $L(\word{w},\vectorise{n})$ can be computed in terms of $T(\word{w}, \vectorise{f})$ using the equation:

$$
|L(\word{w}, \vectorise{n})| = \sum\limits_{f_1 | n_1} \mu\left(\frac{n_1}{f_1}\right) \sum\limits_{f_2 | n_2} \mu\left(\frac{n_2}{f_2}\right) \hdots \sum\limits_{f_d | n_d} \mu\left(\frac{n_d}{f_d}\right) |T(\word{w},\vectorise{f})|
$$
\end{lemma}

\begin{proof}
Observe that every word in $T(\word{w},\vectorise{n})$ is either aperiodic, in which case it is in $L(\word{w},\vectorise{n})$, or periodic, in which case the period of $\word{w}$ is in $L(\word{w},\vectorise{f})$ where $f_i$ is a factor of $n_i$.
Following the same arguments as given in Section \ref{subsec:generation}, the size of $T(\word{w},\vectorise{n})$ is equal to $\sum\limits_{f_1 | n_1}\sum\limits_{f_2 | n_2} \hdots \sum\limits_{f_d | n_d} |L(\word{w},\vectorise{f})|$.
By repeated application of the M\"{o}bius inversion formula, the size of $L(\word{w},\vectorise{n})$ can be computed as:

$$
|L(\word{w},\vectorise{n})| = \sum\limits_{f_1 | n_1} \mu\left(\frac{n_1}{f_1}\right) \sum\limits_{f_2 | n_2} \mu\left(\frac{n_2}{f_2}\right) \hdots \sum\limits_{f_d | n_d} \mu\left(\frac{n_d}{f_d}\right) |T(\word{w},\vectorise{f})|
$$
\end{proof}

\begin{lemma}
\label{lem:a_words_to_l_words}
The size of $A\left(\word{w},\vectorise{n}\right)$ equals $$|L(\word{w},\vectorise{n})| -  \sum\limits_{i \in [d]} \sum\limits_{l | n_i} \begin{cases}
0 & l = n_i\\
\left(\prod\limits_{t = i + 1}^{d - 1} -\mu(n_t)\right)\left(-\mu\left(\frac{n_i}{l}\right)\right)|A(\word{w},n_1,n_2,\hdots,n_{i - 1},l)| \cdot H(i,l,\vectorise{n},d) & l < n_i
\end{cases}
$$
\end{lemma}

\begin{proof}
Following the arguments given in Lemma \ref{lem:atransaltional_to_lyndon}, observe that any Lyndon word in $L(\word{w},\vectorise{n})$ is either be atranslational, or of the form $\word{a} : \Angle{\word{a}}_g : \hdots : \Angle{\word{a}}_{g^{t - 1}}$.
In the latter case, let $l = |\word{a}|_d$.
Note that $\word{a}$ must be either in $A(\word{w}_{[1,l]}, n_1, n_2, \hdots, n_{d - 1}, l)$, if $l > 1$ or $L(\word{w}_{1})$ if $l = 1$.
Repeating the same arguments as in Lemma \ref{lem:atransaltional_to_lyndon} allows the size of $A\left(\word{w}, \vectorise{n}\right)$ to be written as:

$$|L(\word{w},\vectorise{n})| -  \sum\limits_{i \in [d]} \sum\limits_{l | n_i} \begin{cases}
0 & l = n_i\\
\left(\prod\limits_{t = i + 1}^{d - 1} -\mu(n_t)\right)\left(-\mu\left(\frac{n_i}{l}\right)\right)|A(\word{w},n_1,n_2,\hdots,n_{i - 1},l)| \cdot H(i,l,\vectorise{n},d) & l < n_i
\end{cases}
$$
\end{proof}

\begin{lemma}
\label{lem:a_words_to_a_rank}
The rank $RA(\word{w},\vectorise{n}) =  \frac{1}{N}|A\left(\word{w},\vectorise{n}\right)|$, where $N = n_1 \cdot n_2 \cdot \hdots \cdot n_d$.
\end{lemma}

\begin{proof}
Observe that any atranslational necklace of size $\vectorise{n}$ has exactly $N$ representations.
Therefore the number of atranslational necklaces smaller than $\word{w}$ is $\frac{1}{N} |A\left(\word{w},\vectorise{n}\right)|$.
Hence $$RA(\word{w},\vectorise{n}) = \frac{1}{N}|A\left(\word{w},\vectorise{n}\right)|.$$
\end{proof}

\noindent
In order to use the rank $RA(\word{w},\vectorise{n})$ to determine the rank $RL(\word{w},\vectorise{n})$, it is necessary to consider the special case where $\word{w}$ is a translational, aperiodic word.
Let $\word{u} \in \mathcal{A}_q^{g_1,g_2,\hdots,g_{i - 1},l}$ be the translational period of $\word{w}$, where $g \in \mathbf{G}\left(\frac{n_i}{l},(n_1, n_2, \hdots, n_i)\right)$ is the smallest translation such that $\word{w} = \Angle{\word{w}}_g$ and $i \in [d]$ is the smallest index such that $g_{j} = 1$ for all $j \in [i + 1, d]$.
Further, let $\word{u}[j]$ be the Lyndon word of size $(g_1, g_2, \hdots, g_{i - 1}, \frac{n_i}{l}, n_{i + 1}, \hdots, n_{j})$ such that $\word{u}[j]_{\vectorise{i}} = \word{w}_{\vectorise{i}}$ for every $\vectorise{i} \in [g_1, g_2, \hdots, g_{i - 1}, \frac{n_i}{g_i}, n_{i + 1}, \hdots, n_{j}]$ $j \in [i + 1, d]$.
Note that $\word{u}[j]$ can be written as $\word{u}[j] = \word{u}[j - 1] : \Angle{\word{u}[j - 1]}_{r_j} : \hdots : \Angle{\word{u}[j - 1]}_{r_j^{n_{j} - 1}}$, for some $r_j \in \mathbf{G}(l_j,(n_1,n_2,\hdots,n_j))$ where $l_j = 1$ if $j > i$ and $l$ if $j = i$.
Observe that the number of Lyndon words with a translational period of $\word{u}$ with size $\vectorise{n}$ that are smaller than $\word{w}$ is equal to the sum of the number of translations in $\mathbf{G}(l_j, n_1, n_2, \hdots, n_j)$ smaller than $r_j$, multiplied by $H(i,l_i,\vectorise{n})$ for every $j \in [i,d]$.
For simplicity, Let $S(g,l,(n_1,n_2,\hdots,n_j))$ return the number of translations in $\mathbf{G}(l,(n_1,n_2,\hdots,n_j))$ smaller than $g$.
Further, let $U(\word{w})$ return either:
\begin{itemize}
    \item $0$ if $\word{w}$ is either atranslational or periodic.
    \item $\sum\limits_{j = i}^d \begin{cases}
        S(r_j, l, (n_1, n_2, \hdots, n_j)) & j = i\\
        S(r_j, 1, (n_1, n_2, \hdots, n_j)) & otherwise.
    \end{cases}$ if $\word{w}$ is a Lyndon word with a translational period of $g$.
\end{itemize}

\noindent
%with size $(g_1,g_2,\hdots,g_{d - 1}, \frac{n_i}{l}, 1,1,\hdots,1)$ for  $g \in \mathbf{G}(l,n_1,n_2,\hdots,n_i)$ and $i \in [d]$.
% Let $\word{u}$ be the word of size $g$ such that $\word{u}_{\vectorise{i}} = \word{w}_{\vectorise{i}}$.
% Further, let $\word{u}[j]$ be the Lyndon word of size $(g_1, g_2, \hdots, g_{i - 1}, \frac{n_i}{l}, n_{i + 1}, \hdots, n_{j})$ such that $\word{u}[j]_{\vectorise{i}} = \word{w}_{\vectorise{i}}$ for $j \in [i + 1, d]$.
% Note that $\word{u}[j]$ can be written as $\word{u}[j] : \word{u}[j - 1] : \Angle{\word{u}[j - 1]}_{r_j} : \hdots : \Angle{\word{u}[j - 1]}_{r_j^{n_{j} - 1}}$, for some $r_j \in \mathbf{G}(l_j,(n_1,n_2,\hdots,n_j))$ where $l_j = 1$ if $j > i$ and $1$ otherwise.
% Observe that the number of Lyndon words with a translational period of $\word{u}$ with size $\vectorise{n}$ that are smaller than $\word{w}$ is equal to the sum of the number of translations in $\mathbf{G}(l_j, n_1, n_2, \hdots, n_j)$ multiplied by $H(i,l_i,\vectorise{n})$.
Let $S(g,l,(n_1,n_2,\hdots,n_j))$ return the number of translations in $\mathbf{G}(l,(n_1,n_2,\hdots,n_j))$ smaller than $g$.
% To this end let $U(\word{w})$ return either:
% \begin{itemize}
%     \item $0$ if $\word{w}$ is either atranslational or periodic.
%     \item $\sum\limits_{j = i}^d \begin{cases}
%         S(r_j, l, (n_1, n_2, \hdots, n_j)) & j = i\\
%         S(r_j, 1, (n_1, n_2, \hdots, n_j)) & otherwise.
%     \end{cases}$ if $\word{w}$ is a Lyndon word with a translational period of $g$.
% \end{itemize}
Using $U(\word{w})$, the number of Lyndon words can be computed from $RA(\word{w},n_1,n_2,\hdots,n_d)$ as follows.

\begin{lemma}
\label{lem:a_rank_to_l_rank}
The rank
$$RL(\word{w},\vectorise{n}) = RA(\word{w},\vectorise{n}) + U(\word{w}) + $$

$$
\sum\limits_{i \in [d]} \sum\limits_{l | n_i} \begin{cases}
0 & l = n_i\\
\left(\prod\limits_{t = i + 1}^{d - 1} -\mu(n_t)\right)\left(-\mu\left(\frac{n_i}{l}\right)\right) |RA(\word{w}_{[1,l]},n_1,n_2,\hdots,n_{i - 1})| \cdot H(i,l,\vectorise{n},d) & 1 < l < n_d
\end{cases}
$$
\end{lemma}

\begin{proof}
Note that every necklace smaller than $\word{w}$ is either atranslational, in which case it is counted by $RA(\word{w},\vectorise{n})$, or is translational.
In the latter case following Lemma \ref{lem:atransaltional_to_lyndon} for each necklace counted by $RA(\word{w}_{[1,l]},n_1,n_2,\hdots,n_{d - 1},l)$, there are $H(i,l,\vectorise{n})$ translational necklace counted by $RL(\word{w},\vectorise{n})$.
Further, if $\word{w}$ is a translational Lyndon word of the form $\word{v} : \Angle{\word{v}}_{g} : \hdots : \Angle{\word{v}}_g$, then there are there are $U(\word{w})$ Lyndon words of the form $\word{v} : \Angle{\word{v}}_{g} : \hdots : \Angle{\word{v}}_g$ where $\word{v}_{\vectorise{i}} = \word{w}_{\vectorise{i}}$ for every $\vectorise{i} \in [|\word{v}|]$.
Following Lemma \ref{lem:atransaltional_to_lyndon} $RL(\word{w},\vectorise{n})$ is counted in terms of $RA(\word{w},n_1,n_2,\hdots,n_{d - 1},l)$ as:

$$RL(\word{w},\vectorise{n}) = RA(\word{w},\vectorise{n}) + U(\word{w}) + $$

$$
\sum\limits_{i \in [d]} \sum\limits_{l | n_i} \begin{cases}
0 & l = n_i\\
\left(\prod\limits_{t = i + 1}^{d - 1} -\mu(n_t)\right)\left(-\mu\left(\frac{n_i}{l}\right)\right) |RA(\word{w}_{[1,l]},n_1,n_2,\hdots,n_{i - 1})| \cdot H(i,l,\vectorise{n},d) & 1 < l < n_d
\end{cases}
$$\end{proof}

\begin{lemma}
\label{lem:l_rank_to_T_rank}
The rank $RN(\word{w},\vectorise{n}) = \sum\limits_{f_1 | n_1} \sum\limits_{f_2 |n_2} \hdots \sum\limits_{f_d | n_d} RL(\word{w},\vectorise{f})$.
\end{lemma}

\begin{proof}
Observe that every necklace counted by $RN(\word{w},\vectorise{n})$ has a period of $\vectorise{m}$ where $m_i$ is a factor of $|\word{w}|_i$ for every $i \in 1\hdots d$.
As $RL(\word{w},\vectorise{f})$ counts the rank among aperiodic necklaces of size $\vectorise{f} = (f_1, f_2, \hdots, f_d)$, the rank among necklaces is given by:

$$RN(\word{w},\vectorise{n}) = \sum\limits_{f_1 | n_1} \sum\limits_{f_2 |n_2} \hdots \sum\limits_{f_d | n_d} RL(\word{w},\vectorise{f})$$
\end{proof}

\subsection{Computational Tools}
\label{subsec:comp_tool}

Following the theoretical tools provided in Section \ref{sec:ranking_tool}, the remaining problem is to compute the size of $T(\word{w},\vectorise{f})$.
To this end, $T(\word{w},\vectorise{f})$ is partitioned into the sets $\mathbf{B}(\word{w},g_d,j,\vectorise{f})$ such that $\mathbf{B}(\word{w},g_d,j,\vectorise{f})$ contains every word $\word{v}\in T(\word{w},\vectorise{f})$ where:
\begin{itemize}
    \item $g_d$ is the smallest translation in dimension $d$ of $\word{v}$ such that $\Angle{\word{v}}_{(\theta_1, \theta_2,\hdots, \theta_{d - 1}, g_d)} < \word{w}$ for some translation $\theta \in Z_{(f_1,f_2,\hdots,f_{d - 1})}$.
    \item $j$ is the largest value such that $(\Angle{\word{v}'}_{(\theta_1, \theta_2,\hdots, \theta_{d - 1}, g_d)})_{[1,j]} = \word{w}_{[1,j]}'$.
\end{itemize}

\noindent
Observe that $|T(\word{w},\vectorise{f})| = \sum\limits_{g_d \in [n_d]} \sum\limits_{j \in [0, n_d - 1]} |\mathbf{B}(\word{w},g_d,j,\vectorise{f})|$.
To compute the size of $\mathbf{B}(\word{w},g_d,j,\vectorise{f})$, there are two cases to consider based on the values of $g_d$ and $j$.
The following propositions formalise the structure of each word $\word{v} \in \mathbf{B}(\word{w},g_d,j,\vectorise{f})$.

\begin{proposition}
\label{prop:ranking_gd_j_less}
Given any word $\word{v} \in \mathbf{B}(\word{w},g_d,j,\vectorise{f})$, where $g_d + j \leq f_d$, $\word{v} = \word{a} : \Angle{\word{w}_{[1,j]}:\word{b}}_{\theta} : \word{c}$, where:
\begin{itemize}
    \item $\word{a} \in \Sigma^{(f_1,f_2,\hdots,f_{d - 1}, g_d)}$ is word such that for every $i \in [g_d]$ and translation $r \in Z_{(f_1,f_2,\hdots,g_d)}$, $\Angle{\word{a}_{[i,g_d]}}_r \geq \word{w}$ and $\Angle{\word{a}_{[i,g_d]} : \word{w}_{[1,j]}}_r > \word{w}_{[1,g_d - i + j]}$.
    \item $\word{b}$ is some word of size $(f_1,f_2,\hdots,f_{d - 1})$ that is smaller than $\word{w}_{j + 1}$.
    % Further if $\Angle{\word{a}_{[i,g_d]} : \word{w}_{[1,j]}}_r = \word{w}_{[1,g_d - i + j]}$ then $\word{b} \geq \word{w}_{[g_d - i + j + 1]}$.
    % % \item $\word{c}$ is a word of size $(f_1,f_2,\hdots,f_{d - 1}, t)$ for some $t \in [0,f_d - (g_d + j + 1)]$ where $t = 0$ if and only if $\Angle{\word{a}_{[i,g_d]} : \word{w}_{[1,j]} : \word{b}}_r > \word{w}_{[1,g_d - i + j + 1]}$.
    % If $t > 0$ then $\left(\Angle{\word{a}_{[i,g_d]} : \word{w}_{[1,j]} : \word{b} : \word{c}}_r\right)_{[k,g_d + j + 1 + t]} > \word{w}_{[1,g_d + j + 1 + t - k]}$ for every $k \in [g_d + j + 1 + t]$ and $r \in Z_{f_1, f_2, \hdots, f_{d - 1}}$.
    \item $\theta$ is a translation in the set $\mathbf{\Theta} = \{r \in Z_{(f_1,f_2, \hdots f_{d - 1})} : \nexists s \in Z_{(f_1,f_2, \hdots f_{d - 1})}$ where $s < r$ and $\Angle{\word{w}_{[1,j]}}_r = \Angle{\word{w}_{[1,j]}}_s\}$.
    \item $\word{c}$ is an unrestricted word of size $(f_1,f_2,\hdots,f_{d - 1},f_d - (g_d + j + 1 ))$.
\end{itemize}
\end{proposition}

\begin{proof}
Note that if there exists some subword $\Angle{\word{v}_{[i,g_d]}} < \word{w}_{[1,j - i]}$ where $i \in [g_d - 1]$, then there exists some translation $t$ smaller than $g_d$ such that $\Angle{\word{v}}_t < \word{w}$, contradicting the original assumption.
Therefore, the prefix of $\word{v}$ of length $g_d$, $\word{a} = \word{v}_{[1,g_d]}$ must satisfy the property that $\Angle{\word{a}_{[i,g_d]}} \geq \word{v}_{[1,g_d - i]}$ for every $i \in [g_d - 1]$.
Additionally, for $\Angle{\word{v}}_g$ to be the smallest translation such that $\Angle{\word{v}}_g < \word{w}$ while sharing a prefix with $\word{w}$ of length $i$, the value of $\word{b}$ must be less than $\word{w}_{j + 1}$.
Similarly, if $\Angle{\word{a}_{[i,g_d]}} : \Angle{\word{w}_{[1,j]}}_r = \word{w}_{[1,g_d - i + j]}$ then $\Angle{\word{b}}_r \geq \word{w}_{g_d - i + j + 1}$.
Observe that for $\Angle{\word{a}_{[i,g_d]}} : \Angle{\word{w}_{[1,j]}}_r = \word{w}_{[1,g_d - i + j]}$ to hold, $\word{w}_{[1,j]}$ must equal $\Angle{\word{w}_{[g_d, g_d + j - i]}}_r$.
Therefore, $\word{w}_{j + 1} \leq \Angle{\word{w}_{g_d + j - i + 1}}_r$ as otherwise $\word{w}$ would not be the canonical representation of a necklace.
Therefore as $\word{w}_{j + 1} < \Angle{\word{w}_{g_d + j - i + 1}}_r$ hence if $\word{b} < \word{w}_{j + 1}$ then $\Angle{\word{b}}_r < \word{w}_{g_d + j - i + 1}$.
Further every suffix of $\word{a}$ must be strictly greater than the prefix of $\word{w}$ of the same length.
Finally, the suffix of $\word{v}$ is unconstrained.
\end{proof}

\begin{proposition}
\label{prop:ranking_gd_j_more}
Given any word $\word{v} \in \mathbf{B}(\word{w},g_d,j,\vectorise{f})$ where $g_d + j > f_d$, $\word{v} = \Angle{\word{w}_{[j + g_d - f_d,j]} : \word{b}}_{\theta} : \word{a} : \Angle{\word{w}_{[1,j + g_d - f_d]}}_{\theta}$ where:
\begin{itemize}
    \item $s$ is the longest suffix of $\word{w}_{[j + g_d - f_d,j]}$ such that $\Angle{\word{w}_{[j - s,j]}}_r = \word{w}_{[1,s]}$ for some translation $s \in Z_{f_1,f_2, \hdots, f_{d - 1}}$.
    \item $\word{a}$ is a $(f_1, f_2, \hdots, f_{d - 1}, f_d - (j + 1))$ dimensional word for which there exists no translation $r \in Z_{(f_1,f_2,\hdots,f_{d - 1}, g_d)}$ such that $\Angle{\word{a}}_r < \word{w}_{[1,g_d]}$.
    \item $\word{b}$ is some word of size $(f_1,f_2,\hdots,f_{d - 1})$ that is smaller than $\word{w}_{j + 1}$, and further $\Angle{\word{b}}_r \geq \word{w}_{s + 1}$.
    \item $\theta$ is a translation in the set $\mathbf{\Theta} = \{r \in Z_{(f_1,f_2, \hdots f_{d - 1})} : \nexists s \in Z_{(f_1,f_2, \hdots f_{d - 1})}$ where $s < r$ and $\Angle{\word{w}_{[1,j]}}_r = \Angle{\word{w}_{[1,j]}}_s\}$.
\end{itemize}
\end{proposition}

\begin{proof}
For $g$ to be the smallest translation such that $\Angle{\word{v}}_g < \word{w}$ while $j$ is the length of the longest prefix such that $\left(\Angle{\word{v}}_g\right)_{[1,j]} = \word{w}_{[1,j]}$, $\word{v}_{g_d + j - f_d + 1} < \word{w}_{j + 1}$.
Further, to ensure that no translation smaller than $g_d + j - f_d$ is smaller than $\word{w}$, $\Angle{\word{b}}_{r}$ must be no less than $\word{w}_{s + 1}$.
Additionally, the subword $\word{v}_{[g_d + j + 2 - f_d, g_d]}$ must satisfy the property that $\Angle{\word{v}_{[i, g_d]}}_r > \word{w}_{[1,g_d + (f_d - j - i)]}$ for every $i \in [g_d + j + 2 - f_d, g_d]$ and $r \in Z_{f_1, f_2, \hdots,f_{d - 1}}$.
Further, if $\Angle{\word{b}}_r = \word{w}_{s + 1}$ then for every $i \in [g_d + j - f_d - s, g_d]$ and $r \in Z_{f_1, f_2, \hdots,f_{d - 1}}$, $\Angle{\word{v}_{[i, g_d]}}_r > \word{w}_{[1,g_d + (f_d - j + s - i)]}$.
\end{proof}

\noindent
Using Propositions \ref{prop:ranking_gd_j_less} and \ref{prop:ranking_gd_j_more} as a basis, the problem of computing the size of $\mathbf{B}(\word{w},g_d,j,\vectorise{f})$ can be split into two cases.

\noindent
\paragraph*{Case 1:} $g_d + j < f_d$.
Following Proposition \ref{prop:ranking_gd_j_less}, every word $\word{v} \in \mathbf{B}(\word{w},g_d,j,\vectorise{f})$ can be written as $\word{a} : \Angle{\word{w}_{[1,j]}:\word{b}}_{\theta} : \word{c}$ where:
\begin{itemize}
    \item $\word{a} \in \Sigma^{(f_1,f_2,\hdots,f_{d - 1}, g_d)}$ is word such that for every $i \in [g_d]$ and translation $r \in Z_{(f_1,f_2,\hdots,g_d)}$, $\Angle{\word{a}_{[i,g_d]}}_r \geq \word{w}$ and $\Angle{\word{a}_{[i,g_d]} : \word{w}_{[1,j]}}_r > \word{w}_{[1,g_d - i + j]}$.
    \item $\word{b}$ is some word of size $(f_1,f_2,\hdots,f_{d - 1})$ that is smaller than $\word{w}_{j + 1}$.
    \item $\theta$ is a translation in the set $\mathbf{\Theta} = \{r \in Z_{(f_1,f_2, \hdots f_{d - 1})} : \nexists s \in Z_{(f_1,f_2, \hdots f_{d - 1})}$ where $s < r$ and $\Angle{\word{w}_{[1,j]}}_r = \Angle{\word{w}_{[1,j]}}_s\}$.
    \item $\word{c}$ is an unrestricted word of size $(f_1,f_2,\hdots,f_{d - 1},f_d - (g_d + j + 1 ))$.
\end{itemize}

% \begin{itemize}
%     \item $\word{a}$ is a $(f_1,f_2,\hdots,f_{d - 1}, g_d)$ dimensional word for which there exists no translation $r \in Z_{(f_1,f_2,\hdots,g_d)}$ such that $(\Angle{\word{a}}_r)_{[1,g_d - r_d]} < \word{w}_{[1,g_d - r_d]}$.
%     \item $\word{b}$ is some word of size $(f_1,f_2,\hdots,f_{d - 1})$ that is smaller than $\word{w}_{j + 1}$.
%     \item $\theta$ is some translation in $Z_{(f_1, f_2, \hdots, f_{d - 1})}$.
%     \item $\word{c}$ is an unrestricted word of size $(f_1,f_2,\hdots,f_{d - 1},f_d - (g_d + j + 1))$.
% \end{itemize}

\noindent
The main challenge for computing the size of $\mathbf{B}(\word{w}, g_d,j, \vectorise{f})$ is due to calculating the number of possible values of $\word{a}$.
To this end a new set $\beta(\word{w},i,j, f_1, f_2,\hdots,f_{d - 1})$ is introduced containing every word $\word{u}$ where:
\begin{itemize}
    \item The size of $\word{u}$ are $(f_1,f_2,\hdots,f_{d - 1},i)$.
    \item There exists no translation $h \in Z((f_1,f_2,\hdots,f_{d - 1}))$ where $\Angle{\word{u}_{[i - l,i]}}_h \leq \word{w}_{[1,l]}$.
    \item The first $j$ slices of $\word{u}$ are equal to the first $j$ slices of $\word{w}$, i.e. $\word{u}_{[1,j]} =\word{w}_{[1,j]}$.
\end{itemize}
When it is clear from context $\beta(\word{w},i,j,f_1,f_2,\hdots,f_{d - 1})$ is denoted $\beta(\word{w},i,j, \vectorise{f})$.
A method to compute the size of $\beta(\word{w},i,j, \vectorise{f})$ is given in Subsection \ref{subsec:beta}.
Using $|\beta(\word{w},i,j, \vectorise{f})|$ as a black box, the number of possible values of $\word{a}$ is $|\beta(\word{w},i,j, \vectorise{f})|$.
Similarly, the number of possible values of $\word{b}$ is given by $q^{f_1\cdot f_2\cdot \hdots \cdot f_{d - 1}} - |\beta(\word{w}_{j + 1},1,0, \vectorise{f})| - 1$.
The number of possible values of $\theta$ is equal to the size of the set $\mathbf{\Theta} = \{r \in Z_{\vectorise{f}} : \nexists s \in Z_{\vectorise{f}}$ where $s < r$ and $\Angle{\word{w}}_r = \Angle{\word{w}}_s\}$.
Finally, the number of values of $\word{c}$ is given by $q^{f_1\cdot f_2\cdot \hdots\cdot f_{d - 1} \cdot (f_d - (g_d + j + 1))}$.
Therefore the size of $\mathbf{B}(\word{w},g,j,\vectorise{f})$ when $g_d + j < n_d$ is given by:

\[
|\beta(\word{w},g_d,0,\vectorise{f})| \cdot (q^{f_1\cdot f_2\cdot \hdots \cdot f_{d - 1}} - |\beta(\word{w}_{j + 1},1,0,\vectorise{f})| - 1) \cdot |\mathbf{\Theta}|\cdot q^{f_1\cdot f_2\cdot \hdots\cdot f_{d - 1} \cdot (f_d - (g_d + j + 1))}
\]

% This gives the total number of words of the form $\word{a} : \Angle{((\word{w}_{[1,j]}:\word{b})}_{\theta}) : \word{c}$ as $|\beta(\word{w},g_d,0)| \cdot (q^{n_1\cdot n_2\cdot \hdots \cdot n_{d - 1}} - |\beta(\word{w}_{j + 1},1,0)| - 1) \cdot q^{n_1\cdot n_2\cdot \hdots\cdot n_{d - 1} \cdot (n_d - (g_d + j + 1))}$.

\noindent
\paragraph*{Case 2:} $g_d + j > f_d$.
In this case every word $\word{v} \in \mathbf{B}(\word{w},g_d,j,\vectorise{f})$ can be written as $\Angle{\word{w}_{[j + g_d - f_d,j]} : \word{b}}_{\theta} : \word{a} : \Angle{\word{w}_{[1,j + g_d - f_d]}}_{\theta}$ where:
\begin{itemize}
    \item $\word{a}$ is a $d$-dimensional word of size $(f_1, f_2, \hdots, f_{d - 1}, f_d - (j + 1))$ for which there exists no translation $r \in Z_{(f_1,f_2,\hdots,f_{d - 1}, g_d)}$ such that $\Angle{\word{a}}_r < \word{w}_{[1,g_d]}$.
    \item $\word{b}$ is some word of size $(f_1,f_2,\hdots,f_{d - 1})$ that is smaller than $\word{w}_{j + 1}$.
    \item $\theta$ is a translation in the set $\mathbf{\Theta} = \{r \in Z_{(f_1,f_2, \hdots f_{d - 1})} : \nexists s \in Z_{(f_1,f_2, \hdots f_{d - 1})}$ where $s < r$ and $\Angle{\word{w}_{[1,j]}}_r = \Angle{\word{w}_{[1,j]}}_s\}$.
\end{itemize}
The number of possible values of $\theta$ is equal to the size of the set $\mathbf{\Theta}$ as in Case 1.
The number of possible values of $\word{b}$ in this case is somewhat more complicated than in Case 1.
Let $t$ be the length of the longest suffix of $\word{w}_{[j + g_d - f_d,j]}$ such that $\word{w}_{[j - t,j]} = \word{w}_{[1,t]}$. %, for some translation $\psi \in Z_{\{f_1,f_2,\hdots,f_{d-1},f_d-g_d\}}$.
To avoid $\Angle{\word{v}}_{\psi}$, for some $\psi \in Z_{(f_1,f_2,\hdots,f_{d-1},f_d - g_d)}$, being smaller than $\word{w}$, $\word{b}$ must be greater than or equal to $\word{w}_{t + 1}$.
Note that the number of words greater than $\word{w}_{t + 1}$ is given by $\beta(\word{w}_{t + 1},1,0,\vectorise{f})$.
Therefore the number of possible values of $\word{b}$ as $(q^{f_1\cdot f_2\cdot \hdots\cdot f_{d - 1} \cdot (f_d - (g_d + j + 1))} - \beta(\word{w}_{j + 1},1,0) - 1) - (q^{n_1\cdot n_2\cdot \hdots\cdot f_{d - 1} \cdot (f_d - (g_d + j + 1))} - \beta(\word{w}_{t + 1},1,0,\vectorise{f})) = \beta(\word{w}_{t + 1},1,0,\vectorise{f}) - \beta(\word{w}_{j + 1},1,0,\vectorise{f}) + 1$.
If $\word{b} = \word{w}_{t + 1}$, the number of possible values of $\word{a}$ is given by $|\beta(\word{w},f_d + t - j,t + 1,\vectorise{f})|$.
Otherwise the number of possible values of $\word{a}$ is given by $|\beta(\word{w},f_d- j - 1,0,\vectorise{f})|$.
Therefore the total number of words of the form $\Angle{\word{w}_{[j + g_d - f_d,j]} : \word{b}}_{\theta} : \word{a} : \Angle{\word{w}_{[1,j + g_d - f_d]}}_{\theta}$ is:

$$|\beta(\word{w},f_d + t - j,t + 1, \vectorise{f})| + \left(|\beta(\word{w}_{t + 1},1,0,\vectorise{f})| - |\beta(\word{w}_{j + 1},1,0,\vectorise{f})|\right)\cdot |\beta(\word{w},f_d- j - 1,0,\vectorise{f})|\cdot|\mathbf{\Theta}|$$
% Combining both of these cases, the size of $T(\word{w})$ equals:

% \[
% \sum\limits_{j \in [n_d]}TP(\word{w}_{[1,j]})\cdot \sum\limits_{t \in [n_d]} \begin{cases}
% |\beta(\word{w},t,0)| \cdot (q^{N/n_d} - |\beta(\word{w}_{j + 1},1,0)| - 1) \cdot q^{(N/n_d)(n_d - (t + j + 1))} & j + t \leq n_d\\
% \left(\begin{split}
% &|\beta(\word{w},n_d + t - j,t + 1)| + \\
% &(|\beta(\word{w}_{t + 1},1,0)| - |\beta(\word{w}_{j + 1},1,0)|)\cdot |\beta(\word{w},n_d- j - 1,0)| 
% \end{split}\right)& j + t > n_d
% \end{cases}
% \]

\subsubsection{Computing the Number of Prefixes Greater than \texorpdfstring{$\word{w}$}{w}}
\label{subsec:beta}

% Following the previous section, the size of the set $\alpha(\word{w},i,j,\vectorise{f})$ must be computed.
% Recall from the previous section that $\alpha(\word{w},i,j,\vectorise{f})$ contains every word $\word{u}$ where:

% \begin{itemize}
%     \item $\left(\Angle{\word{u} : \word{w}_{[1,j]} \geq \word{w}_{[1,i + j]}}\right)$.
% \end{itemize}

% For some $\word{u} \in \alpha(\word{w},i,j,\vectorise{f})$, 

% \begin{lemma}
% Let $\word{u} \in \alpha(\word{w},i,j,\vectorise{f})$.
% Then $\word{u}$ has the form $\word{u} = \word{v} : \Angle{\word{w}_{[1,k]}}_r$ where $k \in [0,i], r \in Z_{f_1,f_2,\hdots,f_{d - 1}}, \word{v} \in \alpha(\word{w},i,j,\vectorise{f})$ and for every $l \in [j + k], \left(\word{w}_{[1,k]} : \Angle{\word{w}_{[1,k]}}_{f_1 - r_1, f_2 - r_2, \hdots, f_{d - 1} - r_{d - 1}}\right)_{[l, k + j]} > \word{w}_{[1,k + j - l]}$.
% \end{lemma}

Following Propositions \ref{prop:ranking_gd_j_less} and \ref{prop:ranking_gd_j_more}, in order to compute the size of $T(\word{w})$ require the size of the set $\beta(\word{w},i,j,\vectorise{f})$ to be computed.
Recall that $\beta(\word{w},i,j,\vectorise{f})$ contains every word $\word{v} \in \Sigma^{f_1,f_2,\hdots,f_{|f|}, i}$ where:
\begin{itemize}
    \item For every translation $h \in Z_{\vectorise{f}}$ and index $k \in [i]$ $\Angle{\word{v}_{[i - k, i]} }_{h} > \word{w}$.
    \item The prefix of $\word{v}$ of length $j$ equals the prefix of $\word{w}$ of length $j$, i.e. $\word{v}_{[1,j]} = \word{w}_{[1,j]}$.
\end{itemize}

\noindent
% Let $\word{v} \in \beta(\word{w},i,j,\vectorise{f})$.
The value of $\beta(\word{w},i,j,\vectorise{f})$ is given defined recursively, noting that any suffix $\word{u} = \word{v}_{[i - k, i]} \in \beta(\word{w},i,j,\vectorise{f})$ must also belong to $\beta(\word{w},i',j',\vectorise{f})$ for some $i',j' \in [i]$.
Additionally, observe that when $i = j$ then either $|\beta(\word{w},i,j,\vectorise{f})| = 0$, if $i > 0$, or $|\beta(\word{w},i,j,\vectorise{f})| = 1$ if $i = 0$.

This leaves the problem of partitioning $\beta(\word{w},i,j,\vectorise{f})$ into sets based of the $j + 1^{th}$ slice.
The key observation is that given some word $\word{v} \in \beta(\word{w},i,j,\vectorise{f})$ where $\word{v}_{j + 1} = \word{w}_{j + 1}$, $\word{v}$ also belongs to $\beta(\word{w},i,j + 1,\vectorise{f})$.
On the other hand, given some word $\word{u} \in \beta(\word{w},i,j,\vectorise{f})$ where $\word{u}_{j + 1} > \word{w}_{j + 1}$, the suffix $\word{u}_{[j + 1, i]}$ belongs to the set $\beta(\word{w},i - j - 1,0,\vectorise{f})$.
The following Lemma strengthens this property by showing that given any word $\word{u} \in \beta(\word{w},i - j - 1,0,\vectorise{f})$, $\word{w}_{[1,j]} : \word{u} \in \beta(\word{w},i,j,\vectorise{f})$.

\begin{lemma}
\label{lem:beta_cartesian}
Given any word $\word{u} \in \beta(\word{w},i,0,\vectorise{f})$ and word $\word{v} > \word{w}_{j + 1}$, $\word{w}_{[1,j]} : \word{v} : \word{u} \in \beta(\word{w},i + j + 1,j,\vectorise{f})$.
\end{lemma}

\begin{proof}
% Assume for the sake of contradiction that there exists some word $\word{u} \in \beta(\word{w},i,0,\vectorise{f})$ and word $\word{v} > \word{w}_{j + 1}$ such that $\word{w}_{[1,j]} : \word{v} : \word{u} \notin \beta(\word{w},i + j + 1,j,\vectorise{f})$.
As $\word{u} \in \beta(\word{w},i,0,\vectorise{f})$, any suffix of $\word{w}_{[1,j]} : \word{v} : \word{u}$ starting at index $t \geq j + 1$ must satisfy the condition that $\Angle{(\word{w}_{[1,j]} : \word{v} : \word{u})_{[1,t]}}_h > \word{w}_{[1,i + j + 1 - t]}$ for every $h \in Z_{\vectorise{f}}$.
Similarly, for every $h \in Z_{\vectorise{f}}$, $\Angle{\word{v} : \word{u})}_h > \word{w}_{[1,i + 1]}$ as $\word{v} > \word{w}_{j + 1} \geq \word{w}_1$.
Further, for any index $t \in [1,j]$ as $\Angle{\word{w}_{[t,j]} : \word{v}}_h > \word{w}_{[1,j + 1 - t]}$ for every $h \in Z_{\vectorise{f}}$, hence $\Angle{\word{w}_{[t,j]} : \word{v} : \word{u}}_h > \word{w}_{[1,i + j + 1 - t]}$.
Therefore  $\word{w}_{[1,j]} : \word{v} : \word{u} \in \beta(\word{w},i + j + 1,j,\vectorise{f})$.
\end{proof}

\noindent
Following Lemma \ref{lem:beta_cartesian}, it is possible to define the size of $\beta(\word{w},i,j,\vectorise{f})$ recursively.
Let $NS(\word{w},j,\vectorise{f})$ return the number of possible slices of size $(f_1, f_2, \hdots, f_{d - 1}$ that are greater than $\word{w}_{j + 1}$.
Using $NS(\word{w},j,\vectorise{f})$ as a black box, the size of $\beta(\word{w},i,j,\vectorise{f})$ can be computed as:

% Observe that if $v_{j + 1} > \word{w}_{j + 1}$, then for any translation $g \in Z((f_1,f_2,\hdots,f_{d - 1}, j + 1))$, $\word{v}_{[1,j + 1]}> \word{w}_{[1,j + 1]}$.
% Therefore the number of possible values of $\word{v}_{[j + 2,i]} = |\beta(\word{w}, i - j - 1, 0,\vectorise{f})|$.
% Similarly the number of values of $\word{v}$ where $\word{v}_{j + 1} = \word{w}_{j + 1}$ is $|\beta(\word{w},i,j + 1,\vectorise{f})|$.
% This allows the size of $\beta(\word{w},i,j,\vectorise{f})$ to be computed in a recursive manner.
% In the special case where $j = i$, there is either one word in $\beta(\word{w},i,j,\vectorise{f})$, if $j = 0$, or none if $j > 0$.
% Let $NS(\word{w},j,\vectorise{f})$ return the number of possible slices of size $(f_1, f_2, \hdots, f_{d - 1}$ that are greater than $\word{w}_{j + 1}$.
% Using $NS(\word{w},j,\vectorise{f})$ as a black box, the size of $\beta(\word{w},i,j,\vectorise{f})$ can be computed as:

\[
|\beta(\word{w},i,j,\vectorise{f})| = \begin{cases}
0 & i = j, j > 0\\
1 & i = j = 0\\
NS(\word{w},j,\vectorise{f}) \cdot |\beta(\word{w},i-j-1,0,\vectorise{f})| + |\beta(\word{w},i, j + 1,\vectorise{f})| & Otherwise.
\end{cases}
\]

\noindent
This leaves the problem of computing $NS(\word{w},j,\vectorise{f})$.
This is done by considering two cases.
First are the set of slices that belong to a necklace class greater than $\word{w}_{j + 1}$.
The number of such necklaces can be computed as $|\mathcal{N}_q^{(f_1,f_2,\hdots,f_{d - 1})}| - RN(\word{w}_j,f_1,f_2 \hdots, f_{d-1})$, i.e. the number of necklaces of size $(f_1,f_2,\hdots, f_{d - 1})$ minus the necklaces smaller than $\word{w}_{j}$.
To account for the number of possible translations of each necklace, it is easiest to use the sets of aperiodic words instead.
The number of such words are determined by counting the number of atranslational words of size $(f_1,f_2,\hdots,f_{i - 1}, h_i,1,\hdots,1)$ for every $i \in [d]$ and factor $h_i$ of $f_i$.
This rank is then multiplied by the number of possible translations, given by $f_1 \cdot f_2 \cdot \hdots \cdot f_{i - 1} \cdot h_i$, and $H(i,h,(f_1,f_2,\hdots,f_d),d)$ to account for the number of necklaces with a translational period in $T(\word{w},f_1,f_2\hdots,f_d)$.
The second case to consider are translations of $\word{w}_{j _ 1}$ greater than $TR(\word{w}_{j + 1})$.
This is given by $TP(\word{w}_{j + 1}) - TR(\word{w}_{j + 1})$.
This allows the number of necklaces greater than $\word{w}_j$ along with the number of translations of these necklaces to be counted as:

{ % \footnotesize
\[
NS(\word{w},j,\vectorise{f}) = (TP(\word{w}_{j + 1}) - TR(\word{w}_{j + 1})) +  \sum\limits_{i \in [d - 1]} \sum\limits_{h_i | f_i} RA(\word{w}_{j}, \vectorise{h[i]})) \cdot |\vectorise{h[i]}| \cdot H(i,h,\vectorise{f},d)
\]
}

\noindent
Where $\vectorise{h[i]} =  (f_1, \hdots, f_{i - 1}, h_i, \hdots, 1)$ and $|\vectorise{h[i]}| = f_1 \cdot f_2 \cdot \hdots \cdot f_{i - 1} \cdot h_i$.

\subsection{Complexity of ranking multidimensional necklaces}

The tools are now in place to show the complexity of our ranking algorithm.

\begin{theorem2}
The rank of a $d$-dimensional necklace in the set $\mathcal{N}_q^{\vectorise{n}}$ can be computed in \RankingComplexity time, where $N = \prod_{i = 1}^d n_i$.
\end{theorem2}

\begin{proof}
Lemmas \ref{lem:t_words_to_l_words}, \ref{lem:a_words_to_l_words}, \ref{lem:a_words_to_a_rank}, \ref{lem:a_rank_to_l_rank}, and \ref{lem:l_rank_to_T_rank} show that to rank $RN(\word{w})$, the first step is to compute the size of $T(\word{w},\vectorise{f})$.
Following Lemma \ref{lem:a_words_to_l_words}, to compute the size of $A(\word{w},\vectorise{f})$, the set $A(\word{w}_{[1,l]},f_1,f_2 \hdots, f_{d-1},l)$ must be computed for every factor $l$ of $f_d$, alongside the set $L(\word{w},\vectorise{f})$ and $L(\word{w}_1,f_1,f_2 \hdots, f_{d-1},l)$.
Note that this requires at most $\log_2(n_d)$ sets to be computed.
The size of the set $L(\word{w},f_1,f_2 \hdots, f_{d-1})$ can be computed by computing the size of $T(\word{w},h_1,h_2,\hdots,h_{d})$ where $h_i$ is a factor of $f_i$.
Therefore for $L(\word{w},f_1,f_2 \hdots, f_{d-1},l)$, the size of at most $\log_2(N)$ sets $T(\word{u},h_1,h_2 \hdots, h_{d})$ must be computed.

Following the above observations, $T(\word{w},\vectorise{n})$ can be computed by determining the cardinality of the set $\mathbf{B}(\word{w},g,j,n_1,n_2 \hdots, n_{d-1})$ using $n_d^2$ combinations of $j$ and $g$.
For each pair $j$ and $g$, the size of $\beta(\word{w},i,j,n_1,n_2 \hdots, n_{d-1})$ must be computed for some value of $i$.
This is done in a dynamic programming approach.
Starting with $i = j$, the size of $|\beta(\word{w},i,j,\vectorise{n})|$ is computed using the previously computed values as a basis.
As such, the size of $|\beta(\word{w},i,j,\vectorise{n})|$ for every pair $i$ and $j$ can be computed in $n_d^2$ time multiplied by the complexity of computing $NS(\word{w},j,\vectorise{n})$.
To compute $NS(\word{w},j,\vectorise{n})$, $d \cdot \frac{\log_2 N}{d} = \log_2 \frac{N}{n_d}$ words of size $d - 1$ must be ranked.

As there are $n_d^2$ values of $\beta(\word{w},i,j,\vectorise{n})$, and $\log_2(\frac{N}{n_d})$ words of size $d - 1$ must be ranked for each of the $n_d^2$ values of $\beta(\word{w},i,j,\vectorise{n})$, to precompute every value of $\beta(\word{w},i,j,\vectorise{n})$ $n_d^2 \cdot \log_2(\frac{N}{n_d})$ time is needed, multiplied by the cost of ranking a $d - 1$ word.
If $d = 2$, then the rank at this step can be computed in $O(n_1^2)$ time using existing algorithms due to Sawada and Williams \cite{Sawada2017}.
Hence the size of $\beta(\word{w},i,j,\vectorise{n})$ for every value of $i$ and $j$ can be computed in the two dimensional case in $O(n_d \cdot N \cdot \log_2(\frac{N}{n_d}) \cdot n_1^2) = O(N^2 \cdot \log_2(\frac{N}{n_d}))$ time.
To get the rank of a two dimensional word, a further $n_2^2$ time is needed to compute the size of $T(\word{w},\vectorise{n})$, with $\log_2(N)$ sets of $T(\word{w})$ to be computed.
Therefore the rank of a two dimensional word can be computed in $O(n_2^2 \cdot \log_2(N) N^2 \cdot \log_2(\frac{N}{n_d}))$.

Similarly in the three dimensional case, the set of all values of $\beta(\word{w},i,j,\vectorise{n})$ can be computed in $O(n^2_3 \cdot n_2^2 \cdot \log_2(N) \frac{N^2}{n_3^2} \cdot \log_2(n_1)) = O(N^2 \cdot n_2^2 \cdot \log_2(N) \cdot \log_2(n_1))$.
Thus the complexity of ranking a three dimensional word is $O(n_3^2 \cdot \log_2(N) \cdot N^2 \cdot n_2^2 \cdot \log_2(N) \cdot \log_2(n_1))$ time.
In the more general case, a total of $n_d^2 \cdot \log_2(N)$ words of dimension $d - 1$ must be ranked.
Using the two and three dimensional cases as a base, the total complexity of ranking a $d$ dimensional word is $O(\left(\prod_{i = 2}^{d} n_i^4 \cdot \log_2(n_i)\right) n_1^2) \leq \RankingComplexity$.% $O(\left(\prod_{i = 2}^{d} n_i^2 \cdot \log_2(\frac{N}{\prod\limits_{j \in [i + 1,d]} n_j})\right) n_1^2) = O\left(N^2 \cdot \log^d_2(N)\right)$.
\end{proof}

\subsection{Ranking Fixed Content Necklaces}

The same tools used in the unrestricted case are used in the fixed content case.
As before, the goal is to count the number of words of size $\vectorise{f}$ that belong to a necklace class smaller than the ranked word $\word{w}$, with the additional constraint that $F = f_1 \cdot f_2 \cdot \hdots \cdot f_d$ is a factor of $P_i$ for every $P_i \in \vectorise{P}$.
The main complexity is generalising the previous approach comes from the constraint on the content.
Let $\mathbf{T}(\word{w}, i, j, \vectorise{f}, t, \vectorise{Q})$ be the set of words of size $\vectorise{f}$ with fixed content $\vectorise{Q}$ belonging to a necklace class smaller than $\word{w}$.
As in the unconstrained case, this set is subdivided based on two values $g_d$ and $j$.
Formally, the set $\mathbf{B}(\word{w},\vectorise{f}, \vectorise{q}) \subseteq  \mathbf{T}(\word{w}, i, j, \vectorise{f}, t, \vectorise{Q})$ contain every word $\word{v} \in \mathbf{T}(\word{w}, \vectorise{f}, \vectorise{Q})$ where:

\begin{itemize}
    \item $h = (h_1,h_2,\hdots, h_{d - 1}, g_d)$ is the smallest translation such that $\Angle{\word{v}}_h < \word{w}$.
    \item $j$ is the largest value such that $(\Angle{\word{w}}_h)_{[1,j]} = \word{w}_{[1,j]}$.
    \item $\Parikh(\word{w}_{[1,j]}) + \vectorise{q} = \vectorise{Q}$.
\end{itemize}

\noindent
In order to compute the size of $\mathbf{B}(\word{w},\vectorise{f}, \vectorise{q})$, a generalisation of $\beta(\word{w},i,j,\vectorise{f})$ is needed.
More precisely, due to the constraint on the content, it is necessary not only to count the number of words for which every suffix is greater than $\word{w}$, as in $\beta(\word{w},i,j,\vectorise{f})$, but instead to count the number of such suffixes of the words in $\mathbf{B}(\word{w},\vectorise{f}, \vectorise{q})$ for each prefix of $\word{w}$.
To this end, let $\gamma(\word{w},i,j,\vectorise{f},\vectorise{q},t,l)$ return the number of triples $(\word{x}, \word{y}, \word{z})$ where:
\begin{itemize}
    \item $\word{x}$ is a word of size $(f_1, f_2, \hdots, f_{d - 1}, i)$ such that every suffix of $\word{x}$ belongs to a necklace class larger than the prefix of $\word{w}$ of the same length and $\word{x}_{[1,j]} = \word{w}_{[1,j]}$.
    \item $\word{y}$ is a word of size $(f_1, f_2, \hdots, f_{d - 1})$ such that $\word{y} < \word{w}_{t}$.
    \item $\word{z}$ has size $(f_1, f_2, \hdots, f_{d - 1}, l)$.
    \item $\Parikh(\word{x} : \word{y} : \word{z}) = \vectorise{q}$.
\end{itemize}

\noindent
As with $\beta(\word{w},i,j,\vectorise{f})$, the problem of computing $\gamma(\word{w},i,j,\vectorise{f},\vectorise{q},t,l)$ is solved recursively.
In effect, the problem is solved in three stages.
First, the number of possible values of $\word{x}$ are computed.
Secondly, for each value of $\word{x}$, the number of possible values of $\word{y}$ are computed.
Finally, the number of possible values of $\word{z}$ are computed using the remaining symbols.

To compute the number of values of $\word{x}$, observe that $\word{x}_{[j + 1,i]}$ is counted by either $\gamma(\word{w},i - j - 1,0,\vectorise{f},\vectorise{q} - P(\word{x}_{j + 1}),t,l)$, if $\word{x}_{j + 1} > \word{w}_{j + 1}$, or by $\gamma(\word{w},i,j + 1,\vectorise{f},\vectorise{q} - P(\word{w}_{j + 1}),t,l)$ if $\word{x}_{j + 1} = \word{w}_{j + 1}$.
In order to count the number of possible values of $\word{x}_1$ that are greater than $\word{w}_{j + 1}$, the same approach as in the unrestricted setting is used.
Let $V(\vectorise{q}, \vectorise{f})$ contain every Parikh vector $\vectorise{q}'$ where $\vectorise{q}'_i \leq \vectorise{q}_i$ and $\sum\limits_{i = 1}^q\vectorise{q}'_i = f_1 \cdot f_2 \cdot \hdots \cdot f_{d - 1}$.
Further let $X(\word{w}_{j + 1},\vectorise{f},\vectorise{q})$ return the number of values of $\word{x}_{j + 1}$ with a Parikh vector $\vectorise{q}$ that are greater than $\word{w}_{j + 1}$.
$X(\word{w}_{j + 1},\vectorise{f},\vectorise{q})$ is computed in a similar manner to $NS(\word{s},j,\word{f})$.
Formally:% $X(\word{w}_{j + 1},\vectorise{f},\vectorise{q})$ is equal to:

\[
X(\word{w}_{j + 1},\vectorise{f},\vectorise{q}) = \left(\begin{split}
    &(TP(\word{w}_{j + 1}) - RP(\word{w}_{j + 1})) + \\
    &\sum\limits_{i \in [d - 1]} \sum\limits_{h_i | f_i} RA(\word{w}_{j + 1}, \vectorise{h[i]}, \vectorise{q}) \cdot (h_i \cdot f_{i - 1} \cdot f_{i - 2} \cdot \hdots \cdot f_1)
\end{split}\right)
\]

\noindent
Therefore the number of possible values of $\word{x}_{j + 1}$ of size $\vectorise{f}$ is given by $$\sum\limits_{\vectorise{q}' \in V(\vectorise{q}, \vectorise{f})} X(\word{w}_{j + 1}, \vectorise{f}, \vectorise{q}').$$
Similarly the number of possible values of $\word{y}$ is the number of words either belonging to a necklace class smaller than $\Angle{\word{w}_{j + 1}}$, or belonging to the same necklace class as $\Angle{\word{w}_{j + 1}}$, while having a smaller translation.
Note that the number of such words for a given Parikh vector $\vectorise{q}$ is given by $\genfrac(){0pt}{2}{F}{\vectorise{q}} - X(\word{w}_{j + 1},\vectorise{f}, \vectorise{q}) - 1$.
Finally, the number of possible words of size $(f_1, f_2, \hdots, f_{d - 1}, l)$ with the Parikh vector $\vectorise{q}$ is given by $ \genfrac(){0pt}{2}{F}{\vectorise{q}}$.
Using these observations $\gamma(\word{w},i,j,\vectorise{f},\vectorise{q},t,l)$ can be computed as:

\[
\gamma(\word{w},i,j,\vectorise{f},\vectorise{q},t,l) = \begin{cases}
\left(
\begin{split}
&\gamma(\word{w},i,j + 1,\vectorise{f},\vectorise{q} - \Parikh(\word{w}_{j + 1}),t,l) +\\ 
&\sum\limits_{\vectorise{q}' \in V(\vectorise{q}, \vectorise{f})} X(\word{w}_{j + 1}, \vectorise{f}, \vectorise{q}') \cdot \gamma(\word{w},i - j - 1,0,\vectorise{f},\vectorise{q} - \vectorise{q}',t,l) \end{split} \right)
& i > j\\
\sum\limits_{\vectorise{q}' \in V(\vectorise{q}, \vectorise{f})} \left(\genfrac(){0pt}{2}{F}{\vectorise{q}'} -  X(\word{w}_{t + 1}, \vectorise{f}, \vectorise{q}')\right) \cdot \genfrac(){0pt}{2}{F \cdot l}{\vectorise{q} - \vectorise{q}'} & i = j = 0\\
0 & i = j, i > 0
\end{cases}
\]

\noindent
Using $\gamma(\word{w},i,j,\vectorise{f},\vectorise{q},l)$, the size of $\mathbf{B}(\word{w},g_d,j, \vectorise{f},\vectorise{q})$ can be computed in the same manner as the size of $\mathbf{B}(\word{w}, g_d,j, \vectorise{f})$.
More precisely, two cases are considered based on the value of $g_d$ and $j$.

\noindent
\paragraph*{Case 1: $g_d + j \leq f_d$.}
In this case every word $\word{v} \in\mathbf{B}(\word{w},g_d,j, \vectorise{f},\vectorise{q})$ can be written as $\word{a} : \Angle{\word{w}_{[1,j]}:\word{b}}_{\theta} : \word{c}$ where:
\begin{itemize}
    \item $\word{a}$ is a $(f_1,f_2,\hdots,f_{d - 1}, g_d)$ dimensional word for which there exists no translation $r \in Z_{(f_1,f_2,\hdots,g_d)}$ such that $(\Angle{\word{a}}_r)_{[1,g_d - r_d]} < \word{w}_{[1,g_d - r_d]}$.
    \item $\word{b}$ is some word of size $(f_1,f_2,\hdots,f_{d - 1})$ that is smaller than $\word{w}_{j + 1}$.
    \item $\theta$ is some translation in $Z_{(f_1, f_2, \hdots, f_{d - 1})}$.
    \item $\word{c}$ is an unrestricted word of size $(f_1,f_2,\hdots,f_{d - 1},f_d - (g_d + j + 1))$.
\end{itemize}

\noindent
Note that $\gamma(\word{w},g_d,0,(f_1, f_2, \hdots, f_{d - 1}),\vectorise{q},j,n - g_d - j - 1)$ counts the number of possible values of $\word{a},\word{b}$ and $\word{c}$.
The number of possible values of $\theta$ is equal to the size of the set $\mathbf{\Theta} = \{r \in Z_{\vectorise{f}} : \nexists s \in Z_{\vectorise{f}}$ where $s < r$ and $\Angle{\word{w}}_r = \Angle{\word{w}}_s\}$.
Therefore the size of $\mathbf{B}(\word{w},g_d,j, \vectorise{f},\vectorise{q})$ when $g_d + j < f_d$ is given by:

\[
\gamma(\word{w},g_d,0,(f_1, f_2, \hdots, f_{d - 1}),\vectorise{q},j,n - g_d - j - 1) \cdot |\mathbf{\Theta}|
\]

% This gives the total number of words of the form $\word{a} : \Angle{(\word{w}_{[1,j]}:\word{b})}_{\theta} : \word{c}$ as $|\beta(\word{w},g_d,0)| \cdot (q^{n_1\cdot n_2\cdot \hdots \cdot n_{d - 1}} - |\beta(\word{w}_{j + 1},1,0)| - 1) \cdot q^{n_1\cdot n_2\cdot \hdots\cdot n_{d - 1} \cdot (n_d - (g_d + j + 1))}$.

\noindent
\paragraph*{Case 2: $g_d + j > f_d$}.
In this case every word $\word{v} \in \mathbf{B}(\word{w},g_d, j,\vectorise{f}, \vectorise{q})$ can be written as $\Angle{\word{w}_{[j + g_d - f_d,j]} : \word{b}}_{\theta} : \word{a} : \Angle{\word{w}_{[1,j + g_d - f_d]}}_\theta$ where:
\begin{itemize}
    \item $\word{a}$ is a $d$-dimensional word of size $(f_1, f_2, \hdots, f_{d - 1}, f_d - (j + 1))$ for which there exists no translation $r \in Z_{(f_1,f_2,\hdots,f_{d - 1}, g_d)}$ such that $\Angle{\word{a}}_r < \word{w}_{[1,g_d]}$.
    \item $\word{b}$ is some word of size $(n_1,n_2,\hdots,n_{d - 1})$ that is smaller than $\word{w}_{j + 1}$.
    \item $\theta$ is a translation in the set $\mathbf{\Theta} = \{r \in Z_{(f_1,f_2, \hdots f_{d - 1})} : \nexists s \in Z_{(f_1,f_2, \hdots f_{d - 1})}$ where $s < r$ and $\Angle{\word{w}_{[1,j]}}_r = \Angle{\word{w}_{[1,j]}}_s\}$.
\end{itemize}
The number of possible values of $\theta$ is equal to the size of the set $\mathbf{\Theta}$ as in Case 1.
The number of possible values of $\word{b}$ in this case is somewhat more complicated than in Case 1.
Let $t$ be the length of the longest suffix of $\word{w}_{[j + g_d - n_d,j]}$ such that $\word{w}_{[j - t,j]} = \word{w}_{[1,t]}$. %, for some translation $\psi \in Z_{\{f_1,f_2,\hdots,f_{d-1},n_d-g_d\}}$.
To avoid $\Angle{\word{v}}_{\psi}$, for some $\psi \in Z_{(f_1,f_2,\hdots,f_{d-1},n_d - g_d)}$, being smaller than $\word{w}$, $\word{b}$ must be greater than or equal to $\word{w}_{t + 1}$.
Let $\gamma'(\word{w}, i,j,\vectorise{f},\vectorise{q})$ return only the number of words with Parikh vector $\vectorise{q}$ that are greater than $\word{w}$ for any translation of the suffix, defined as:

\[
\gamma'(\word{w}, i,j,\vectorise{f},\vectorise{q}) = 
\begin{cases}
\left(
\begin{split}
&\gamma'(\word{w},i,j + 1,\vectorise{f},\vectorise{q} - \Parikh(\word{w}_{j + 1})) +\\ 
&\sum\limits_{\vectorise{q}' \in V(\vectorise{q}, \vectorise{h[i]})} X(\word{w}_{j + 1}, \vectorise{f}, \vectorise{q}') \cdot \gamma'(\word{w},i - j - 1,0,\vectorise{f},\vectorise{q} - \vectorise{q}')
\end{split}
\right) & i > j\\
1 & i = j = 0\\
0 & otherwise.
\end{cases}
\]

\noindent
Using $\gamma'(\word{w}, i,j,\vectorise{f},\vectorise{q})$, the number of words greater than $\word{w}_{t + 1}$ is given by $$\sum\limits_{\vectorise{q}' \in V(\vectorise{q},\vectorise{f})}\gamma'(\word{w}_{t + 1},1,0,\vectorise{f},\vectorise{q}').$$
This gives the number of possible values of $\word{a}$ and $\word{b}$ where $\word{a} > \word{w}_{t + 1}$ as $$\sum\limits_{\vectorise{q}' \in V(\vectorise{q},\vectorise{f})}\left(\gamma'(\word{w}_{t + 1},1,0,\vectorise{f},\vectorise{q}') - \gamma'(\word{w}_{j + 1},1,0,\vectorise{f},\vectorise{q}') \right)\cdot \gamma'(\word{w},i,f_d - (j + 1), \vectorise{f}, \vectorise{q} - \vectorise{q}').$$
Accounting for the case where $\word{a} = \word{w}_{t + 1}$, the size of $\mathbf{B}(\word{w},g_d,j,\vectorise{f}, \vectorise{q})$ when $g_d + j > f_d$ is given by:

\[
|\mathbf{\Theta}| \cdot \left(
\begin{split}
    &\gamma'(\word{w},f_d - j,t + 1,\vectorise{f},\vectorise{q} - \Parikh(\word{w}_{t + 1})) + \\ &\sum\limits_{\vectorise{q}' \in V(\vectorise{q},\vectorise{f})}\left(\gamma'(\word{w}_{t + 1},1,0,\vectorise{f},\vectorise{q}') - \gamma'(\word{w}_{j + 1},1,0,\vectorise{f},\vectorise{q}')\right)\cdot \gamma'(\word{w},i,f_d - (j + 1), \vectorise{f}, \vectorise{q} - \vectorise{q}')
\end{split}
\right)
\]

% Note that $\gamma(\word{w},f_d - (j + 1) ,j,\vectorise{f},\vectorise{q},0)$ returns the product of the possible values of

% Note that the number of possible values of $\word{b}$ that are smaller than $\word{w}_j$ is given by $\gamma(\word{w},1,0,\vectorise{f},\vectorise{q},0)$

% Therefore, the number of possible values of $\word{a}$ and $\word{b}$

% Note that the number of words greater than $\word{w}_{t + 1}$ is given by $\beta(\word{w}_{t + 1},1,0,\vectorise{f})$.
% Therefore the number of possible values of $\word{b}$ as $(q^{n_1\cdot n_2\cdot \hdots\cdot n_{d - 1} \cdot (n_d - (g_d + j + 1))} - \beta(\word{w}_{j + 1},1,0) - 1) - (q^{n_1\cdot n_2\cdot \hdots\cdot n_{d - 1} \cdot (n_d - (g_d + j + 1))} - \beta(\word{w}_{t + 1},1,0,\vectorise{f})) = \beta(\word{w}_{t + 1},1,0,\vectorise{f}) - \beta(\word{w}_{j + 1},1,0,\vectorise{f}) + 1$.
% If $\word{b} = \word{w}_{t + 1}$, the number of possible values of $\word{a}$ is given by $|\beta(\word{w},n_d + t - j,t + 1,\vectorise{f})|$.
% Otherwise the number of possible values of $\word{a}$ is given by $|\beta(\word{w},n_d- j - 1,0,\vectorise{f})|$.
% Therefore the total number of words of the form $\Angle{\word{w}_{[j + g_d - n_d,j]} : \word{b}}_{\theta} : \word{a} : \Angle{\word{w}_{[1,j + g_d - n_d]}}_{\theta}$ is:

% $$|\beta(\word{w},n_d + t - j,t + 1, \vectorise{f})| + (\beta(\word{w}_{t + 1},1,0,\vectorise{f}) - \beta(\word{w}_{j + 1},1,0,\vectorise{f}))\cdot |\beta(\word{w},n_d- j - 1,0,\vectorise{f})|\cdot|\mathbf{\Theta}|$$

% \newtheorem*{thm_fixed_content_ranking}{Theorem \ref{thm:fixed_content_ranking}}

\begin{theorem}
\label{thm:fixed_content_ranking}
The rank of a $d$-dimensional necklace in the set $\mathcal{N}_{\vectorise{p}}^{\vectorise{n}}$ can be computed in $O(N^{6 + q})$ time, where $N = \prod_{i = 1}^d n_i$ and $\vectorise{p}$ is some given Parikh vector of length $q$.
\end{theorem}

% \begin{thm_fixed_content_ranking}
% The rank of a $d$-dimensional necklace in the set $\mathcal{N}_{\vectorise{p}}^{\vectorise{n}}$ can be computed in $O(N^{6 + q})$ time, where $N = \prod_{i = 1}^d n_i$ and $\vectorise{p}$ is some given Parikh vector of length $q$.
% \end{thm_fixed_content_ranking}

\begin{proof}
Following the same arguments from Theorem \ref{thm:ranking_complexity}, the complexity cost of this problem comes from computing $\gamma(\word{w},i,j,\vectorise{f},\vectorise{q},t,l)$.
In order to compute $\gamma(\word{w},i,j,\vectorise{f},\vectorise{q},t,l)$, a dynamic programming approach is used.
Observe that $\gamma(\word{w},i,j,\vectorise{f},\vectorise{q},t,l)$ can be computed in $|V(\vectorise{q})|$ steps if $X(\word{w}_{j + 1}, \vectorise{f}, \vectorise{q}')$ and $\gamma(\word{w},i - j - 1,0,\vectorise{f},\vectorise{q} - \vectorise{q}',t,l)$ have been computed for every $\vectorise{q}' \in V(\vectorise{q})$.
Further, $\gamma(\word{w},i,j,\vectorise{f},\vectorise{q},t,l)$ can be computed in $O(1)$ time when $i = j$ if $X(\word{w}_j, \vectorise{f}, \vectorise{q}')$ has been precomputed for every value of $\word{w}_j, \vectorise{f}$ and $\vectorise{q}'$.

In order to compute $X(\word{w}_j, \vectorise{f}, \vectorise{q}')$, it is necessary to compute the rank of $\word{w}_{j}$ among the set of $d-1$ atranslational necklaces, in turn requiring $\gamma(\word{w},i,j,\vectorise{f}',\vectorise{q},t,l)$ to be  computed for every $\vectorise{f}' \in \{(m_1,m_2,\hdots,m_{d - 2}) : n_i \bmod m_i \equiv 0\}$.
By repeating the same arguments from Theorem \ref{thm:ranking_complexity}, the problem of ranking fixed content necklaces can be done in an additional factor of $O(N^{q + 1})$, accounting for the number of possible Parikh vectors $\vectorise{q}$, and possible values of $l$.
Therefore, the total complexity is $O(N^{6 + q})$.
\end{proof}

\section{Unranking Necklaces}
\label{chap8:subsec:unranking}

This section covers our technique for unranking necklaces.
The key idea behind this technique is to build the canonical representation of the $i^{th}$ necklace, $\word{u}$, by iteratively determining each prefix of $\word{u}$ in increasing length.
The prefix of length $j + 1$ is determined from the prefix of length $j$ through a binary search of the space of necklaces of size $(n_1,n_2, \hdots, n_{d - 1})$.
The binary search process is done using the ranking algorithm as a subroutine.
When evaluating the necklace $\necklace{v} \in \mathcal{N}_q^{n_1,n_2,\hdots,n_{d - 1}}$, the rank of the smallest word with the prefix $\word{u}_{[1,i]} : \Angle{\necklace{v}}$, and the largest word with the prefix $\word{u}_{[1,i]} : \Angle{\necklace{v}}_{TP(\necklace{v})}$ are compared.
The binary search proceeds by comparing the ranks of these words with $i$, until some $d - 1$ dimensional necklace $\necklace{v} \in \mathcal{N}_q^{n_1,n_2,\hdots,n_{d - 1}}$ is found such that $i$ is between the rank of the smallest and largest $d$-dimensional necklaces with $\word{u}_{[1,i]} : \Angle{\necklace{v}}$ and $\word{u}_{[1,i]} : \Angle{\necklace{v}}_{TP(\necklace{v})}$ as a prefix respectively.
Once such a necklace $\necklace{v}$ is found, the same process is repeated on the set of possible translations of $\necklace{v}$ to find the prefix of $\word{u}$ of length $j + 1$.
This process is repeated until the prefix of length $n_d$ is found, corresponding directly to $\word{u}$.

The remainder of this section is organised as follows.
Lemma \ref{chp8:lem:prefix_count} provides the key tool for determining the number of necklaces sharing a given prefix alongside the primary technical arguments for the unranking process.
Using Lemma \ref{chp8:lem:prefix_count} as a basis, Theorem \ref{chp8:thm:unranking} is restated and formally proven.
Finally, Lemma \ref{chp8:lem:fixed_content_prefix_counting} and Corollary \ref{chp8:col:fixed_content_unranking} are used to extend the Lemma \ref{chp8:lem:prefix_count} and Theorem \ref{chp8:thm:unranking} respectively to the fixed content setting.

\begin{lemma}
\label{chp8:lem:prefix_count}
The number of necklaces in $|\mathcal{N}_q^{\vectorise{n}}|$ with a given prefix $\word{w}$ can be determined in $O(N^5)$ time.
\end{lemma}

\begin{proof}
Let $\word{w}$ be a word of size $(n_1,n_2,\hdots,n_{d - 1}, a)$, where $a \leq n_d$.
To determine the number of necklaces with a prefix $\word{w}$, two new words $\word{u}$ and $\word{v}$ are defined such that $\word{u}$ is the smallest necklace reference with the prefix $\word{w}$, and $\word{v}$ the greatest.
The value of $\word{u}$ is determined by first constructing the word $\word{u}'$ where $\word{u}'_i = \word{w}_{i \bmod a}$.
If $\word{u}'$ is the canonical representation of the necklace $\Angle{\necklace{u}'}$, then $\word{u} = \word{u}'$.
Otherwise using Theorem \ref{chp8:thm:generation_complexity}, the value of $\word{u}$ is computed from $\word{u}'$ in at most $O(N)$ operations.
Let $\word{Q} = q^{(n_1,n_2,\hdots,n_{d - 1})}$.
The word $\word{v}$ is defined as being equal to $\word{w}: \word{Q}^{n_d - a}$.
If $\word{v}$ is not the canonical representation of $\necklace{v}$ then there exists no necklace with $\word{w}$ as a prefix.
Otherwise, the number of necklaces with $\word{w}$ as a prefix equals $RN(\word{v}) - RN(\word{u}) + 1$.
\end{proof}

\noindent
Using Lemma \ref{chp8:lem:prefix_count}, a recursive unranking algorithm can be built by iteratively building the prefix of the $i^{th}$ necklace in $|\mathcal{N}_q^{\vectorise{n}}|$.

% \newtheorem*{theorem3}{Theorem \ref{chp8:thm:unranking}}

% \begin{theorem3}
% The $i^{th}$ necklace in $\mathcal{N}_q^{\vectorise{n}}$ can generated (unranked) in $O\left(N^{6(d + 1)} \cdot \log^d(q)\right)$ time.
% \end{theorem3}

\begin{theorem}
\label{chp8:thm:unranking}
The $i^{th}$ necklace in $\mathcal{N}_q^{\vectorise{n}}$ can be generated (unranked) in $O\left(N^{6(d + 1)} \cdot \log^d(q)\right)$ time.
\end{theorem}

\begin{proof}
The unranking procedure is done in a similar manner to the 1D case as presented by Sawada and Williams \cite{Sawada2017}.
At a high level, the idea is to iteratively generate the necklace by generating prefixes of increasing length.
Let $\word{w}$ be the canonical representation of the $i^{th}$ necklace.
Further let $\word{Q} = q^{(n_1,n_2,\hdots,n_{d - 1})}$, the word of size $(n_1,n_2,\hdots,n_{d - 1})$ where every position is occupied by the symbol $k$.
The first slice of $\word{w}$ is determined through a binary search.
Let $\word{u}$ be the canonical representation of $j^{th}$ necklace of size $(n_1,n_2, \hdots, n_{d - 1})$.
Note that if $\word{u}$ is the first slice of $\word{w}$, then the rank of $\word{w}$ must be between the rank of the smallest necklace starting with $\word{u}$ and the greatest.
These necklaces are determined using the same process as laid out in Lemma \ref{chp8:lem:prefix_count}.
Let $\word{a}$ be the smallest such word and $\word{b}$ the greatest.
Therefore $\word{u}$ is the fist slice of $\word{w}$ if and only if $RN(\word{a}) \leq i \leq RN(\word{b}$.
Otherwise, depending on the value of $i$ relative to $RN(\word{a})$ and $RN(\word{b})$ the next value of $\word{u}$ is checked, with $\word{u}$ determined by a binary search.
Note that there are at most $q^{N/n_d}$ necklaces of size $(n_1,n_2,\hdots,n_{d - 1})$, the binary search requires at most $\log(q^{N/n_d}) = \frac{N}{n_d} \log k$ necklaces to be checked.

For the $t^{th}$ slice, where $t \geq 2$, the process is slightly more complicated.
As in the first case, to determine if the $\Angle{\word{w}_t} = \word{u}$, the smallest and largest such words are determined and ranked.
To that end, let $\word{a}$ be the smallest possible word that is the canonical representation of a necklace and has the prefix $\word{w}_{[1,t - 1]}: \Angle{\word{u}}_g$, and let $\word{b}$ be the greatest.
The value of $\word{a}$ is computed in $O(N)$ time following the techniques outlined in Theorem \ref{chp8:thm:generation_complexity}.
The word $\word{b} = \word{w}_{[1,t - 1]}:\Angle{\word{u}}_g : \word{Q}^{n_d - t}$ where $g$ is the largest translation such that $\word{u} \neq \Angle{\word{u}}_g$.
Using these words, $\Angle{\word{w}_{t}} = \word{u}$ if and only if $RN(\word{a}) \leq i \leq RN(\word{b})$.

The complexity of this process comes from the recursive nature of algorithm.
In dimension $d$, $n_d$ slices need to be computed, each requiring at most $\frac{N}{n_d} \cdot \log(q)$
necklaces to be ranked, the ranking having a complexity of $N^5$.
Note that while determining the necklace that needs to be ranked has a complexity of $N^2$, this is not multiplicative with the complexity of ranking as each step is done independently.
To determine each of these necklaces, a necklace of size $(n_1,n_2,\hdots,n_{d - 1})$ must be unranked, adding an additional complexity of $n_{d - 1} \cdot \frac{N}{n_d \cdot n_{d - 1}} \cdot \frac{N^5}{n_d^5} \cdot \log(q)$.
As each dimension requires necklaces of the dimension one lower to be computed, the total complexity is $O\left(\prod\limits_{i = 0}^d \frac{N^6 \cdot \log(q)}{\prod\limits_{j \in [1,i]} n^6_{d - j}}\right)$.
In the worst case, where $n_1 = N$ and $n_i = 1$ for $i \in [2,d]$, this is simplified to $O\left(N^{6(d + 1)} \cdot \log^d(q)\right)$.
\end{proof}

\begin{lemma}
\label{chp8:lem:fixed_content_prefix_counting}
The number of necklaces in the set $\mathcal{N}_{\vectorise{p}}^{\vectorise{n}}$ sharing a given prefix $\word{a}$ can be computed in $O(n^{6 + q})$ time.
\end{lemma}

\begin{proof}
Note that the ranking process outline in Theorem \ref{thm:fixed_content_ranking} allows the rank of the canonical representation of any necklace to be computed within the set $\mathcal{N}_{\vectorise{p}}^{\vectorise{n}}$ in $O(n^{6 + q})$ time.
Therefore by comparing the ranks of the smallest and largest necklaces sharing $\word{a}$ as a prefix, the number of necklaces in $\mathcal{N}_{\vectorise{p}}^{\vectorise{n}}$ sharing the prefix can be computed.
Following Theorem \ref{chp8:thm:generation_complexity}, the smallest and largest necklaces can be found in $O(N)$ time.
As the ranking process requires at most $O(n^{6 + q})$ time, the total complexity of determining the number of necklaces sharing a given prefix is $O(n^{6 + q})$.
\end{proof}

% \newtheorem*{col_fixed_content_unranking}{Corollary \ref{chp8:col:fixed_content_unranking}}

% \begin{col_fixed_content_unranking}
% The $i^{th}$ necklace in $\mathcal{N}_{\vectorise{p}}^\vectorise{n}$ can be unranked in $O(N^{(q + 7)(d + 1)}) \log^d(q)$ time. 
% \end{col_fixed_content_unranking}

\begin{corollary}
\label{chp8:col:fixed_content_unranking}
The $i^{th}$ necklace in $\mathcal{N}_{\vectorise{p}}^\vectorise{n}$ can be generated (unranked) in $O(N^{(q + 7)(d + 1)}) \log^d(q)$ time.
\end{corollary}

\begin{proof}
Fixed content multidimensional necklaces can be unranked in the same manner as unconstrained necklaces, presented in Theorem \ref{chp8:thm:unranking}.
As in that theorem, a binary search is used over the alphabet $\Sigma$ to determine the $i^{th}$ necklace iteratively.
Following Lemma \ref{chp8:lem:fixed_content_prefix_counting}, the number of necklaces sharing a given prefix can be computed in $O(n^{6 + q})$ time.
% Note that the ranking algorithm provided in Theorem \ref{thm:fixed_content_ranking} does not require a fix content necklace to rank, but instead requires only the canonical representation of a necklace.
The complexity of this process is given by the same arguments as in Theorem \ref{chp8:thm:unranking}, with the additional cost due to the added complexity of ranking fixed content necklaces compared to unconstrained necklaces, being $O(N^{6 + q})$ and $O(N^5)$ respectively.
\end{proof}

% \section{Conclusion}

% This work has introduced the $k$-centre problem for multidimensional necklaces.
% We have shown the problem to be at least NP-hard in any number of size.
% Further, we provide foundational results for multidimensional necklaces, including methods to count, rank, generate and unrank this structure.

% There are several open problems left from this work.
% The most obvious is the question of whether a de Bruijn hypertorus can be generated.
% Our work shows that it possible to generate 

\section{The \texorpdfstring{$k$}{k}-Centre Problem on Necklaces}
\label{sec:k-centre}

The final set of problems this paper considers is that of choosing a representative sample from some set of necklaces, both in the 1D and multidimensional cases.
Here we focus on the \emph{local structures}, representing the interactions between ions that are as close as possible.
The motivation for this approach comes from the energy functions which we look at which have a rapid decrease in energy as distance increases.
For example, the Coulomb potential defined as $\frac{q_i \cdot q_j}{r_{ij}}$ tends rapidly towards 0.
As such, finding local structures provides a strong basis for exploring the space of possible solutions.

We use the $k$-centre problem as a basis to formalise these notions as a computer science problem.
% \duncan{In red, what can be used in the introduction.}
% {\color{red}
% In addition to the generalisation of existing results on the 1D case, we also present the $k$-centre problem for necklaces.
The $k$-centre problem is a classical graph problem.
The $k$-centre problem takes as input a weighted graph $G = (V,E)$ and integer $k$, with the goal of finding a set $\mathbf{S}$ of $k$ vertices from $V$ minimising $\max_{v \in V} \min_{u \in \mathbf{S}} D(v,u)$ where $D(v,u)$ returns the distance between vertices $v$ and $u$.
To use the $k$-centre problem as a basis for this setting it is necessary to define a distance between words emphasising local differences.
% The objective in $k$-center problem is to find a set of $k$ vertices for which the largest distance of any vertex of the graph and its closest vertex in this $k$-set is minimised.
The numerous applications of the problem in various areas of computer science have lead to different definitions of connectivity and distance between the vertices depending on the setting at hand.

The $k$-center problem is a classical NP-hard problem, as such a great deal of research has been direct to trying to solve it.
In the general case the problem is known to not be in APX \cite{hochbaum1997various}.
When the distance satisfies the triangle inequality the problem becomes significantly easier, admitting a polynomial time (relative to the size of the graph) approximation algorithm with a factor of $2$ \cite{gonzalez1985clustering, hochbaum1986unified}.
Further, it is known no polynomial time approximation algorithm can achieve a factor better than $2$ unless $P = NP$ \cite{hsu1979easy, plesnik1980computational}.
Additionally the $k$-centre problem is unlikely to be fixed-parameter tractable (FPT) in a context of the most natural parameter $k$ \cite{Algorithmica2020}.

A different form of the $k$-center problem appears in stringology and it was linked with important applications in computational biology; for example to find the approximate gene clusters for a set of words over the DNA alphabet \cite{JACM2002}.
This problem is also NP-hard \cite{frances1997covering,lanctot2003distinguishing}.
Despite the hardness of the problem, there are fixed-parameter algorithms \cite{gramm2003fixed, ma2008more} allowing some guarantee of optimality for solving the problem.
The Closest String problem aims to find a new string within a distance $d$ to each input of $n$ strings and such that $d$ is minimised.
The natural generalisation of $k$-Closest String problem is of finding $k$-center strings of a given length minimising the distance from every string to closest center \cite{SODA99,CPM2004}. 
This problem has been mainly studied for the popular Hamming distance.
The major application of this distance is in the coding theory, but it also has been intensively used in  biological applications aiming to discover a region of similarity or to design both probes and primers \cite{IC2003}. 
% }

The k-center problem can be defined over various distance functions. In this paper  
we study it in respect to the overlap distance
 function which can representing the closeness in relation to the number of common subwords
 and in its turn the closeness of a potential energy in crystals. However, it is not critical for our algorithmic results; all results could be reformulated using other functions by giving of course slightly different approximation bounds. Also, the application of the overlap coefficient, the inverse of which is used as our distance function, is not new and has been successfully used to describe local similarities for “bag-of-words” machine learning techniques, see \cite{gartner2003survey}.

The remainder of this section is organised as follows.
Section \ref{ch:sampling:sec:prelinaries} provides the key definitions for this chapter, including the distance used and some fundamental results for the $k$-centre problem in this setting.
Section \ref{ch5:sec:dbt_apx} provides the first approximation algorithm for solving this problem in the 1D case for unconstrained necklaces, using de Bruijn sequences as a basis.
% Section \ref{subsec:prefix_alg} provides a less precise but more general approximation algorithm, covering the other settings discussed in this thesis, including bracelets, necklaces with constrained Parikh vectors, necklaces with forbidden subwords and multidimensional necklaces.
% Finally Section \ref{sec:k_centre_usage} provides further discussion on the usage of this problem in the context of \textsc{csp}.

% \textbf{TODO}

% \begin{itemize}
%     \item Define a couple of different distances, show how the algs. preform on them.
% \end{itemize}

% % \section{The \texorpdfstring{$k$}{k}-Centre problem for necklaces}

\subsection{The Overlap Distance and the k-Centre Problem}
\label{ch:sampling:sec:prelinaries}

In this section we formally define the $k$-centre problem for necklaces. 
At a high level, the input to our problem is an alphabet of size $q$, a vector of size  $\vectorise{n} = (n_1, n_2, \hdots, n_d)$ that defines the size of the $d$-dimensional words, and a positive integer $k$.
Note that in the 1D case $\vectorise{n}$ may be given as a single scalar value, $n$.
The goal is to choose  a set $\mathbf{S}$ of $k$ necklaces from the set $\mathcal{N}_q^{\vectorise{n}}$ such that the maximum distance between any necklace $\necklace{w}\in \mathcal{N}_q^{\vectorise{n}}$ and the set $\mathbf{S}$ is minimised.
Since there is no standard notion of distance between necklaces, our first task is to define one.
To this end, we introduce the {\em overlap distance}, which aims to capture similarity between crystalline materials emphasising local differences.
At a high level, the overlap distance between two necklaces is the inverse of the \emph{overlap coefficient} between them, in this case $1$ minus the overlap coefficient.
This can be seen as a natural distance based ``bag-of-words'' techniques used in machine learning \cite{gartner2003survey}.

\noindent
\paragraph*{Overlap Distance for Necklaces.}
% To define the overlap distance between necklaces we extend the notion of the overlap coefficient defined for sets. The overlap coefficient of the sets $A$ and $B$, denoted $\mathfrak{C}(A,B)$, is defined as the size of the intersection of the two sets, normalised by the size of the smaller set, i.e. $\mathfrak{C}(A,B) = \frac{|A \cap B|}{\min(|A |,|B|)}$. This is a measure of {\em closeness} between sets $A$ and $B$: it is equal to $1$ if $A=B$ and $0$ if the sets have no elements in common.
% In the context of necklaces the overlap coefficient $\mathfrak{C}(\necklace{w},\necklace{v})$ is defined as the overlap coefficient between the multisets of all subwords of $\necklace{w}$ and $\necklace{v}$.
% For some necklace $\necklace{w}$ of size $\vectorise{n}$, the multiset of subwords of size $\vectorise{l}$ contains
% all $\word{u} \sqsubseteq_{\vectorise{l}} \word{w}$.
% For each subword $\word{u}$ appearing $m$ times in $\necklace{w}$, $m$ copies of $\word{u}$ are added to the multiset labelled by $\word{u}_1, \word{u}_2, \hdots, \word{u}_m$.
% This gives a total of $N$ subwords of size $\vectorise{l}$ for any $\vectorise{l}$, where $N = n_1 \cdot n_2 \cdot \hdots \cdot n_d$.
% For example, given the necklace represented by $aaab$, the multiset of subwords of length 2 are $(aa_1,aa_2,ab_1,ba_1)$.
% The multiset of all subwords is the union of the multisets of the subwords for every set of size, having a total size of $N^2$; see Figure \ref{ch5:fig:overlap_example}.

Our definition of the overlap distance depends of the well studied \emph{overlap coefficient}, defined for a pair of set $A$ and $B$ as $\frac{|A \cap B|}{\min(|A|, |B|)}$.
For notation let $\mathfrak{C}(A,B)$ return the overlap coefficient between two sets $A$ and $B$.
Observe that $\mathfrak{C}(A,B)$ returns a rational value between $0$ and $1$, with $0$ indicating no common elements and $1$ indicating that either $A \subseteq B$ or $B \subseteq A$.
In the context of necklaces the overlap coefficient $\mathfrak{C}(\necklace{w},\necklace{v})$ is defined as the overlap coefficient between the multisets of all subwords of $\necklace{w}$ and $\necklace{v}$.
For some necklace $\necklace{w}$ of size $\vectorise{n}$, the multiset of subwords of size $\vectorise{\ell}$ contains
all $\word{u} \sqsubseteq_{\vectorise{\ell}} \word{w}$.
For each subword $\word{u}$ appearing $m$ times in $\necklace{w}$, $m$ copies of $\word{u}$ are added to the multiset.
This gives a total of $N$ subwords of size $\vectorise{\ell}$ for any $\vectorise{\ell}$, where $N = n_1 \cdot n_2 \cdot \hdots \cdot n_d$.
For example, given the necklace represented by $aaab$, the multiset of subwords of length 2 are $\{aa,aa,ab,ba\} = \{aa \times 2,ab,ba\}$.
The multiset of all subwords is the union of the multisets of the subwords for every vector of size, having a total size of $N^2$; see Figure \ref{fig:overlap_example}.

\begin{figure}[ht]
    \centering
 %   \begin{tabular}{l l | l l}
 %       word & same length representative 
 %& word & same length representative\\
 %       $ab$ & $ababab$ & $abb$ & 
 %$abbabb$\\
 %   \end{tabular}
    \scriptsize{
    \begin{tabular}{c|l|l|r}
        %Length 
        & word $ababab$ 
        %, the \emph{same length representative}  $ababab$  
        & word $abbabb$ & Intersection\\ 
        %, the \emph{same length representative} $abbabb$\\
        \hline
        1 & $a \times 3, b \times 3$ & $a \times 2, b \times 4$ & 5\\
        2 & $ab \times 3, ba \times 3$ & $ab \times 2, bb \times 2, ba \times 2$ & 4\\
        3 & $aba \times 3, bab \times 3$ & $abb \times 2, bba \times 2, bab \times 2$ & 2\\
        4 & $abab \times 3, baba \times 3$ & $abba \times 2, bbab \times 2, babb \times 2$ & 0\\
        5 & $ababa \times 3, babab \times 3$ & $abbab \times 2, bbabb \times 2, babba \times 2$ & 0\\
        6 & $ababab \times 3, bababa \times 3$ & $abbabb \times 2, bbabba \times 2, babbab \times 2$ & 0\\
        Total & & & 11
    \end{tabular}
    }
    \caption{Example of the overlap coefficient calculation for a pair of words $ababab$ and $abbabb$. 
    There are $11$ common subwords out of the total number of $36$ subwords of length from $1$ till $6$, so $\mathfrak{C}(ababab,abbabb)= \frac{11}{36}$
    and $\mathfrak{O}(ababab,abbabb) =  \frac{25}{36}$.
    }
    \label{fig:overlap_example}
\end{figure}

To use the overlap coefficient as a distance between $\necklace{w}$ and $\necklace{v}$, the overlap coefficient is inverted so that a value of $1$ means $\necklace{w}$ and $\necklace{v}$ share no common subwords while a value of $0$ means $\necklace{w} = \necklace{v}$.
The overlap distance (see example in Figure \ref{fig:overlap_example}) between two necklaces $\necklace{w}$ and $\necklace{v}$ is ${\small \mathfrak{O}(\necklace{w}, \necklace{v}) = 1 - \mathfrak{C}(\necklace{w}, \necklace{v})}$.
Proposition \ref{prop:metric_distance} shows that this distance is a metric distance.
 
 \begin{figure}
    \centering
    \begin{tabular}{l l l l l l}
        A & $aaaa$ & B & $aaab$ & C & $aabb$  \\
        D & $abab$ & E & $abbb$ & F & $bbbb$
    \end{tabular}
 %   Overlap between:
    \begin{tabular}{l| l l l l l l}
         $\necklace{w} \backslash \necklace{v}$  & A & B & C & D & E & F \\
        \hline
        A & 0 & $\frac{10}{16}$ & $\frac{13}{16}$ & $\frac{14}{16}$ & $\frac{15}{16}$ & $1$\\
        B & $\frac{10}{16}$ & 0 & $\frac{9}{16}$ & $\frac{10}{16}$ & $\frac{12}{16}$ & $\frac{15}{16}$\\
        C & $\frac{13}{16}$ & $\frac{9}{16}$ & 0 & $\frac{10}{16}$ & $\frac{8}{16}$ & $\frac{13}{16}$\\
        D & $\frac{14}{16}$ & $\frac{10}{16}$ & $\frac{10}{16}$ & 0 & $\frac{6}{16}$ & $\frac{14}{16}$\\
        E & $\frac{15}{16}$ & $\frac{12}{16}$ & $\frac{8}{16}$ & $\frac{10}{16}$ & 0 & $\frac{10}{16}$\\
        F & $1$ & $\frac{15}{16}$ & $\frac{13}{16}$ & $\frac{14}{16}$ & $\frac{8}{16}$ & 0
    \end{tabular}
    
    \caption{Example of the overlap distance $\mathfrak{D}(\langle\necklace{w}\rangle,\langle\necklace{v}\rangle)$ for all binary necklaces of length 4.}
    \label{fig:overlapDistances}
\end{figure}
 
\begin{proposition}
\label{prop:metric_distance}
The overlap distance for necklaces is a metric distance.
\end{proposition}

\begin{proof}
Let $\necklace{a}, \necklace{b}, \necklace{c}\in \mathcal{N}_q^{\vectorise{n}}$, for some arbitrary vector $\vectorise{n} \in \mathbb{N}^d$ and $q \in \mathbb{N}$.
In order for the overlap distance to satisfy the metric property, $\mathfrak{O}(\necklace{a}, \necklace{b})$ must be less than or equal to $\mathfrak{O}(\necklace{a}, \necklace{c}) + \mathfrak{O}(\necklace{b}, \necklace{c})$.
Rewriting this gives $1 - \mathfrak{C}(\necklace{a}, \necklace{b}) \leq 2 + \mathfrak{C}(\necklace{a}, \necklace{b}) - \mathfrak{C}(\necklace{b}, \necklace{c})$ which can be rewritten in turn as $\mathfrak{C}(\necklace{a}, \necklace{b}) + \mathfrak{C}(\necklace{b}, \necklace{c}) \leq 1 + \mathfrak{C}(\necklace{a}, \necklace{b})$.
Observe that if $\mathfrak{C}(\necklace{a}, \necklace{c}) + \mathfrak{C}(\necklace{b}, \necklace{c}) > 1$ then $\frac{|\necklace{a} \cup \necklace{c}|}{N^2} + \frac{|\necklace{b} \cup \necklace{c}|}{N^2} > 1$, meaning that  $|\necklace{a} \cup \necklace{c}| + |\necklace{b} \cup \necklace{c}| > N^2$.
This implies that $\necklace{a}$ and $\necklace{b}$ share at least $|\necklace{a} \cup \necklace{c}| + |\necklace{b} \cup \necklace{c}| - N^2$ subwords.
Therefore $\mathfrak{C}(\necklace{a}, \necklace{n})$ must be at least $\mathfrak{C}(\necklace{a}, \necklace{n}) + \mathfrak{C}(\necklace{b}, \necklace{c}) - 1$.
Hence $\mathfrak{O}(\necklace{a}, \necklace{b}) \leq \mathfrak{O}(\necklace{a}, \necklace{c}) + \mathfrak{O}(\necklace{b}, \necklace{c})$.
\end{proof}

\noindent
\paragraph*{The \texorpdfstring{$k$}{k}-Centre Problem.}
The goal of the $k$-Centre problem for necklaces is to select a set of $k$ necklaces of size $\vectorise{n}$ over an alphabet of size $q$ that are ``central'' within the set of necklaces $\mathcal{N}_q^{\vectorise{n}}$.
Formally the goal is to choose a set $\mathbf{S}$ of $k$ necklaces such that the maximum distance between any necklace $\necklace{w} \in \mathcal{N}_q^{\vectorise{n}}$ and the nearest member of $\mathbf{S}$ is minimised.
Given a set of necklaces $\mathbf{S} \subset \mathcal{N}_q^{\vectorise{n}}$, we use $\mathfrak{D}(\mathbf{S}, \mathcal{N}_q^{\vectorise{n}})$ to denote the maximum overlap distance between any necklace in $\mathcal{N}_q^{\vectorise{n}}$ and its closest necklace in $\mathbf{S}$.

\begin{problem}
    \label{prob:k_sample}
    $k$-Centre problem for necklaces.
\end{problem}

\noindent
\begin{tabular}{l l}
    \textbf{Input:} & A size vector of $d$-dimensions $\vectorise{n} \in \mathbb{Z}^d$, an alphabet $\Sigma$ of size $q$, and an integer\\
    & $k \in \mathbb{Z}$.\\
    \textbf{Question:} & What is the set $\mathbf{S} \subseteq \mathcal{N}_q^{\vectorise{n}}$ of size $k$ minimising $\mathfrak{D}(\mathbf{S}, \mathcal{N}_q^{\vectorise{n}})$?
\end{tabular}

\noindent
There are two major challenges we have to overcome in order to solve Problem~\ref{prob:k_sample}, the exponential size of $\mathcal{N}_q^{\vectorise{n}}$, and the lack of structural, algorithmic, and combinatorial results for multidimensional necklaces.
We show that the conceptually simpler problem of verifying whether a set of necklaces is a solution for Problem \ref{prob:sampling_descion} is NP-hard for any dimension $d$.

\begin{problem}
\label{prob:sampling_descion}
The $k$-Centre verification problem for necklaces.
Given a set of $k$ necklaces $\mathbf{S} \in \mathcal{N}_q^{\vectorise{n}}$ and a distance $\ell$, does there exist some necklace $\necklace{v} \in \mathcal{N}_q^\vectorise{n}$ such that $\mathfrak{O}(\necklace{s},\necklace{v}) > \ell$ for every $\necklace{s} \in \mathbf{S}$?
\end{problem}

\noindent
\begin{tabular}{l l}
    \textbf{Input:} & A size vector of $d$-dimensions $\vectorise{n} \in \mathbb{Z}^d$, an alphabet $\Sigma$ of size $q$, an integer $k \in \mathbb{Z}$,\\
    & and rational distance $\ell \in \mathbb{Q}$.\\
    \textbf{Question:} & Does there exists a set $\mathbf{S} \subseteq \mathcal{N}_q^{\vectorise{n}}$ of size $k$ such that $\mathfrak{D}(\mathbf{S}, \mathcal{N}_q^{\vectorise{n}}) \leq \ell$?
\end{tabular}

% \newtheorem*{thm_sampling_np}{Theorem \ref{thm:sampling_descion_np}}

% \begin{thm_sampling_np}
% \label{thm:sampling_descion_np}
% Given a set of $k$ necklaces $\mathbf{S} \in \mathcal{N}_q^{\vectorise{n}}$ and a distance $\ell$, it is NP-hard to determine if there exists some necklace $\necklace{v} \in \mathcal{N}_q^\vectorise{n}$ such that $\mathfrak{O}(\necklace{s},\necklace{v}) > \ell$ for every $\necklace{s} \in \mathbf{S}$ for any dimension $d$.
% \end{thm_sampling_np}

\begin{theorem}
\label{thm:sampling_descion_np}
Given a set of $k$ necklaces $\mathbf{S} \in \mathcal{N}_q^{\vectorise{n}}$ and a distance $\ell$, it is NP-hard to determine if there exists some necklace $\necklace{v} \in \mathcal{N}_q^\vectorise{n}$ such that $\mathfrak{O}(\necklace{s},\necklace{v}) > \ell$ for every $\necklace{s} \in \mathbf{S}$ for any dimension $d$.
\end{theorem}

\begin{proof}
This claim is proven via a reduction from the Hamiltonian cycle problem on bipartite graphs to Problem \ref{prob:sampling_descion} in 1D.
Note that if the problem is hard in the 1D case, then it is also hard in any dimension $d \geq 1$ by using the same reduction for necklaces of size $(n_1,1,1,\hdots,1)$.
Let $G = (V,E)$ be a bipartite graph containing an even number $n\geq 6$ of vertices.
The alphabet $\Sigma$ is constructed with size $n$ such that there is a one to one correspondence between each vertex in $V$ and symbol in $\Sigma$.
Using $\Sigma$ a set $\mathbf{S}$ of necklaces is constructed as follows.
For every pair of vertices $u,v \in V$ where $(u,v) \notin E$, the necklace corresponding to the word $(uv)^{n/2}$ is added to the set of centres $\mathbf{S}$.
Further the word $v^n$, for every $v \in V$, is added to the set $\mathbf{S}$.

For the set $\mathbf{S}$, we ask if there exists any necklace in $\mathcal{N}_q^{n}$ that is further than a distance of $1 - \frac{3}{n^2}$.
For the sake of contradiction, assume that there is no Hamiltonian cycle in $G$, and further that there exists a necklace $\necklace{w} \in \mathcal{N}_q^{\vectorise{n}}$ such that the distance between $\necklace{w}$ and every necklace $\necklace{v} \in \mathbf{S}$ is greater than $1 - \frac{3}{n^2}$.
If $\necklace{w}$ shares a subword of length $2$ with any necklace in $\mathbf{S}$ then $\necklace{w}$ would be at a distance of no less than $1 - \frac{3}{n^2}$ from $\mathbf{S}$.
Therefore, as every subword of length $2$ in $\mathbf{S}$ corresponds to a edge that is not a member of $E$, every subword of length 2 in $\necklace{w}$ must correspond to a valid edge.

As $\necklace{w}$ can not correspond to a Hamiltonian cycle, there must be at least one vertex $v$ for which the corresponding symbol appears at least 2 times in $\necklace{w}$.
As $G$ is bipartite, if any cycle represented by $\necklace{w}$ has length greater than $2$, there must exist at least one vertex $u$ such that $(v,u) \notin E$.
Therefore, the necklace $(uv)^{n/2}$ is at a distance of no more than $\frac{n^2}{3}$ from $\necklace{w}$.
Alternatively, if every cycle represented by $\necklace{w}$ has length $2$, there must be some vertex $v$ that is represented at least $3$ times in $\necklace{w}$.
Hence in this case $\necklace{w}$ is at a distance of no more than $1 - \frac{3}{n^2}$ from the word $v^n \in \mathbf{S}$.
Therefore, there exists a necklace at a distance of greater than $1 - \frac{3}{n^2}$ if and only if there exists a Hamiltonian cycle in the graph $G$.
Therefore, it is NP-hard to verify if there exists any necklace at a distance greater than $l$ for some set $\mathbf{S}$.
\end{proof}

\noindent
The combination of this negative result with the exponential size of $\mathcal{N}_q^{\vectorise{n}}$ relative to $\vectorise{n}$ and $q$ makes finding an optimal solution for Problem~\ref{prob:k_sample} exceedingly unlikely.
As such the remainder of our work on the $k$-centre problem for necklaces focuses on approximation algorithms.
Lemma~\ref{lem:min_distance} provides a lower bound on the optimal distance.

\begin{lemma}
\label{lem:min_distance}
Let 
$\mathbf{S} \subseteq  \mathcal{N}_q^{\vectorise{n}}$
be an optimal set of $k$ centres
minimising 
$\mathfrak{D}(\mathbf{S}, \mathcal{N}_q^\vectorise{n})$ 
then 
$\mathfrak{D}(\mathbf{S}, \mathcal{N}_q^\vectorise{n}) \geq 
1- \frac{\log_{q} (k \cdot N)}{N}$.
\end{lemma}

\begin{proof}
We first prove the lemma for the 1D case, then extend the proof to the multidimensional setting.
Recall that the distance between any pair of necklaces $\necklace{u}$ and $\necklace{v}$ is determined by the overlap coefficient and by extension the number of shared subwords between $\necklace{u}$ and $\necklace{v}$.
% As such, by bounding the number of shared subwords between $\necklace{u}$ and $\necklace{v}$ from above, the distance between $\necklace{u}$ and $\necklace{v}$ may be bound from bellow.
Hence the distance between the furthest necklace $\necklace{w} \in \mathcal{N}_q^n$ and the optimal set $\mathbf{S}$ is bound from bellow by determining an upper bound on the number of shared subwords between $\necklace{w}$ and the words in $\mathbf{S}$.
For the remainder of this proof let $\necklace{w}$ to be the necklace furthest from the optimal set $\mathbf{S}$.
% In other words, the distance between $\necklace{w}$ and the optimal set $\mathbf{S}$ is bound from bellow by determining the maximum number of subwords that $\necklace{w}$ shares with $\mathbf{S}$.
Further for the sake of determining an upper bound, the set $\mathbf{S}$ is treated as a single necklace $ \necklace{S}$ of length $n \cdot k$.
This may be thought of as the necklace corresponding to the concatenation of each necklace in $\mathbf{S}$.
Note that the length of $\mathbf{S}$ is $k \cdot n$.
As the distance between $\necklace{w}$ and $\necklace{S}$ is no more than the distance between $\necklace{w}$ and any $\necklace{v} \in \mathbf{S}$, the distance between $\necklace{w}$ and $\necklace{S}$ provides a lower bound on the distance between $\necklace{w}$ and $\mathbf{S}$.

In order to determine the number of subwords shared by $\necklace{w}$ and $\necklace{S}$, consider first the subwords of length $1$.
In order to guarantee that $\necklace{w}$ shares at least one subword of length $1$, $\necklace{S}$ must contain each symbol in $\Sigma$, requiring the length of $\necklace{S}$ to be at least $q$.
Similarly, in order to ensure that $\necklace{w}$ shares two subwords of length $1$ with $\necklace{S}$, $\necklace{S}$ must contain $2$ copies of every symbol on $\Sigma$, requiring the length of $\necklace{S}$ to be at least $2q$.
More generally for $\necklace{S}$ to share $i$ subwords of length $1$ with $\necklace{w}$, $\necklace{S}$ must contain $i$ copies of each symbol in $\Sigma$, requiring the length of $\necklace{S}$ to be at least $i \cdot q$.
Hence the maximum number of subwords of length $1$ that $\necklace{w}$ can share with $\necklace{S}$ is either $\floor{\frac{n \cdot k}{q}}$, if $\floor{\frac{n \cdot k}{q}} \leq n$, or $n$ otherwise.

In the case of subwords of length $2$, the problem becomes somewhat more complicated.
Note that in order to share a single word of length $2$, it is not necessary to to have every subword of length $2$ appear as a subword of $\necklace{w}$.
Instead, it is sufficient to use only the prefixes of the canonical representations of each necklace.
For example, given the binary alphabet $\{a,b\}$, every necklace has either $aa,ab$ or $bb$ as the prefix of length $2$.
Note that any necklace of length $2$ followed by the largest symbol $q$ in the alphabet $n - 2$ times belongs to the set $N_q^n$.
As such, a simple lower bound on the number of prefixes of the canonical representation of necklaces is the number of necklaces of length $2$, which in turn is bounded by $\frac{q^2}{2}$.
Noting that these prefixes in $\necklace{S}$ may overlap, in order to ensure that $\necklace{S}$ and $\necklace{w}$ share at least one subword of length $2$, the length of $\necklace{S}$ must be at least $\frac{q^2}{2}$.
Similarly, for $\necklace{S}$ and $\necklace{w}$ to share $i$ subwords of length $2$, the length of $\necklace{S}$ must be at least $\frac{i \cdot q^2}{2}$.
Hence the maximum number of subwords of length $2$ that $\necklace{S}$ and $\necklace{w}$ can share is either $\floor{\frac{2n \cdot k}{q^2}}$, if $\floor{\frac{2n \cdot k}{q^2}} \leq n$, or $n$ otherwise.
More generally, in order for $\necklace{S}$ to share at least one subword of length $j$ with $\necklace{w}$, the length of $\necklace{S}$ must be at least $\frac{q^j}{j}$.
Further the maximum number of subwords of length $j$ that $\necklace{S}$ and $\necklace{w}$ can share is either $\floor{\frac{j \cdot n \cdot k}{q^j}}$, if $\floor{\frac{j \cdot n \cdot k}{q^j}} \leq n$ or $n$ otherwise.

Using these observations, the maximum length of a common subword that $\necklace{w}$ can share with $\necklace{S}$ is the largest value $l$ such that $\frac{q^l}{l} \leq n \cdot k$.
By noting that $\frac{q^l}{l} \geq \frac{q^l}{n}$, a upper bound on $l$ can be derived by rewriting the  inequality  $\frac{q^l}{n} \leq n \cdot k$ as $l = 2\log_q(n\cdot k)$.
Note further that, for any value $l' > l$, there must be at least one necklace that does not share any subword of length $l'$ with $\necklace{S}$ as $\necklace{S}$ can not contain enough subwords to ensure that this is the case.
This bound allows an upper bound number of shared subwords between $\necklace{w}$ and $\necklace{S}$ to be given by the summation $\sum\limits_{i = 1}^{2\log_q(n \cdot k)}\min(\floor{\frac{i \cdot n \cdot k}{q^i}}, n) \leq n \cdot \log_q(n \cdot k) + \frac{\log_q(k \cdot n)}{q - 1} \approx \frac{q\cdot n\log_q(k \cdot n)}{q - 1} \approx n \log_q(k \cdot n)$.
Using this bound, the distance between $\necklace{w}$ and $\necklace{S}$ must be no less than $1 - \frac{\log_q(k \cdot n)}{n}$.

The same arguments can be applied to the multidimensional case.
Let $\vectorise{m} = ($ $m_1$, $m_2$, $\hdots$, $m_d)$ be a size vector of $d$-dimensions such that $M = m_1 \cdot m_2 \cdot \hdots \cdot m_d$.
The largest value of $M$ such that $\necklace{S}$ can contain every subword with $M$ positions is $2\log_q(n \cdot k)$.
The upper bound on the number of words of size $\vectorise{m}$ is $\frac{q^M}{M}$.
Let $F(x,\vectorise{m})$ return the size of the set $[\vectorise{m}]$, i.e. the number of vectors with $x$ positions that are less than or equal to $\vectorise{m}$ in each dimension.
Using this notation, the maximum number of shared subwords between $\necklace{w}$ and $\necklace{S}$ is $\sum\limits_{i = 1}^M F(i, \vectorise{m}) \cdot \frac{i \cdot N \cdot k}{q^{i}}$.
Note that $\sum\limits_{i = 1}^M F(i, \vectorise{m}) \cdot \frac{i \cdot N \cdot k}{q^{i}} \leq \sum\limits_{i = 1}^M \frac{i \cdot N \cdot k}{q^{i}}$.
Therefore, the upper bound on the number of common subwords in the multidimensional setting is $ N \log_q(k \cdot N)$, giving a bound on the distance of $1 - \frac{\log_q(k \cdot N)}{N}$.
\end{proof}

% \section{{\color{green} Two Approximation Algorithms for the \texorpdfstring{$k$}{k}-Centre Problem}}% for solving the \texorpdfstring{$k$}{k}-centre problem on necklaces.}
% \label{ch:sampling:sec:algorithms}

\noindent
Sections \ref{ch5:sec:dbt_apx} provides an approximation algorithms for the $k$-centre problem using Lemma \ref{lem:min_distance} as a lower bound.
% We now provide two approximation algorithms for the $k$-centre problem.
The first of these is 
%$1 + \frac{2N\log_q(kN) - \log^2_q(kN)}{2N^2 - 2N\log_q(kN)}$
$1+(\frac{\log_q{(k N)}}{N-\log_q{(kN)}}-\frac{\log^2_q(kN)}{2N(N-\log_q{(kN))}})$-approximate with a running time $O(N\cdot k)$, but it requires access to the de-Bruijn hypertori of the multidimensional necklaces; this is a generalisation of de-Bruijn sequences.
When $d=1$, there exists an efficient algorithm for computing the de-Bruijn sequence. However, for
$d>1$, no algorithm is known for computing a de-Bruijn hypertori.
Therefore, we develop a second algorithm that is 
%$1 + \frac{2N\log_q(kN) - \log^2_q(k)}{2N^2 - 2N\log_q(kN)}$
$1+(\frac{
\log_q{(k N)}}{N-\log_q{(kN)}}-\frac{\log^2_q(k)}{2N(N-\log_q{(kN))}})$-approximation with a running time $O(N^6)$, requiring techniques presented in Section \ref{sec:ranking}.

The main idea behind both algorithms is to try to find the largest size vector $\vectorise{\ell} = (l_1,l_2,\hdots,l_d)$ such that every subword of size $\vectorise{\ell}$ appears at least once in some word within the set.
In this setting $\vectorise{m}$ is larger than $\vectorise{\ell}$ if $m_1 \cdot m_2 \cdot \hdots \cdot m_d > l_1 \cdot l_2 \cdot \hdots \cdot l_d$.
This is motivated by observing that if two necklaces share a subword of length $l$, they must also share 2 subwords of length $l - 1$, 3 of length $l - 2$, and so on.
Lemma \ref{lem:max_distance} provides an upper bound for the overlap distance between any necklace in $|\mathcal{N}_q^{n}|$ and the set $\mathbf{S}$ containing all subwords of length $l$.

\begin{lemma}
\label{lem:max_distance}
Given $\necklace{w},\necklace{v} \in \mathcal{N}_q^{\vectorise{n}}$ sharing a common subword $\word{a}$ of size $\vectorise{m}$, let $x_i = n_i \cdot m_i$ if $n_i = m_i$, and $x_i= \frac{m_i(m_i + 1)}{2}$ otherwise.
The distance between $\word{w}$ and $\word{v}$ is bounded from above by
$\mathfrak{O}(w,v) \leq 1 - \frac{ \prod_{i=1}^d x_i}{N^2} \leq 1 - \frac{M^2}{N^2}$ where $N = n_1 \cdot n_2 \cdot \hdots \cdot n_d$ and $M = m_1 \cdot m_2 \cdot \hdots \cdot m_d$.
\end{lemma}

\begin{proof}
Note that the minimum intersection between $\necklace{w}$ and $\necklace{v}$ is the number of subwords of $\word{a}$, including the word $\word{a}$ itself.
To compute the number of subwords of $\word{a}$, consider the number of subwords starting at some position $\vectorise{j} \in [|\word{a}|]$.
Assuming that $|\word{a}|_i < n_i$ for every $i \in [d]$, the number of subwords starting at $\vectorise{j}$ corresponds to the size of the set $[\vectorise{j}, |\word{a}|]$, equal to $\prod\limits_{i = 1}^d m_i - |\word{a}|_i$.
This gives the number of shared subwords as being at least $\sum\limits_{\vectorise{j} \in [|\word{a}|]}\prod\limits_{i \in [d]} m_i - |\word{a}|_i \geq \sum\limits_{j \in [M]} j \geq \frac{M^2}{2}$.
Therefore, the distance between $\necklace{w}$ and $\necklace{v}$ is no more than $1 - \frac{M^2}{2N^2}$.
\end{proof}

\subsection{Approximating the \texorpdfstring{$k$}{k}-Centre Problem using de-Bruijn Sequences}
\label{ch5:sec:dbt_apx}

In this section we provide our first approximation algorithm that requires access to de-Bruijn sequences for the 1D case and to de-Bruijn hypertori for higher dimensions.
The main idea is to determine the largest de-Bruijn sequence that can fit into the set of $k$-centres.
As the de Bruijn sequence of order $l$ contains every word in $\Sigma^l$ as a subword, by representing the de Bruijn sequence of order $l$ in the set of centres we ensure that every necklace shares a subword of length $l$ with the set of $k$-centres.
Therefore, by determining the longest sequence that can be represented by $k$ centres, an upper bound on the distance between the furthest necklace and the set of $k$-centres is derived.

\begin{definition}
\label{def:de_bruijn_torus}
A \textbf{de Bruijn hypertorus} of order $\vectorise{n}$ is a cyclic $d$-dimensional word $\word{T}$ containing, as a subword, every word of size $\vectorise{n}$ over the alphabet $\Sigma$ of size $q$.
Further, each such word of size $\vectorise{n}$ over the alphabet $\Sigma$ appears exactly once as a subword of $T$.
\end{definition}

% \begin{figure}
%     \begin{tabular}{l|l}
%         Sequence: & 0000001000011000101000111001001011001101001111010101110110111111\\
%         Centre & Word \\
%         1 & {\color{red} 000000}1000011000{\color{blue} 10100}\\
%         2 & \hspace{3.1cm}{\color{blue} 10100}01110010010{\color{darkgreen} 11001}\\
%         3 & \hspace{6.2cm}{\color{darkgreen} 11001}10100111101{\color{purple}01011}\\
%         4 & {\color{red}000000}\hspace{8.1cm}{\color{purple}01011}10110111111
%     \end{tabular}
%     \caption{Example of how to split the de Bruijn sequence of order 6 between 4 centres.
%     Highlighted parts are the shared subwords between two centres.
%     }
%     \label{fig:deBrujinExample}
% \end{figure}

\begin{lemma}
\label{lem:dbs_alg}
There exists an $O(n \cdot k)$ time algorithm for the $k$-Centre problem on $\mathcal{N}_{q}^{n}$ such that every word in $\mathcal{N}_{q}^{n}$ shares a common subword of length at least $\log_q(n \cdot k)$ with one or more centres.
Further, no word in $\mathcal{N}_{q}^{n}$ is at a distance of more than $1 - \frac{\log_q^2(k \cdot n)}{2n^2}$ from the nearest centre.
\end{lemma}

\begin{proof}
The high level idea of this algorithm is to spilt a de Bruijn sequence of order $\lambda$ between the $k$ centres.
The motivation behind this approach is to represent every word of length $\lambda$ as a subword of at least one centre.
% This is constrained by having to ensure that all subwords of length $\lambda$ occur as a subword of a centre.
Note that the length of the de Bruijn sequence of order $\lambda$ is $q^{\lambda}$.

Given a de Bruijn sequence  $\word{s}$, naively splitting $\word{s}$ into $k$ words may lead to subwords being lost.
For example, take the de Bruijn sequence of order 4 over the alphabet $\{a,b\}$ $aaaabaabbababbbb$, dividing this between two words of length 8 results in the samples $aaaabaab$ and $bababbbb$, missing the words $aabb,abba,$ and $bbaa$.
In order to account for this, the sequence is split into centres of size $n - \lambda + 1$, with the final $\lambda - 1$ symbols of the $i^{th}$ centre being shared with the $(i + 1)^{th}$ centre.
In this manner, the first centre is generated by taking the first $n$ symbols of the de Bruijn sequence.
To ensure that every subword of length $\lambda$ occurs, the first $\lambda - 1$ symbols of the second centre is the same as the last $\lambda - 1$ symbol of the first centre.
Repeating this, the $i^{th}$ centre is the subword of length $n$ starting at position $i(n - \lambda + 1) + 1$ in the de Bruijn sequence.
An example of this is given in Figure 
\ref{fig:deBrujinExample}.

The leaves the problem of determining the largest value of $\lambda$ such that $q^{\lambda} \leq k \cdot (n - \lambda + 1)$.
The inequality $q^{\lambda} \leq k \cdot (n - \lambda + 1)$ can be rearranged in terms of $\lambda$ as $\lambda \leq \log_q(k \cdot (n + 1) - k \cdot \lambda)$.
Noting that $\lambda$ must be no more than $\log_q(k \cdot n)$, this upper bound on the value of $\lambda$ can be rewritten as $\log_q(k \cdot (n + 1 - \log_q(k \cdot n))) \approx \log_q(k \cdot n)$.
Using Lemma \ref{lem:max_distance}, along with $\log_q(k\cdot n)$ as an
approximate value of $\lambda$ gives an upper bound on the distance between between each necklace in $\mathcal{N}_q^{n}$ and the set of samples of $1 -  \frac{\log_q^2(k n)}{2n^2}$. 

As the corresponding de Bruijn sequence can be computed in no more than $O(k \cdot n)$ time \cite{Ruskey2000} and the set of samples can be further derived from the sequence in at most $O(k \cdot n)$ time, the total complexity is at most $O(k \cdot n)$.
Note that any algorithm that outputs such a set of centres takes at most $\Omega(k \cdot n)$ time.
\end{proof}

% \newtheorem*{thm_db1}{Theorem \ref{thm:de_bruijn_1}}

% \begin{thm_db1} 
% % \label{thm:de_bruijn_1}
% Problem~\ref{prob:k_sample} in 1D can be approximated in $O(n \cdot k)$ time with an approximation factor of $1 + f(n,k)$ where $f(n,k) = \frac{\log_q{(k \cdot n)}}{n-\log_q{(k \cdot n)}}-\frac{\log^2_q(k\cdot n)}{2n(n-\log_q{(k\cdot n))}}$ and $f(n,k) \rightarrow 0$ for $n \rightarrow \infty$.
% \end{thm_db1}

\begin{theorem} 
\label{thm:de_bruijn_1}
The $k$-centre problem for $\mathcal{N}_q^n$
% Problem~\ref{prob:k_sample} in 1D
can be approximated in $O(n \cdot k)$ time with an approximation factor of $1 + f(n,k)$ where $f(n,k) = \frac{\log_q{(k \cdot n)}}{n-\log_q{(k \cdot n)}}-\frac{\log^2_q(k\cdot n)}{2n(n-\log_q{(k\cdot n))}}$ and $f(n,k) \rightarrow 0$ for $n \rightarrow \infty$.
\end{theorem}

\begin{proof}
Recall from Lemma \ref{lem:min_distance} that the overlap distance is bounded by $1 -\frac{\log_q(k \cdot n)}{n}$.
Using the lower bound of $1 - \frac{\log_q^2(k n)}{2n^2}$ given by Lemma \ref{lem:dbs_alg} gives an approximation ratio of $\frac{1 - \frac{\log_q^2(k n)}{2n^2}}{1 -\frac{\log_q(k \cdot n)}{n}}$
= $\frac{2n^2-\log^2_q(kn)}{2n^2-2n\log_q{(kn)}}$
= 
$1+\frac{
2n\log_q{(kn)}-\log^2_q(kn)}{2n^2-2n\log_q{(kn)}}$=
$1+ \frac{
\log_q{(k n)}}{n-\log_q{(kn)}}-\frac{\log^2_q(kn)}{2n(n-\log_q{(kn))}}$.
Note that $f(n,k)=\frac{
2n\log_q{(kn)}-\log^2_q(kn)}{2n^2-2n\log_q{(kn)}}$ converges to $0$ when
$n \rightarrow \infty$ for a constant $k<q^n/n$.
\end{proof}

\begin{theorem}
\label{thm:de_bruijn_d}
Let $T$ be a $d$-dimensional de Bruijn hyper torus of size $(x,x,\hdots,x)$.
There exist $k$ subwords of $T$ that form a solution to the $k$-centre problem for $\mathcal{N}_{q}^{(y,y,\hdots,y)}$ with an approximation factor of $1 + f(n,k)$ where $f(n,k) = \frac{\log_q{(k N)}}{N-\log_q{(k \cdot N)}}-\frac{\log^2_q(k \cdot N)}{2N(N-\log_q{(k \cdot N))}}$, $f(n,k) \rightarrow 0$, $N \rightarrow \infty$.
\end{theorem}

\begin{proof}
Recall from Lemma \ref{lem:min_distance} that the lower bound on the distance between the centre and every necklace in $\mathcal{N}_q^{\vectorise{n}}$ is $ 1- \frac{log_{q} (k \cdot N)}{N}$.
As in Theorem \ref{thm:de_bruijn_1}, the goal is to find the largest de Bruijn torus that can ``fit'' into the centres.
To simplify the reasoning, the de Bruijn hyper tori here is limited to those corresponding to the word where the length of each dimension is the same.
Formally, the de Bruijn hypertori are restricted to be of the size $m_1 = m_2 = \hdots = m_j = \sqrt[j]{N}$ for some $j \in [d]$, giving the total number of positions in the tori as $M$.
Similarly, the centres is assumed to have size $n_1 = n_2 = \hdots = n_d = \sqrt[d]{N}$, giving $N$ total positions.

Observe that the largest torus that can be represented in the set of centres has $M$ positions such that $q^{M} \leq k \cdot N^{(d - j)/ d}( \sqrt[d]{N} - \sqrt[j]{M} + 1)^j$.
This can be rewritten to give $M \leq \log_q(k \cdot N^{(d - j)/d}( \sqrt[d]{N} - \sqrt[j]{M} + 1)^j)$.
Noting that $M$ is of logarithmic size relative to $N$, this is approximately equal to $M \leq \log_q(k \cdot N)$.
Using Lemma \ref{lem:max_distance}, the minimum distance between any necklace in $\mathcal{N}_q^{\vectorise{n}}$ is $1 - \frac{\log^2_q(kN)}{2N^2}$.
This is compared to the optimal solution, following the arguments from Theorem \ref{thm:de_bruijn_1} giving a ratio of $1 + f(N,k)$ where 
$f(N,k) = \frac{2\cdot N\log_q{(k\cdot N)}-\log^2_q(k\cdot N)}{2\cdot N^2-2\cdot N \cdot \log_q{(k\cdot N)}}$=
$\frac{
\log_q{(k N)}}{N-\log_q{(kN)}}-\frac{\log^2_q(kN)}{2N(N-\log_q{(kN))}}$.
\end{proof}

\noindent
For both cases table providing some explicit examples of the approximation ratio for different values of $n$ and $k$ is given in Table \ref{tab:approximation_ratios}.
While this provides a good starting point for solving the $k$-Centre problem for $\mathcal{N}^{\vectorise{n}}_q$, results on generating de Bruijn tori are highly limited, focusing on the cases with small dimensions \cite{Chung1992, Horan2016, Hurlbert1993, Hurlbert1995,Hurlbert1996}.
% While this provides a good starting point for solving the $k$-Centre problem for $\mathcal{N}^{\vectorise{n}}_q$, the results on generating de Bruijn tori are highly limited, focusing on the cases with small dimensions \cite{Chung1992, Horan2016, Hurlbert1993, Hurlbert1995,Hurlbert1996}.
As such an alternate approach is needed.

\begin{table}
    %\centering
\begin{tabular}{l | l  l  l  l  l  l  l  l }
	$k$\textbackslash $n$  & 1  & 2  & 3  & 4  & 5  & 6  & 7  & 8 \\
	\hline
	1  & $1.0$  & $1.75$  & $1.8242$  & $1.75$  & $1.6657$  & $1.59388$  & $1.53532$  & $1.4875$ \\
	2  & 1.0  & 1.0  & $4.54496$  & $2.875$  & $2.322$  & $2.04096$  & $1.86822$  & $1.75$ \\
	3  & 1.0  & 1.0  & 1.0  & $5.76696$  & $3.17774$  & $2.48677$  & $2.15592$  & $1.95785$ \\
	4  & 1.0  & 1.0  & 1.0  & 1.0  & $4.61912$  & $3.00217$  & $2.43963$  & $2.14583$ \\
	5  & 1.0  & 1.0  & 1.0  & 1.0  & $7.98402$  & $3.65337$  & $2.73732$  & $2.32623$ \\
	6  & 1.0  & 1.0  & 1.0  & 1.0  & $27.84082$  & $4.54496$  & $3.06221$  & $2.50535$ \\
	7  & 1.0  & 1.0  & 1.0  & 1.0  & 1.0  & $5.88615$  & $3.4276$  & $2.68724$ \\
	8  & 1.0  & 1.0  & 1.0  & 1.0  & 1.0  & $8.19368$  & $3.84946$  & $2.875$
\end{tabular}

\begin{tabular}{l |  l  l  l  l  l  l  l  l }
	$k$\textbackslash$n$  & 1  & 2  & 3  & 4  & 5  & 6  & 7  & 8 \\
	\hline
	1  & $1.0$  & $1.18333$  & $1.19493$  & $1.18333$  & $1.16897$  & $1.15565$  & $1.144$  & $1.13393$ \\
	2  & $1.41667$  & $1.41667$  & $1.34509$  & $1.29167$  & $1.25296$  & $1.22393$  & $1.20138$  & $1.18333$ \\
	3  & $1.8242$  & $1.59388$  & $1.44797$  & $1.36238$  & $1.30633$  & $1.26659$  & $1.23682$  & $1.2136$ \\
	4  & $2.33333$  & $1.75$  & $1.53018$  & $1.41667$  & $1.34644$  & $1.29825$  & $1.2629$  & $1.23575$ \\
	5  & $3.09914$  & $1.89704$  & $1.6006$  & $1.46153$  & $1.379$  & $1.32369$  & $1.28372$  & $1.25334$ \\
	6  & $4.54496$  & $2.04096$  & $1.66333$  & $1.50021$  & $1.40664$  & $1.34509$  & $1.30113$  & $1.26799$ \\
	7  & $8.75423$  & $2.18549$  & $1.72065$  & $1.53449$  & $1.4308$  & $1.36364$  & $1.31615$  & $1.28059$ \\
	8  & 1.0  & $2.33333$  & $1.77396$  & $1.56548$  & $1.45235$  & $1.38007$  & $1.32939$  & $1.29167$ 
\end{tabular}
    \caption{Table of approximation ratio for the algorithm given in Theorem \ref{thm:de_bruijn_1} for different values of $n$ and $k$ for a binary alphabet (top) and an alphabet of size 8 (below).
    Note that when $k \geq q^n$ the approximation ratio is $1$ as every necklace can be represented in the set.}
    \label{tab:approximation_ratios}
\end{table}

% \duncan{New Theorem/Algorithm}

Theorem \ref{thm:alg_3} presents such an alternate approach.
At a high level, the idea is to reduce the problem from the multidimensional setting to the 1D problem which we can already solve.
Given a size vector $\vectorise{n}$, integer $k$ and alphabet $\Sigma$, this approach can be thought of as finding a set of $k \cdot n_1 \cdot \hdots \cdot n_{d - 1}$ samples of length $n_d$ over $\Sigma$, taking advantage of the added number of samples to increase the lower bound on the length of shared subwords.
There are two cases to consider based on the values of $\vectorise{n}$.

\noindent
\textbf{Case 1, $q^{n_d} \leq k \cdot \frac{N}{n_d}$:}
In this case the set of samples is constructed by using $k' = \frac{k \cdot N}{n_d}$ samples of $\mathcal{N}_q^{n_d}$.
The motivation behind this approach is to optimise the length of the 1D subwords that are shared by the sample and every other necklace in $\mathcal{N}_q^{\vectorise{n}}$.
Let $\mathbf{S} \subseteq \mathcal{N}_q^{n_d}$ be a set of samples $k \cdot \frac{N}{n_d}$ from $\mathcal{N}_q^{n_d}$ constructed following the algorithm outline in Lemma \ref{lem:dbs_alg}.
Following the arguments from Lemma \ref{lem:dbs_alg}, every necklace in $\mathcal{N}_q^{n_d}$ must share a subword of length $\log_q(k \cdot N)$ with at least one sample in $\mathbf{S}$.
As every subword of size $(1,1,\hdots,1,n_d)$ of any necklace in $\mathcal{N}_q^{\vectorise{q}}$ belongs to a necklace $\necklace{w} \in \mathcal{N}_q^{n_d}$, by ensuring that every necklace in $\mathbf{S}$ appears as a subword in the sample $\mathbf{S}'\subseteq \mathcal{N}_q^{\vectorise{n}}$ it is ensured that $\necklace{w}$ shares at least one subword of length $\log_q(k \cdot N)$ with some necklace in $\mathbf{S}'$.
This can be done by simply splitting $\mathbf{S}$ into $k$ sets of $\frac{N}{n_d}$ samples, each of which can be made into a word of size $\vectorise{n}$ through concatenation.
From Lemma \ref{lem:max_distance}, the maximum distance between any necklace in $\mathbf{S}'$ and necklace in $\mathcal{N}_q^{\vectorise{n}}$ is $1 - \frac{\log_q^2(k \cdot N)}{2N^2}$.
Note that this is equal to the bound given by Lemma \ref{lem:dbs_alg}, resulting in the same approximation ratio.

\noindent
\textbf{Case 2, $q^{n_d} > k \cdot \frac{N}{n_d}$:}
In this case, following the process outlined above, it is possible to represent every word of length $n_d$ over $\Sigma$ with some redundancy.
In order to make better use of the samples, and reduce the redundancy, an alternative reduction from the 1D setting is constructed.
The high level idea is to construct a new alphabet such that each symbol corresponds to some word in $\Sigma^{\vectorise{m}}$ for some size vector $\vectorise{m}$.

The first problem becomes determining the size vector such that this reduction can be done.
Let $\Sigma(\vectorise{m})$ denote the alphabet of size $q^{m_1 \cdot m_2 \cdot \hdots \cdot m_d}$ such that each symbol in $\Sigma(\vectorise{m})$ corresponds to some word in $\Sigma^{\vectorise{m}}$.
Given a word $\word{w} \in \Sigma(\vectorise{m})^{n_1/m_1, n_2/m_2,\hdots, n_d/m_d}$ a word $\word{v} \in \Sigma^{\vectorise{n}}$ can be constructed by replacing each symbol in $\word{w}$ with the corresponding word in $\Sigma^{\vectorise{m}}$.
Note that the largest value of $\vectorise{m}$ such that every symbol in $\Sigma(\vectorise{m})$ can be represented in $k$ words from $\Sigma(\vectorise{m})^{n_1/m_1, n_2/m_2,\hdots, n_d/m_d}$ is bounded by the inequality $q^{m_1 \cdot m_2 \cdot \hdots \cdot m_d} \leq k \cdot \floor{\frac{n_1}{m_1}} \cdot \floor{\frac{n_2}{m_2}} \cdot \hdots \cdot \floor{\frac{n_d}{m_d}}$.
Letting $M = m_1 \cdot m_2 \cdot \hdots \cdot m_d$, this inequality can be rewritten as approximately $q^M \leq k \cdot \frac{N}{M}$.
Treating $M$ as being approximately $N$ for the purpose of giving an upper bound to $M$ gives $M \leq \log_q(k)$.

Using this bound on $M$ let $\vectorise{m}$ be some set of vectors such that $M = m_1 \cdot m_2 \cdot \hdots \cdot m_d$.
We may assume without loss of generality that $m_d = 1$.
The samples for $\mathcal{N}_q^{\vectorise{n}}$ are constructed by making a set $\mathbf{S}$ of $k \frac{N}{M \cdot n_d}$ samples for $\mathcal{N}^{n_d}_{q^M}$.
Following the arguments from Lemma \ref{lem:dbs_alg}, every necklace in $\mathcal{N}^{n_d}_{q^M}$ must share a subword of length at least $\log_{q^M}(k \cdot \frac{N}{M}) = \frac{\log_q(k \cdot \frac{N}{M})}{M} = \frac{\log_q(k \cdot \frac{N}{\log_q(k)})}{\log_q(k)}$.
Note further that, as each symbol in $\Sigma(\vectorise{m})$ corresponds to a word in $\Sigma^{\vectorise{m}}$, converting each word in $\mathbf{S}$ to a word of size $(m_1,m_2, \hdots, m_{d - 1}, n_1)$ provides a sample such that every necklace in $\mathcal{N}_q^{(m_1,m_2, \hdots, m_{d - 1}, n_1)}$ shares a subword of size $(m_1,m_2, \hdots, m_{d - 1}, \frac{\log_q(k \cdot \frac{N}{\log_q(k)})}{\log_q(k)})$ with some member of the sample.
Converting this new sample into a sample $\mathbf{S}' \subseteq \mathcal{N}_q^{\vectorise{n}}$ maintains the same size of shared subwords.
From Lemma \ref{lem:max_distance}, the furthest distance between $\mathbf{S}'$ and any necklace in $\mathcal{N}_q^{\vectorise{n}}$ is bounded from above by $1 - \frac{\log^2_q(k) \cdot  \frac{\log_q^2(k \cdot \frac{N}{\log_q(k)})}{\log_q^2(k)}}{2N^2} = 1 - \frac{\log_q^2(k \cdot \frac{N}{\log_q(k)}}{2N^2} \approx 1 - \frac{\log^2_q(k \cdot N)}{2N^2}$.

% \newtheorem*{thm_alg_3}{Theorem \ref{thm:alg_3}}

% \begin{thm_alg_3}
% % \label{thm:alg_3}
% Problem \ref{prob:k_sample} can be approximated in $O(N^2 k)$ time within an approximation factor of $1 +  \frac{
% \log_q{(k N)}}{N-\log_q{(kN)}}-\frac{\log^2_q(kN)}{2N(N-\log_q{(kN))}}$.
% \end{thm_alg_3}

\begin{theorem}
\label{thm:alg_3}
The $k$-centre problem for $\mathcal{N}_q^{\vectorise{n}}$
% Problem \ref{prob:k_sample}
can be approximated in $O(N^2 k)$ time within an approximation factor of $1 +  \frac{
\log_q{(k N)}}{N-\log_q{(kN)}}-\frac{\log^2_q(kN)}{2N(N-\log_q{(kN))}}$, where $N = \prod_{i = 1}^d n_i$.
\end{theorem}

\begin{proof}
Following the above construction, note that in both cases the algorithm bounds the upper distance between samples by approximately $1 - \frac{\log^2_q(k \cdot N)}{2N^2}$.
Following the same arguments as in Theorem \ref{thm:de_bruijn_1} gives the approximation ratio of $1 +  \frac{\log_q{(k N)}}{N-\log_q{(kN)}}-\frac{\log^2_q(kN)}{2N(N-\log_q{(kN))}}$.
Regarding time complexity, in the first case the problem can be solved in $O(k \cdot N)$ time using Theorem \ref{thm:de_bruijn_1}.
In the second case, a brute force approach to find to best value of $\vectorise{m}$ can be done in $O(N)$ additional time steps giving a total complexity of $O(k \cdot N^2)$.
\end{proof}

\bibliography{stacs_bib.bib}
\bibliographystyle{plainurl}

\newpage

\end{document}